\documentclass[11pt,reqno]{amsart}

\usepackage[utf8]{inputenc}
\usepackage{amsfonts,latexsym,amssymb,amsmath,amsthm,slashed,cancel}
\usepackage{amsrefs}
\usepackage{todonotes}
\usepackage{hyperref}
\usepackage{a4wide}
\usepackage{enumerate}
\usepackage{mathrsfs}
\usepackage{nicefrac}
\usepackage{MnSymbol}
\usepackage{ourbib}
\usepackage[all]{xy}

\hypersetup{
    colorlinks,
    linkcolor={red!50!black},
    citecolor={blue!50!black},
    urlcolor={blue!80!black}
}

\def\D{\slashed\nabla}

\newcommand{\myicon}{$\triangleright$}
\renewcommand{\SS}{\mathscr{S}}
\newcommand{\LL}{\mathscr{L}}
\newcommand{\reg}{\mathrm{reg}}

\newcommand{\nS}{\slashed n_\Sigma}
\newcommand{\nSp}{\slashed n_{\Sigma_+}}
\newcommand{\nSm}{\slashed n_{\Sigma_-}}
\newcommand{\Feyn}{\mathrm{Feyn}}

\newcommand{\dV}{\,\mathrm{dV}}

\DeclareMathOperator{\Tr}{Tr}
\DeclareMathOperator{\tr}{tr}

\newcommand{\vol}{\mathrm{vol}}

\newcommand{\N}{\mathbb{N}}
\newcommand{\Z}{\mathbb{Z}}
\newcommand{\R}{\mathbb{R}}
\newcommand{\C}{\mathbb{C}}
\newcommand{\CC}{\mathscr{C}}
\newcommand{\J}{\mathsf{J}}

\newcommand{\id}{\operatorname{id}}

\newcommand{\loc}{\mathrm{loc}}
\newcommand{\APS}{\mathrm{APS}}

\DeclareMathOperator{\supp}{supp}
\DeclareMathOperator{\singsupp}{sing--supp}

\newcommand{\res}{\mathrm{res}}

\renewcommand{\Im}{\mathrm{Im}}
\renewcommand{\Re}{\mathrm{Re}}
\newcommand{\eps}{\varepsilon}
\newcommand{\grad}{\mathrm{grad}}
\newcommand{\grads}{\cancel{\grad}}
\newcommand{\grade}{\grad_{(1)}}
\newcommand{\grades}{\cancel{\grad}_{(1)}}

\newcommand{\gradzs}{\cancel{\grad}_{(2)}}
\newcommand{\Boxz}{\Box_{(1)}}
\renewcommand{\div}{\mathrm{div}}
\newcommand{\divz}{\div_{(1)}}
\renewcommand{\Box}{\largesquare}
\newcommand{\WF}{\mathsf{WF}}

\newcommand{\FpmU}{\mathscr{F}^{\pm}_{\UU}}

\newcommand{\DS}{\D_{\!\Sigma}}
\newcommand{\Sol}{\mathrm{Sol}}

\newcommand{\dA}{\mathrm{dA}}
\newcommand{\Adach}{{\hat{\mathrm{A}}}}
\newcommand{\ch}{\mathrm{ch}}
\newcommand{\nE}{\nabla^E}

\newcommand{\rmi}{\mathrm{i}\,}
\newcommand{\e}{\mathrm{e}}

\newcommand{\UU}{\mathscr{U}}

\newcommand{\DD}{\mathscr{D}}
\newcommand{\<}{\langle}
\renewcommand{\>}{\rangle}
\newcommand{\ml}{\llangle}
\newcommand{\mr}{\rrangle}

\newcommand{\ret}{\mathrm{ret}}
\newcommand{\adv}{\mathrm{adv}}
\newcommand{\diag}{\mathrm{diag}}
\newcommand{\oo}{\mathfrak{o}}
\newcommand{\LOG}{\mathrm{LOG}}
\newcommand{\dirac}{\mathscr{D}}
\newcommand{\rg}{\mathrm{rg}}
\renewcommand{\O}{\mathrm{O}}

\renewcommand{\Xi}{\pmb{\xi}}

\DeclareMathOperator{\OO}{\mathsf{O}}

 \newtheorem{theorem}{Theorem}[section]
 \newtheorem*{theorem*}{Theorem}
 \newtheorem{lemma}[theorem]{Lemma}
 \newtheorem{corollary}[theorem]{Corollary}
 \newtheorem{prop}[theorem]{Proposition}
 \newtheorem*{prop*}{Proposition}

 \theoremstyle{definition}
 \newtheorem{definition}[theorem]{Definition}
 \newtheorem{example}[theorem]{Example}
 \newtheorem{remark}[theorem]{Remark}

 \newtheorem{notation}[theorem]{Notation}

\allowdisplaybreaks
\title[Local Index Theory for Lorentzian Manifolds]{Local Index Theory for Lorentzian Manifolds}

\author[C.~B\"ar]{Christian B\"ar}
\address{Institut f\"ur Mathematik, Universit\"at Potsdam, 14476 Potsdam, Germany}
\email{cbaer@uni-potsdam.de}

\author[A.~Strohmaier]{Alexander Strohmaier}
\address{School of Mathematics \\
University of Leeds\\
Leeds, LS2 9JT, UK}
\email{a.strohmaier@leeds.ac.uk}
\address{Leibniz University Hannover, Institute of Analysis, 30167 Hannover, Germany}  \email{a.strohmaier@math.uni-hannover.de} 

\subjclass[2010]{primary: 58J20, 58J45; secondary: 35L02, 35L05, 58J32}
\keywords{Dirac-type operator, globally hyperbolic Lorentzian manifold, Atiyah-Patodi-Singer boundary conditions, Feynman propagator, local index theorem, Hadamard expansion}
\thanks{}

\date{\today}

\begin{document}

\begin{abstract}
Index theory for Lorentzian Dirac operators is a young subject with significant differences to elliptic index theory. 
It is based on microlocal analysis instead of standard elliptic theory and one of the main features is that a nontrivial index is caused by topologically nontrivial dynamics rather than nontrivial topology of the base manifold.
In this paper we establish a local index formula for Lorentzian Dirac-type operators on globally hyperbolic spacetimes. 
This local formula implies an index theorem for general Dirac-type operators on spatially compact spacetimes with Atiyah-Patodi-Singer boundary conditions on Cauchy hypersurfaces.
This is significantly more general than the previously known theorems that require the compatibility of the connection with Clifford multiplication and the spatial Dirac operator on the Cauchy hypersurface to be selfadjoint with respect to a positive definite inner product.
\bigskip
\bigskip

\noindent
\textsc{R\'esum\'e.}
\textbf{Théorie locale de l'indice  pour les variétés lorentziennes.}
La théorie de l'indice pour les opérateurs de Dirac lorentziens est un sujet récent qui présente des différences importantes avec la théorie d'indice elliptique.
Elle est basée sur l'analyse microlocale au lieu de la théorie elliptique standard et l'une de ses principales caractéristiques est qu'un indice non trivial est causé par une dynamique topologiquement non triviale plutôt que par la topologie non triviale de la variété de base.
Dans cet article, nous obtenons une formule locale d'indice pour les opérateurs lorentziens de type Dirac sur les espaces-temps globalement hyperboliques.
Cette formule locale implique un théorème d'indice pour les opérateurs généraux de type Dirac sur des espaces-temps spatialement compacts avec des conditions limites d'Atiyah-Patodi-Singer sur les hypersurfaces de Cauchy.
Il s'agit d'un théorème beaucoup plus général que les théorèmes connus précédemment, qui exigent la compatibilité de la connexion avec la multiplication de Clifford et que l'opérateur de Dirac spatial sur l'hypersurface de Cauchy soit auto-adjoint par rapport à un produit scalaire défini positif.
\end{abstract}

\maketitle


\section*{Introduction}

The Atiyah-Singer index theorem \cite{MR236950} is one of the most important achievements of mathematics in the 20th century.
It provides many conceptual insights by identifying topological invariants as indices of geometrically defined differential operators, thereby linking geometry, topology, and analysis. 
The index theorem is most conveniently stated for a twisted Dirac operator $\D$ on a closed Riemannian manifold, and it then gives a formula for the index as an integral over a density that is determined by local geometric quantities.
In fact, there is a local version of the index theorem that expresses the difference of the local traces of two heat kernels associated to the Laplace-type operators $\D \D^{*}$ and $\D^{*} \D$ in the limit as time goes to zero. 
This local index formula implies the Atiyah-Singer index theorem on a closed Riemannian manifold by the McKean-Singer formula.
We refer to the monograph \cite{BGV} for details and further references.
\smallskip

It is the local index theorem that allows the more general treatment of operators in the $L^2$-settings, in relative settings, and on manifolds with boundary.
As an example we mention Atiyah's theorem on the equality of the index and the $L^2$-index on the universal cover \cite{MR0420729}, the proof of which relies heavily on the local index theorem.
For manifolds with boundary an index formula was given by Atiyah, Patodi, and Singer, see \cites{MR397797, MR397798, MR397799} or also \cite{MR1348401}.
This formula contains a boundary contribution which gave rise to the study of the $\eta$-invariant and all its applications.
Local index theory also established the link between Quillen's theory of superconnections with the index theorem for families \cite{MR813584}.
There are further generalizations of the index theorem dealing with elliptic operators on manifolds with singularities, noncompact manifolds, or hypoelliptic operators, see for example \cites{MR730920,MR720933,MR1876286,MR1156670,MR2680395} to mention only a few.
\smallskip

Developing index theory for Lorentzian manifolds seems hopeless at first, since Dirac-type operators are hyperbolic in this case.
On a closed manifold an operator needs to be elliptic to be Fredholm.
So there is no Lorentzian analog to the Atiyah-Singer index theorem.
Surprisingly, the situation changes in the presence of boundary.
In \cite{Baer:2015aa} we proved a Lorentzian index theorem for Dirac operators on globally hyperbolic spacetimes with appropriate boundary conditions imposed in the timelike future and past.
This is an analog to the Atiyah-Patodi-Singer index theorem in the Riemannian setting.
In fact, the boundary conditions and the geometric formula for the index are exactly the same as in the Riemannian setting.
\smallskip

This allows for the direct application of index theory in situations relevant in relativistic physics, when separation of variables and reduction to an elliptic system or Wick rotation is not possible.
For example, the chiral anomaly in quantum field theory can be understood as such an index \cite{baerstroh2015chiral}.
\smallskip

The Lorentzian index theorem was extended to more general boundary conditions by the first author and Hannes in \cite{BaerHannes}, to the noncompact Callias setting by Braverman in \cite{MR4054812}, and to a noncompact $\Gamma$-equivariant setting by Damaschke in \cite{damaschke}. 
Parts of the theorem have also been generalized to a more abstract framework by van den Dungen and  Ronge in \cites{ronge,dungen}.
In \cite{shenwrochna}, Shen and Wrochna replace the timelike compact dynamics by asymptotic conditions for infinite times.
\smallskip

The fact that this hyperbolic operator is Fredholm is in all these cases not based on elliptic theory, but rather relies on H\"ormander's propagation of singularities theorem. 
This is reminiscent of the use of propagation estimates by H\"ormander showing finite dimensionality of the space of solutions of operators of real principal type, see \cite{MR0388464}. 
We mention that refined propagation of singularities is also a key behind the Fredholm theory introduced by Vasy \cite{MR3117526}, which has found many applications.
\smallskip

The present paper extends the Lorentzian index theorem in two directions:
\smallskip

{(1)}
We drop the assumption of a positive definite inner product on the vector bundle along the spacelike hypersurfaces with respect to which the induced spatial Dirac operator is formally selfadjoint.
This requirement is indeed not natural from a Lorentzian point of view and it is not satisfied by many geometric operators.
Dropping this assumption, one can no longer reduce to the Riemannian APS-index theorem using spectral flow arguments as in \cite{Baer:2015aa}.
The deformation argument that was used there is manifestly nonlocal and relies heavily on the invariance properties of characteristic classes.
A local proof cannot be carried out in that way.
Instead, we will give a purely Lorentzian proof and the Riemannian index theorem by Atiyah, Patodi, and Singer is no longer used.
This also allows us to treat the significantly larger class of general Dirac-type operators instead of certain twisted spinorial Dirac operators only.
\smallskip

(2)
We provide a local interpretation of the index in the Lorentzian context, and we show that a local index formula holds for this quantity, also in the context of spatially noncompact spacetimes.
The role played by the heat kernel in the Riemannian setting is now taken over by Feynman parametrices.
The short-time asymptotics of the heat kernel is replaced by the Hadamard expansion.
We show that the index of the Lorentzian Dirac operator is the integral of a locally defined current over the Cauchy hypersurface. 
This can be seen as the Lorentzian analog of the McKean-Singer formula. 
The $\eta$-invariant appears naturally in this context. 
We note that the treatment of analogous currents using the Hadamard expansion has appeared in the physics literature on quantum field theory on curved spacetimes (see the work by Zahn \cite{zahn2015locally}) and our approach also generalizes these to Dirac-type operators.
\smallskip

Due to the fundamental role played by Feynman parametrices we study them in quite some detail. 
We give a detailed construction that is similar to that of the classical Hadamard parametrix but is based on a family of distributions with distinguished microlocal properties.
This construction results in a similar expansion as the one for Hadamard states in the physics literature (see \cite{MR1133130}) but is simpler and gives a more conceptual approach to the corresponding expansion for Feynman parametrices.  
Indeed, it has been known for a while that in even dimensions one has to add logarithmic terms to such an expansion (see for example \cite{Lewandowski} for a direct construction of the Feynman propagator or \cite{Zelditch} for a nice survey in the ultrastatic case). 
In our treatment, the logarithmic terms  appear naturally as derivatives, and the recursion formulae are obtained by simply differentiating the Hadamard recursion formulae.
More concretely, a special case of Proposition~\ref{prop:ZerlegeG} says

\begin{prop*}
Let $P$ be a normally hyperbolic operator acting on sections of a vector bundle $\SS$ over a $2m$-dimensional globally hyperbolic manifold $X$.

Then every Feynman parametrix $G$ of $P$ is of the form 
$$
G = G^\loc + G^\reg
$$%
where $G^\reg$ is a distribution on $X\times X$ which is $C^2$ on a neighborhood of the diagonal and near the diagonal we have
\begin{equation*}
G^\loc 
=
\sum\limits_{k=0}^{m+1} V_{k} \cdot G^{+,\UU}_{1-m+k}.
\end{equation*}
\end{prop*}

Here the $G^{+,\UU}_{\ell}$ are explicitly given scalar distributions (including logarithmic terms) which encode the singularity structure of $G$ while the Hadamard coefficients $V_{k}$ are smooth and can be computed locally from the coefficients of $P$ and the underlying geometry. 
Hadamard expansions involving distributions with distinguished microlocal properties have also been used by Dang and Wrochna in \cite{DangWrochna1} and \cite{DangWrochna3} in the context of complex powers and residues of normally hyperbolic operators. 
These expansions involve slightly different families of distributions that cover another asymptotic parameter-range.
In particular, in even spacetime dimensions the distinguished properties of the family $G^{+,\UU}_{\beta}$ are crucial for the compution of the index density.
\smallskip

Both, the local part $G^\loc$ and the regular part $G^\reg$ of $G$ contribute to the index formula.
Combining Theorems~\ref{thm:indexJ} and \ref{thm:deltaJHadamard}, we obtain

\begin{theorem*}
Let $X$ be a $2m$-dimensional compact globally hyperbolic manifold with boundary $\partial X=\Sigma_- \sqcup \Sigma_+$. 
Let $\dirac=
   \begin{pmatrix}
   0 & \dirac_R \\
   \dirac_L & 0
   \end{pmatrix}$ be an odd Dirac-type operator over $X$.
Let $V_{R,k}$ be the Hadamard coefficients of $\dirac_L\dirac_R$ and $V_{L,k}$ those of $\dirac_R\dirac_L$.
Suppose that $X$ and $\dirac$ have product structure near $\Sigma_\pm$.

Then under $\APS$-boundary conditions, the Dirac operator
$$
\dirac_L:C^\infty_{\APS}( X;\SS_L)\to C^\infty( X;\SS_R)
$$
as a continuous linear map between Fr\'echet spaces is Fredholm and its index is given by
$$
\mathrm{ind}(\dirac_L)
=
\Xi(\D_{\Sigma_+}) - \Xi(\D_{\Sigma_-}) + \int_{X} \frac{\tr(V_{R,m})  -\tr(V_{L,m})}{(4 \pi)^m\cdot m!}  \dV.
$$
\end{theorem*}
   
If the Riemannian Dirac-type operators $\D_{\Sigma_\pm}$ on the boundary of $X$ are selfadjoint, then $\Xi(\D_{\Sigma_\pm})$ coincides with the $\xi$-invariant (which is essentially the $\eta$-invariant) by Atiyah, Patodi, and Singer.
We provide a definition also for the nonselfadjoint case and express it in Proposition~\ref{prop:etah} in terms of the regular part of the Feynman propagator of $\dirac_L\dirac_R$.
This is where the assumption that $X$ and $\dirac$ have product structure near the boundary enters.
Feynman propagators of a normally hyperbolic operator are not unique.
They all have the same local part and hence the same Hadamard coefficients, but the regular parts can be different.
For the $\xi$-invariant we need to use a specific choice of Feynman propagator canonically constructed on a product manifold, see Theorem~\ref{thm:FeynmanProductWave}.
Even in the selfadjoint case, our definition of the $\xi$-invariant is better suited for local computations as it avoids the heat kernel and associated technical difficulties with its properties on noncompact manifolds.
In fact, there is a pointwise version of the relation between the $\xi$-invariant and the regular part of the Feynman propagator provided we assume that the Dirac-type operator is a (generalized) Dirac operator in the sense of Gromov and Lawson.
\smallskip

The paper is organized as follows. 
After introducing the basic setup in Section~\ref{setup} we describe the functional calculus of Laplace-type and Dirac-type operators on a Riemannian manifold $\Sigma$ in Section~\ref{sec:LaplaceDirac}.
Here we either assume that the operators are selfadjoint with respect to a positive definite bundle metric or that $\Sigma$ is compact. 
Since we do not assume selfadjointness in the compact case, the spectrum will not necessarily be contained in the real line and nontrivial Jordan-blocks may be present.
Section~\ref{GanzFein:Section} reviews the properties of Feynman parametrices as introduced by Duistermaat and H\"ormander and defines a Feynman propagator associated to a Laplace-type operator.  
In case the operator is selfadjoint, this construction is analogous to the construction based on frequency splitting in the physics literature, and also the one given in \cite{Baer:2015aa}. 
In Section~\ref{Dirac-Fein} we construct the Feynman propagator for the Dirac operator and give a relation to the $\xi$-invariant in the case $\Sigma$ is compact and the Dirac operator is selfadjoint. 
This motivates the definition of a more general $\xi$-invariant that is associated with any Feynman propagator. 
Section~\ref{index:Section} contains the statement and the proof of the local index theorem. 
In the case of compact Cauchy hypersurface, we find the index theorem as stated above as a consequence.
Finally, in the appendix we give a new construction of Feynman parametrices based on a family of distributions which is similar to the family of Riesz distributions but has the singularity structure of the Feynman propagator.
\smallskip

\emph{Acknowledgments.}
This work was financially supported by \href{https://www.spp2026.de/projects/2nd-period/boundary-value-problems-and-index-theory-on-riemannian-and-lorentzian-manifolds}{SPP~2026} funded by Deutsche Forschungsgemeinschaft.


\section{The setup} \label{setup}

Let $(X,g)$ be an $n$-dimensional Lorentzian manifold where $g$ has signature $(-,+,\ldots,+)$.
Denote the induced ``musical'' isomorphisms $TX\to T^*X$ and $T^*X\to TX$ by $v\mapsto v^\flat$ and $\xi\mapsto \xi^\sharp$, respectively.

\subsection{Normally hyperbolic and Dirac-type operators}
Let $\SS\to X$ be a complex vector bundle.

\begin{definition}
A linear differential operator $P$ of second order acting on sections of $\SS$ is called \emph{normally hyperbolic} if its principal symbol is given by the metric, i.e.\ $\sigma_P(\xi) = g(\xi^\sharp,\xi^\sharp)\cdot\id_{\SS_x}$ for each $x\in X$ and $\xi\in T^*_xX$.
\end{definition}
In other words, in local coordinates $x^1,\dots ,x^n$ on $X$ and with respect to a local trivialization of $\SS$, the operator $P$ takes the form
$$
P = -\sum_{i,j=1}^n g^{ij}(x)\frac{\partial^2}{\partial x^i \partial x^j}
+ \sum_{j=1}^n A_j(x) \frac{\partial}{\partial x^j} + B(x)
$$
where $A_j$ and $B$ are matrix-valued coefficients depending smoothly on $x$ and $(g^{ij})_{ij}$ is the inverse matrix of  $(g_{ij})_{ij}$ with
$g_{ij}=g\big(\frac{\partial}{\partial x^i},\frac{\partial}{\partial x^j}\big)$.\footnote{We use the convention that if an operator is locally given as $P=\sum_{|\alpha|\le m}A_\alpha \frac{\partial^{|\alpha|}}{\partial x^\alpha}$ and $\xi=\sum_k \xi_k dx^k$ then its principal symbol is $\sigma_P(\xi)=\rmi^m\sum_{|\alpha|= m}A_\alpha \xi^\alpha$.}

\begin{example}\label{ex:dAlembert}
Let $\nabla$ be a connection on $\SS$.
Together with the Levi-Civita connection on $TX$ this induces a connection on $T^*X\otimes\SS$, again denoted by $\nabla$.
Then $P = -\tr(\nabla\circ\nabla)$ is a normally hyperbolic operator.
Here $\tr: T^*X\otimes T^*X \otimes \SS \to\SS$ is the trace induced by the Lorentzian metric.
We will write 
$$
\Box^\nabla = -\tr(\nabla\circ\nabla)
$$
for this operator and call it the connection-d'Alembert operator for $\nabla$.
If $\SS$ is the trivial line bundle and $\nabla$ the standard flat connection, then $\Box=\Box^\nabla$ is the usual d'Alembertian.
\end{example}

\begin{remark}\label{rem:dAlembert}
Up to zero-order terms, all normally hyperbolic operators are of the form given in Example~\ref{ex:dAlembert}.
Namely, for each normally hyperbolic operator $P$ there exists a unique connection $\nabla$ on $\SS$ and a unique endomorphism field $B$ on $\SS$ such that 
$$
P = \Box^\nabla + B,
$$
compare \cite{BGP07}*{Lemma~1.5.5}.
We call $\nabla$ the connection \emph{induced by $P$}.
\end{remark}

\begin{definition}
A linear differential operator $\dirac$ of first order acting on sections of $\SS$ is said to be of \emph{Dirac type} if its square $\dirac^2$ is normally hyperbolic.
\end{definition}

Then the principal symbol $\sigma_\dirac$ satisfies $\sigma_\dirac(\xi)^2=\sigma_{\dirac^2}(\xi) = g(\xi^\sharp,\xi^\sharp)\cdot\id_{\SS_x}$ and hence, by polarization, 
\begin{equation}
\sigma_\dirac(\xi)\sigma_\dirac(\eta)+\sigma_\dirac(\eta)\sigma_\dirac(\xi) = 2g(\xi^\sharp,\eta^\sharp)\cdot\id_{\SS_x}.
\label{eq:Clifford1}
\end{equation}

We use Feynman's slash notation $\slashed v = \sigma_\dirac(v^\flat)$ for any $v\in TX$.
Then we have for any function $f$ and section $u$ of $\SS$
\begin{equation}
\dirac(fu) = f\dirac u - \rmi \sigma_\dirac(df)u = f\dirac u - \rmi \cancel{\grad f}\, u.
\label{eq:DiracLeibnitz}
\end{equation}
Equation~\eqref{eq:Clifford1} can be rewritten as
\begin{equation}
\slashed v \slashed w + \slashed w \slashed v = 2 g(v,w)\cdot\id_{\SS_x}
\label{eq:Clifford2}
\end{equation}
for any $v,w \in T_xX$.
Similarly, if $\nabla$ is a connection on $\SS$, we write $\D = \sum_{ij} g^{ij}\slashed b_i\nabla_j$
where $b_i=\nicefrac{\partial}{\partial x^i}$.
This definition of $\D$ is independent of the choice of coordinates.

From $\sigma_{\nabla_v}(\xi) = \rmi \xi(v) \id_\SS$ for any vector field $v$, we find
$$
\sigma_{-\rmi\D}(\xi)
=
-\rmi \sum_{k\ell} g^{k\ell}\slashed b_k \rmi \xi(b_\ell)
=
\slashed{\xi}^\sharp
=
\sigma_\dirac(\xi).
$$ 
Thus $\dirac$ and $-\rmi\D$ coincide up to a zero-order term, $\dirac=-\rmi\D + V$.

\begin{definition}
A Dirac-type operator $\dirac$ is called a \emph{Dirac operator in the sense of Gromov and Lawson} if $\dirac=-\rmi\D$ where $\nabla$ is a connection on $\SS$ such that $\sigma_\dirac$ is parallel, i.e., $\nabla_u(\sigma_\dirac(\xi)\varphi) = \sigma_\dirac(\nabla_u\xi)\varphi + \sigma_\dirac(\xi)\nabla_u\varphi$ for all $u\in TX$ and all differentiable 1-forms $\xi$ and sections $\varphi$ of $\SS$.
\end{definition}

\begin{example}
We assume a spin structure on $X$ is given so that the complex spinor bundle $SX\to X$ is defined.
Suppose that $E\to X$ is an additional complex vector bundle with a connection $\nE$. 

Spinors are sections of the twisted spinor bundle $\SS := SX\otimes E$.
Let $\nabla$ be the connection on $\SS$ induced by the Levi-Civita connection $\nabla^S$ on $SX$ and the connection $\nE$ on $E$.
Then Clifford multiplication on $SX$ induces the spinorial Dirac operator $\dirac = -\rmi\D$.
The principal symbol of $\dirac$ coincides with Clifford multiplication and is parallel with respect to $\nabla$.
Thus, the spinorial Dirac operator is a Dirac operator in the sense of Gromov and Lawson.
\end{example}

\subsection{Restriction to hypersurfaces}

Let $\dirac$ be a Dirac-type operator acting on sections of $\SS$ and let $\Sigma\subset X$ be a smooth spacelike hypersurface.
We write $\dirac = -\rmi\D + V$ where $\nabla$ is a connection on $\SS$.
Let $n_\Sigma$ be a unit normal field along $\Sigma$.
Choose a local orthonormal tangent frame $b_2,\ldots,b_n$ along $\Sigma$.
We compute, using $\nS^2 = g(n_\Sigma,n_\Sigma)=-1$,
\begin{align*}
-\rmi\D
&=
-\rmi \Big(-\nS \nabla_{\nS} + \sum_{k=2}^n \slashed b_k\nabla_{b_k}\Big) \\
&=
\rmi\nS \Big(\nabla_{\nS} +\nS \sum_{k=2}^n \slashed b_k\nabla_{b_k}\Big) \\
&=
\rmi\nS \Big(\nabla_{\nS} +\rmi\sum_{k=2}^n \gamma_\Sigma(b_k)\nabla_{b_k}\Big) 
\end{align*}
where we have put $\gamma_\Sigma(b) = -\rmi\nS\slashed b$.
Equation~\eqref{eq:Clifford2} implies the relations 
$$
\gamma_\Sigma(b)\gamma_\Sigma(c)+\gamma_\Sigma(c)\gamma_\Sigma(b)=-2g(b,c)\cdot\id_\SS
$$ 
for all $b,c\in T\Sigma$.
Defining the Riemannian Dirac operator on $\Sigma$ by $\DS = \sum_{k=2}^n \gamma_\Sigma(b_k)\nabla_{b_k}$ we find
$$
\dirac = \rmi\nS \big(\nabla_{\nS} +\rmi\DS\big) + V.
$$
Note that the induced Dirac operator $\DS$ depends on the chosen sign of the unit normal vector field $n_\Sigma$. 
We will later choose this field to be future directed.

\begin{example}
Let $\dirac$ be the spinorial Dirac operator acting on sections of $SX$ and $\nabla$ the Levi-Civita connection.
Then $V=0$ and we have 
\begin{equation}
\dirac = \rmi\nS \big(\nabla^{SX}_{\nS} +\rmi\DS^{SX}\big).
\label{eq:DiracSpin}
\end{equation}
Let $n$ be odd; the even-dimensional case will be treated in Example~\ref{ex:Dirac-evendim}.
For $n$ odd, $SX|_\Sigma$ can be a naturally identified with the spinor bundle of $\Sigma$ and $\gamma_\Sigma$ is Clifford multiplication on $\Sigma$.
The operator $\DS$ is not in general the spinorial Dirac operator of $\Sigma$ however, because the restriction of the Levi-Civita connection of $X$ to $\Sigma$ does not coincide with the Levi-Civita connection of $\Sigma$.
Instead, we have the spinorial Gauss formula
$$
\nabla^{SX}_b = \nabla^{S\Sigma}_b - \tfrac12  \nS \cancel{W(b)}
$$
where $W$ is the Weingarten map with respect to $n_\Sigma$, see \cite{BGM}*{eq.~(3.5)}.
Using an orthonormal eigenbasis $b_2,\ldots,b_n$ of $W$ for the eigenvalues $\kappa_2,\ldots,\kappa_n$ we find
$$
\DS^{SX} 
= \DS^{S\Sigma} - \tfrac12 \sum_{k=2}^{n} \gamma_\Sigma(b_k)\nS \kappa_k \slashed b_k
= \DS^{S\Sigma} - \tfrac\rmi 2 \sum_{k=2}^{n} \kappa_k 
= \DS^{S\Sigma} - \tfrac{\rmi (n-1)}{2}H .
$$
Here $H=\frac{1}{n-1}\tr(W)$ is the mean curvature of $\Sigma$ with respect to the normal $n_\Sigma$.
Inserting this into \eqref{eq:DiracSpin} we get 
\begin{equation}
\dirac = \rmi\nS \big(\nabla^{SX}_{\nS} +\rmi\DS^{S\Sigma} +\tfrac{n-1}{2}H \big).
\label{eq:spinDirac}
\end{equation}
Tensoring by a complex vector bundle $E$, i.e.\ replacing $SX$ by $\SS = SX\otimes E$ does not affect this formula.
\end{example}

\begin{definition}\label{def:ProductStructure}
We say that a Dirac-type operator $\dirac$ has \emph{product structure near} $\Sigma$ if the following hold:
\begin{enumerate}[(1)]
\item 
The manifold $X$ is a metric product in a neighborhood $\UU$ of $\Sigma$.
This means that there is an isometry $\Psi:\UU\to (-\delta,\delta)\times\Sigma$ with $\Psi(s)=(0,s)$ for $s\in\Sigma$ and $(-\delta,\delta)\times\Sigma$ carries the product metric $-dt^2+h$.
Here $t$ is the variable in $(-\delta,\delta)$ and $h$ is the Riemannian metric induced on $\Sigma$.
Note that $h$ does not depend on $t$.
There is no loss of generality in additionally assuming that $\Psi_*n_\Sigma=\nicefrac{\partial}{\partial t}$.
\item
The isometry $\Psi$ is covered by a connection-preserving vector bundle isomorphism $\SS|_\UU \to \pi^*(\SS|_\Sigma)$ where $\pi:(-\delta,\delta)\times\Sigma\to\Sigma$ is the projection.
Here $\SS$ carries some connection and $\pi^*(\SS|_\Sigma)$ carries the pullback of the induced connection on $\SS|_\Sigma$.
\item
Under this vector bundle isomorphism the Dirac-type operator takes the form
\begin{equation}
\dirac = \rmi\nS \bigg(\frac{\partial}{\partial t} +\rmi\DS\bigg).
\label{eq:DiracProduct}
\end{equation}
\end{enumerate}
If we want to emphasize the neighborhood, we also say that $\dirac$ has \emph{product structure over} $\UU$.
\end{definition}

Note that the Riemannian Dirac-type operator $\DS$ in \eqref{eq:DiracProduct} is independent of $t$.
For the spinorial Dirac operator there is no mean curvature correction term in this case because the product structure ensures that $\Sigma$ is totally geodesic in $X$.

If $\dirac$ has product structure then $\nS$ and $\nicefrac{\partial}{\partial t}$ commute, $\nS$ and $\DS$ anticommute, and $\DS$ and $\nicefrac{\partial}{\partial t}$ commute.
This then implies
$$
\dirac^2 = \frac{\partial^2}{\partial t^2} + \DS^2 
$$
and \eqref{eq:DiracProduct} can be rewritten as
\begin{equation}\label{eq:DiracProduct2}
\dirac = \bigg(\frac{\partial}{\partial t} -\rmi\DS\bigg)\rmi\nS .
\end{equation}

In particular, if $\dirac$ has product structure then so has $\dirac^2$ in the following sense:

\begin{definition}\label{def:ProductStructureNormHyp}
We say that a normally hyperbolic operator $P$ has \emph{product structure near} $\Sigma$ (or \emph{product structure over} $\UU$) if (1) and (2) hold as in Definition~\ref{def:ProductStructure} and
\begin{enumerate}[(1)]
\setcounter{enumi}{2}
\item
under the vector bundle isomorphism $\Phi$ the operator $P$ takes the form
\begin{equation*}
P = \frac{\partial^2}{\partial t^2} +\Delta
\end{equation*}
where $\Delta$ is a Laplace-type operator on $\Sigma$, independent of $t$.
\end{enumerate}

\end{definition}

\subsection{Odd Dirac-type operators}

Classical spinors on even-dimensional spacetimes decompose into left-handed and right-handed spinors.
This leads to the concept of odd Dirac-type operators.

\begin{definition}
Let the bundle $\SS$ carry a connection $\nabla$ and come with a splitting $\SS=\SS_L\oplus\SS_R$ which is parallel with respect to $\nabla$.
Then a Dirac-type operator $\dirac$ is called \emph{odd} if it interchanges the subbundles, i.e.\ it takes the form
$$
\dirac =
\begin{pmatrix}
0 & \dirac_R \\
\dirac_L & 0
\end{pmatrix}
$$
with respect to this splitting.
\end{definition}

\begin{example}\label{ex:Dirac-evendim}
If $X$ is even-dimensional then the spinor bundle $SX$ splits into the bundles of left-handed and right-handed spinors, $SX=S_LX\oplus S_RX$ and, accordingly, $\SS=\SS_L\oplus\SS_R$ where $\SS_{L/R} = S_{L/R}X\otimes E$.
The spinorial Dirac operator is then indeed an odd operator.

Let $\Sigma\subset X$ be a spacelike hypersurface.
Both restrictions $S_LX|_\Sigma$ and $S_RX|_\Sigma$ can be identified with the spinor bundle $S\Sigma$ of $\Sigma$.
Clifford multiplication on $S\Sigma$ corresponds to $\gamma_\Sigma$ on $S_LX|_\Sigma$ and to $-\gamma_\Sigma$ on $S_RX|_\Sigma$.
Note that $\nS:S_LX|_\Sigma\to S_RX|_\Sigma$ is a vector bundle isomorphism which intertwines Clifford multiplication.
Denoting the spinorial Dirac operator on $\Sigma$ by $\DS$, equation~\eqref{eq:spinDirac} becomes
\begin{align*}
\dirac_L &= \rmi\nS \big(\nabla^{SX}_{\nS} +\rmi\DS +\tfrac{n-1}{2}H \big), \\ 
\dirac_R &= \rmi\nS \big(\nabla^{SX}_{\nS} -\rmi\DS +\tfrac{n-1}{2}H \big). 
\end{align*}
Again, twisting with $E$ does not affect these formulas.
\end{example}

\begin{definition}
We say that an odd Dirac-type operator $\dirac$ has \emph{product structure near} $\Sigma$ if in addition to the conditions in Definition~\ref{def:ProductStructure} the vector bundle isomorphism $\SS|_\UU = \SS_L|_\UU\oplus\SS_R|_\UU \to \pi^*(\SS|_\Sigma)=\pi^*(\SS_L|_\Sigma)\oplus\pi^*(\SS_R|_\Sigma)$ preserves the splittings.
\end{definition}

\subsection{The Cauchy problem}

From now on let us assume that $X$ is globally hyperbolic meaning that there exist Cauchy hypersurfaces.
Denote by $C^\infty_\mathrm{sc}(X;\SS)$ the space of smooth sections with \emph{spatially compact} support.
This means that the support is contained in a set of the form $\J^+(K)\cup \J^-(K)$ where $K\subset X$ is compact and $\J^\pm$ denotes the causal future and past, respectively.
Denote the space of smooth solutions of $\dirac u=0$ with spatially compact support by 
$$
\Sol(\dirac):=\{u\in C^\infty_\mathrm{sc}(X;\SS) \mid \dirac u=0\}
$$ 
and similarly for $\dirac_R$ and $\dirac_L$ in the odd case.
The Cauchy problem for the Dirac equation on globally hyperbolic manifolds is well posed, see e.g.\ \cite{MR637032}*{Thm.~2.3}.
This means that for any smooth spacelike Cauchy hypersurface $\Sigma\subset X$ the restriction map
$$
\rho_\Sigma : \Sol(\dirac) \to C^\infty_\mathrm{c}(\Sigma;\SS),\quad
u \mapsto u|_\Sigma,
$$
is an isomorphism of topological vector spaces.
If $\Sigma$ and $\Sigma'$ are two smooth spacelike Cauchy hypersurfaces of $X$, then we put 
$$
U_{\Sigma',\Sigma}:=\rho_{\Sigma'}\circ(\rho_\Sigma)^{-1} : 
C^\infty_\mathrm{c}(\Sigma;\SS)\to C^\infty_\mathrm{c}(\Sigma';\SS).
$$
By duality, $U_{\Sigma',\Sigma}$ extends to an isomorphism on distributional sections,
\begin{equation*}
U_{\Sigma',\Sigma}:\DD'(\Sigma;\SS)\to \DD'(\Sigma';\SS)
\end{equation*}
and then restricts to an isomorphism 
\begin{equation}
U_{\Sigma',\Sigma}:L^2_\mathrm{c}(\Sigma;\SS)\to L^2_\mathrm{c}(\Sigma';\SS).
\label{eq:U}
\end{equation}
Here $L^2_\mathrm{c}(\Sigma;\SS)$ denotes compactly supported square integrable sections.

\subsection{The formally dual operator}
Let $\dirac$ be a Dirac-type operator acting on sections of $\SS$.
Then there exists a unique first-order differential operator $\dirac^*$ acting on sections of $\SS^*$ such that the following \emph{Green's formula} holds:
\begin{equation}
\int_X v(\dirac u) \dV - \int_X (\dirac^*v)(u) \dV
=
\rmi\left\{\int_{\Sigma_-}v(\nSm u) \dA - \int_{\Sigma_+}v(\nSp u) \dA \right\} .
\label{eq:Green}
\end{equation}
Here $\Sigma_-, \Sigma_+ \subset X$ are two disjoint smooth spacelike Cauchy hypersurfaces of $X$ with $\Sigma_-$ lying in the past of $\Sigma_+$ and $X=\J^+(\Sigma_-)\cap \J^-(\Sigma_+)$ is the spacetime region in between these two Cauchy hypersurfaces.
Note that $\partial X = \Sigma_- \cup \Sigma_+$.
Moreover, $u$ and $v$ are compactly supported smooth sections of $\SS$ and $\SS^*$, respectively, and $\slashed n_{\Sigma_\pm}$ are the future directed unit normal fields along $\Sigma_\pm$.

The operator $\dirac^*$ is again a Dirac-type operator and is called the \emph{formal dual} of $\dirac$.
Now let 
$$
U = U_{\Sigma_+,\Sigma_-}\colon L^2_\mathrm{c}(\Sigma_-;\SS)\to L^2_\mathrm{c}(\Sigma_+;\SS)
$$
be the Cauchy solution operator for $\dirac$ and
$$
\tilde U = \tilde U_{\Sigma_+,\Sigma_-}\colon L^2_\mathrm{c}(\Sigma_-;\SS^*)\to L^2_\mathrm{c}(\Sigma_+;\SS^*)
$$
that of $\dirac^*$.
This induces an operator
$$
\tilde U^*\colon L^2_\mathrm{c}(\Sigma_+;\SS)\to L^2_\mathrm{c}(\Sigma_-;\SS)
$$
characterized by 
$$
\int_{\Sigma_+} (\tilde Uv)(u)\dA = \int_{\Sigma_-} v(\tilde U^*u)\dA 
$$
for all $u\in L^2_\mathrm{c}(\Sigma_+;\SS)$ and $v\in L^2_\mathrm{c}(\Sigma_-;\SS^*)$.

Green's formula \eqref{eq:Green} with $\dirac^*u=\dirac v=0$ on $X$ yields 
$$
\int_{\Sigma_-}v_-(\nSm u_-) \dA  
= \int_{\Sigma_+}(\tilde Uv_-)(\nSp Uu_-) \dA
= \int_{\Sigma_+}v_-(\tilde U^*\nSp Uu_-) \dA
$$
for all $u_-\in C^\infty_\mathrm{c}(\Sigma_-;\SS)$ and $v_-\in C^\infty_\mathrm{c}(\Sigma_-;\SS^*)$.
Hence, $\tilde U^*\nSp U=\nSm$.
From $\nSp^2=\nSm^2=-\id$ we conclude that the following diagram commutes:
\begin{equation}
\begin{gathered}
\xymatrix{
L^2_\mathrm{c}(\Sigma_-;\SS) \ar[r]^{U}  & L^2_\mathrm{c}(\Sigma_+;\SS) \\
L^2_\mathrm{c}(\Sigma_-;\SS) \ar[u]^{\nSm} & L^2_\mathrm{c}(\Sigma_+;\SS) \ar[l]_{\tilde U^*} \ar[u]^{\nSp}.
}
\label{eq:diagram-U}
\end{gathered}
\end{equation}

\section{\texorpdfstring{Laplace-type and Dirac-type operators on $\Sigma$}{Laplace-type and Dirac-type operators on Sigma}}
\label{sec:LaplaceDirac}

In case $X$ is a metric product, $(X,g)= (\R \times \Sigma, -dt^2+h)$, the study of normally hyperbolic operators of product type as well as Dirac-type operators of product type is reduced to the study of Laplace-type and Dirac-type operators on the Cauchy hypersurface $\Sigma$.
Note that the metric product $\R \times \Sigma$ is globally hyperbolic if and only if $(\Sigma,h)$ is a complete Riemannian manifold, which we will assume.

\subsection{Laplace-type operators} 
Let $\SS \to \Sigma$ be a complex vector bundle and $\Delta$ a second-order differential operator of Laplace type, i.e.\ $\sigma_\Delta(\xi)=h(\xi^\sharp,\xi^\sharp)\cdot\id_\SS$.
We are now going to assume that \eqref{FallA} or \eqref{FallB} below holds:
\begin{enumerate}[(A)]
\item\label{FallA}
There exists a positive definite Hermitian inner product on $\SS$ with respect to which $\Delta$ is formally selfadjoint and nonnegative.
Furthermore, the essential spectrum of the Friedrichs extension on compactly supported smooth sections is contained in $[\delta,\infty)$ for some $\delta>0$.
\item\label{FallB}
$\Sigma$ is compact.
\end{enumerate}

Condition \eqref{FallA} is a condition on the behavior of the operator at infinity. There are natural examples, even squares of Dirac operators on non-compact spin manifolds, where the spectral gap can be made arbitrarily large or infinite (see \cite{Baer-Hyp} for an example).

In either case, there is a well-defined meaning of $L^2(\Sigma; \SS)$ as a Hilbertizable locally convex topological vector space.
In case~\eqref{FallA}, the operator $\Delta$ is essentially selfadjoint on $C^\infty_\mathrm{c}(\Sigma;\SS)$ and the spectrum of the selfadjoint closure is contained in the nonnegative real axis.
In case~\eqref{FallB}, the resolvent of the elliptic operator $\Delta$ is compact and the spectrum is discrete (but not necessarily real).

In either case, zero lies in the resolvent set of $\Delta$ or is an isolated eigenvalue of finite multiplicity.
In case \eqref{FallB}, $\Delta$ restricted to the range of $p_0$ may have a nontrivial Jordan normal form.
We can define the projection onto the generalized kernel of $\Delta$ by 
\begin{equation}
 p_0 := \frac{1}{2 \pi \rmi} \int_{C} (\Delta-\lambda)^{-1} d \lambda,
\label{eq:p0}
\end{equation}
where $C$ is a small circle about $0$ that encloses no point in the spectrum of $\Delta$ other than zero.
Integration is taken counter-clockwise.
In case \eqref{FallA}, $p_0$ coincides with the operator obtained by applying the characteristic function of $0$ in $\R$ to the operator $\Delta$ using Borel functional calculus for selfadjoint operators.

In both cases we have $L^2(\Sigma;\SS) = \rg(p_0)\oplus \ker(p_0)$.
By elliptic regularity, $\rg(p_0)\subset C^\infty(\Sigma;\SS)$ and, by our assumptions, $\dim\rg(p_0)<\infty$.
Hence, $p_0$ has a smooth Schwartz kernel.

Next we need a square root of $\Delta$ and its inverse, at least on $\ker(p_0)$.
In case~\eqref{FallA}, $\Delta+p_0$ is invertible and selfadjoint and, using functional calculus for selfadjoint operators, we put for $s\in\C$: 
$$
\Delta^{s}_\theta := \exp(s\log(\Delta+p_0)) \circ (1-p_0).
$$ 
Here $\log$ is the standard branch of the logarithm defined on $\C\setminus (-\infty,0]$.

In case~\eqref{FallB}, $\Delta+p_0$ is an invertible elliptic pseudodifferential operator of second order whose principal symbol is scalar and takes values on the positive real half-line.
Thus, we can choose a ray of minimal growth $\{\arg(z)=\theta\}$ and form the complex powers $(\Delta+p_0)^s$ as constructed in \cite{Seeley}.
We put 
$$
\Delta^s_\theta := (\Delta+p_0)^s \circ (1-p_0).
$$
The subindex $\theta$ reminds us of the dependence of the complex powers on the choice of ray of minimal growth in case~\eqref{FallB}.
In case~\eqref{FallA}, the definition of $\Delta^{s}_\theta$ is independent of $\theta$.

\begin{prop}\label{prop:powers}
Let $s,t\in\C$.
In both cases \eqref{FallA} and \eqref{FallB}, the operator $\Delta_\theta^{s}$ is closed and densely defined in $L^2(\Sigma;\SS)$ with core $C^\infty_\mathrm{c}(\Sigma;\SS)$.
The following holds:
\begin{enumerate}[(1)]
\item\label{power1}
$\Delta_\theta^{s}$ is bounded if $\Re(s)\le0$.
\item\label{power2}
If $\Re(s)\leq0$ and $B$ is a bounded operator that commutes with the resolvent of $\Delta$, then $B$ commutes with $\Delta_\theta^{s}$.
\item\label{power3}
$\Delta_\theta^0=1-p_0$.
\item\label{power4}
$\Delta_\theta^1=\Delta\circ(1-p_0)$.
\item\label{power5}
$\Delta_\theta^s\circ\Delta_\theta^t = \Delta_\theta^{s+t}|_{\mathrm{dom}(\Delta_\theta^{s+t})\cap\mathrm{dom}(\Delta_\theta^{t})}$.
\item\label{power6}
$\Delta_\theta^{s}\circ p_0 = 0$ and $p_0 \circ\Delta_\theta^{s}=0|_{\mathrm{dom}(\Delta_\theta^{s})}$.
\item\label{power7}
$\Delta\circ\Delta_\theta^s = \Delta_\theta^{1+s}$ and $\Delta_\theta^s\circ\Delta = \Delta_\theta^{1+s}|_{\mathrm{dom}(\Delta_\theta^{1+s})\cap\mathrm{dom}(\Delta)}$.
\item\label{power9}
The Schwartz kernel of $p_0$ is smooth.
\end{enumerate}
In case \eqref{FallA} we have furthermore:
\begin{enumerate}[(1)]
\setcounter{enumi}{8}
\item\label{power8}
If $s\in\R$ then the operator $\Delta_\theta^{s}$ is selfadjoint and nonnegative.
\end{enumerate}
In case \eqref{FallB} we have furthermore:
\begin{enumerate}[(1')]
\setcounter{enumi}{8}
\item\label{power8strich}
The operator $\Delta_\theta^{s}$ is a classical pseudodifferential operator of order $2\Re(s)$ with positive scalar principal symbol.
\end{enumerate}
\end{prop}

\begin{proof}
In case \eqref{FallA}, assertions~\eqref{power1}--\eqref{power7} and \eqref{power8} follow from properties of the Borel functional calculus for selfadjoint operators.
Note that all functions of $\Delta$ commute.
In particular, we have $\exp(s\log(\Delta+p_0)) \circ (1-p_0)= (1-p_0)\circ\exp(s\log(\Delta+p_0))$.

In case \eqref{FallB}, assertions \eqref{power1}, \eqref{power3}--\eqref{power5}, and (\ref{power8strich}') follow from \cite{Seeley}*{Thm.~3}. 
Assertion~\eqref{power2} follows from the contour integration formula for the complex powers in \cite{Seeley}*{p.~289}.
This also shows $(\Delta+p_0)^s \circ (1-p_0)= (1-p_0)\circ(\Delta+p_0)^s$.
Assertions \eqref{power6} and \eqref{power7} follow by elementary computation using $(1-p_0)p_0=0$.

It remains to show \eqref{power9}.
Since $\Delta|_{\rg(p_0)}$ is nilpotent, $\rg(p_0)$ coincides with the kernel of $L:=\Delta^N$ for $N$ sufficiently large.
Let $L^*$ be the formally dual operator acting on sections of $\SS^*\to\Sigma$ and consider the elliptic operator $\LL := L\otimes 1 + 1 \otimes L^*$ acting on sections of $\SS\boxtimes\SS^*\to\Sigma\times\Sigma$.
Let $K\in \DD'(\Sigma\times\Sigma;\SS\boxtimes\SS^*)$ be the Schwartz kernel of $p_0$ and let $u\in C^\infty_\mathrm{c}(\Sigma; \SS^*)$ and $v\in C^\infty_\mathrm{c}(\Sigma; \SS)$ be test sections.
Then
\begin{align*}
\LL(K)(u\otimes v)
&=
K(\LL^*(u\otimes v))
=
K(L^*u\otimes v+u\otimes Lv) \\
&=
p_0(v)(L^*u) + p_0(Lv)(u)
=
2p_0(Lv)(u)
=
0.
\end{align*}
Thus, $\LL(K)=0$ and hence $K$ is smooth by elliptic regularity.
\end{proof}

We are mostly interested in the cases $s=\nicefrac12$ and $s=-\nicefrac12$.

\begin{prop}
The spectrum of $\Delta_\theta^{\nicefrac{1}{2}}$ is contained in a strip of the form
\begin{equation}
  \mathrm{spec}\big(\Delta_\theta^{\nicefrac{1}{2}}\big) \subset \{ z \in \C \mid \Re(z) \geq -C, \; | \Im(z) | \leq C  \}
  \label{eq:SpecWurzel}
\end{equation}
for some $C>0$, see Figure~\ref{fig:strip}.
The operator $\Delta_\theta^{\nicefrac{1}{2}}$ generates a strongly continuous semigroup $t\mapsto e^{\rmi t \Delta_\theta^{\nicefrac{1}{2}}}$ on the closed upper half-plane, that is analytic on the open upper half-plane. 
\end{prop}

\begin{figure}[!ht]
\includegraphics[width=0.6\textwidth]{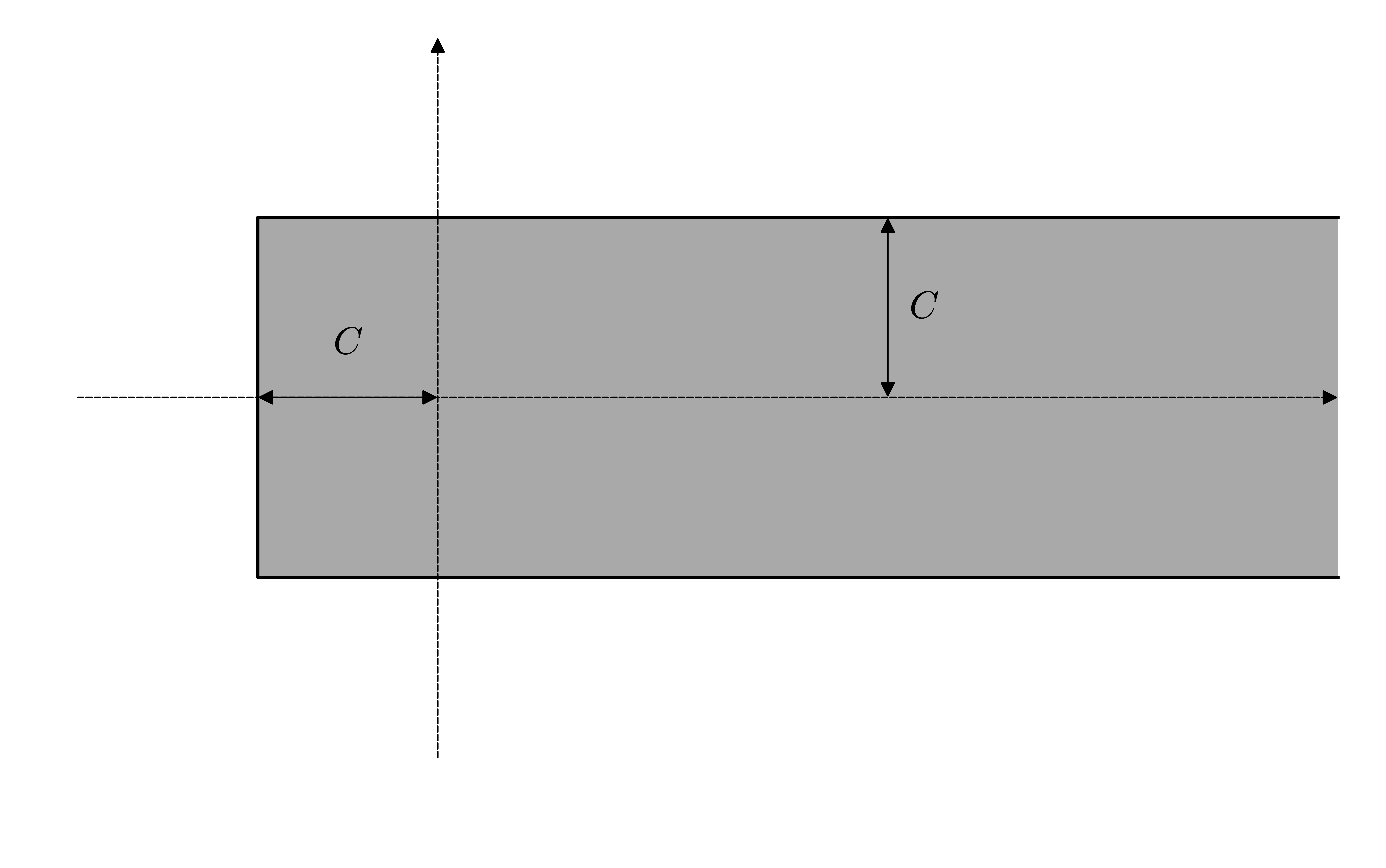}
\caption{Strip containing the spectrum of $\Delta_\theta^{\nicefrac{1}{2}}$}
\label{fig:strip}
\end{figure}

\begin{proof}
In case \eqref{FallA}, the spectrum of $\Delta_\theta^{\nicefrac{1}{2}}$ is positive real.
Thus, \eqref{eq:SpecWurzel} holds trivially in that case.

In case \eqref{FallB}, choose a positive definite inner product on the vector bundle $\SS$ and a positive selfadjoint second-order differential operator $A$ with the same principal symbol as $\Delta$.
Then $A^{\nicefrac{1}{2}}$ is a positive selfadjoint pseudodifferential operator of first order and $B:=\Delta_{\theta}^{\nicefrac{1}{2}} - A^{\nicefrac{1}{2}}$ is a zero-order pseudodifferential operator.
In particular, $B$ is $L^2$-bounded. 

For any $z$ with $\mathrm{dist}(z,\R_+)> 2 \| B \|$ we have $\|B (A^{\nicefrac{1}{2}} - z)^{-1}\|\le\|B\|\cdot\frac{1}{\mathrm{dist}(z,\R_+)}<\frac{1}{2}$ and therefore
$$
  (\Delta_\theta^{\nicefrac{1}{2}} - z)^{-1} =  (A^{\nicefrac{1}{2}} - z)^{-1} ( 1 + B (A^{\nicefrac{1}{2}} - z)^{-1})^{-1}.
$$
In particular, the spectrum of $\Delta_{\theta}^{\nicefrac{1}{2}}$ is contained in the $2\|B\|$-neighborhood of $\R_+$.
This proves \eqref{eq:SpecWurzel} in case~\eqref{FallB}.

Moreover, still in case \eqref{FallB}, we get for $\mathrm{dist}(z,\R_+)> 2 \| B \|$ the resolvent estimate
 \begin{align*}
  \| (\Delta_\theta^{\nicefrac{1}{2}} - z)^{-1} \| 
  &\leq
  \| (A^{\nicefrac{1}{2}} - z)^{-1} \| \cdot \| ( 1 + B (A^{\nicefrac{1}{2}} - z)^{-1})^{-1}\| \\
  &\leq
  \| (A^{\nicefrac{1}{2}} - z)^{-1} \| \cdot \frac{1}{1-\|B (A^{\nicefrac{1}{2}} - z)^{-1}\|}   \\
  &\leq 2 \,\| (A^{\nicefrac{1}{2}} - z)^{-1} \|  \\
  &\leq \frac{2}{\mathrm{dist}(z,\R_+)} .
 \end{align*}
 
In case \eqref{FallA}, the resolvent estimate $\| (\Delta_\theta^{\nicefrac{1}{2}} - z)^{-1} \| \leq \frac{1}{\mathrm{dist}\big(z,\mathrm{spec}(\Delta_\theta^{\nicefrac{1}{2}})\big)}$ holds because $\Delta_\theta^{\nicefrac{1}{2}}$ is selfadjoint.

In both cases \eqref{FallA} and \eqref{FallB}, the spectrum of the operator $\Delta_\theta^{\nicefrac{1}{2}}+C'$ is contained in the strip $S(C'):=\{ z \in \C \mid \Re(z) \geq C'-C, \; | \Im(z) | \leq C  \}$ and for $z$ outside this strip we have the resolvent estimate 
$$
\| (\Delta_\theta^{\nicefrac{1}{2}}+C' - z)^{-1}\| \le \frac{2}{\mathrm{dist}(z,S(C'))}.
$$
In particular, the spectrum of $\Delta_\theta^{\nicefrac{1}{2}}+C'$ is contained in the closed sector $\Lambda:=\{|\arg(z)|\le\omega\}\cup\{0\}$.
Here $\omega>0$ can be made arbitrarily small by choosing $C'$ sufficiently large, see Figure~\ref{fig:sector}.
Outside any larger sector $\Lambda'$, the ratio $\frac{|z|}{\mathrm{dist}(z,S(C'))}$ is bounded.
Hence, we get the resolvent estimate $\| (\Delta_\theta^{\nicefrac{1}{2}}+C' - z)^{-1}\| \le \frac{C''}{|z|}$ on $\C\setminus\Lambda'$.
Thus, $\Delta_\theta^{\nicefrac{1}{2}}+C'$ is sectorial of angle $\omega$.

\begin{figure}[!ht]
\includegraphics[width=0.6\textwidth]{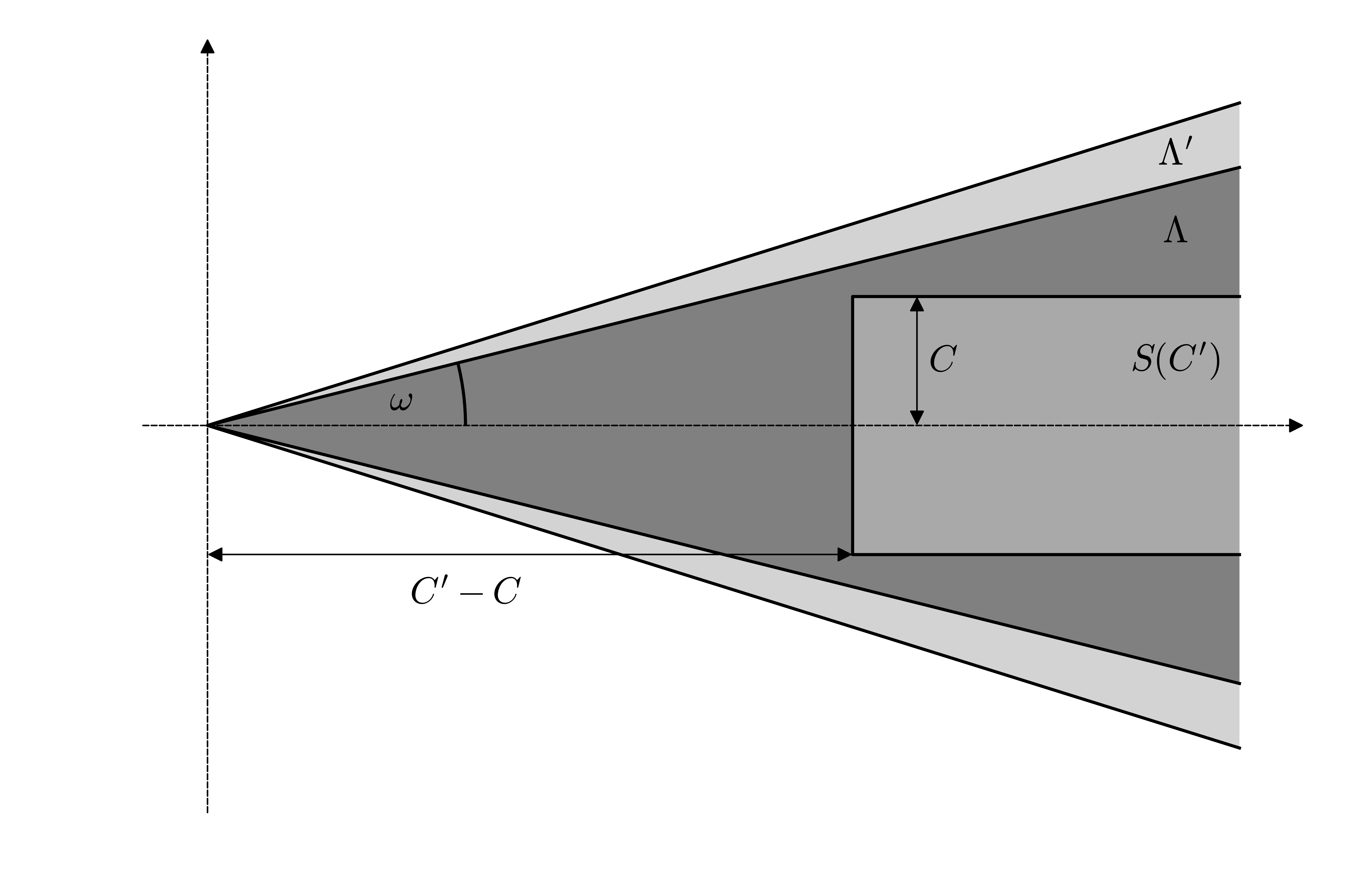}
\caption{$\omega$-sectoriality of $\Delta_\theta^{\nicefrac{1}{2}}+C'$}
\label{fig:sector}
\end{figure}

By \cite{Kato}*{Ch.~9, \S1.6}, $\Delta_\theta^{\nicefrac{1}{2}}+C'$ generates a holomorphic semigroup $e^{\rmi t(\Delta_\theta^{\nicefrac{1}{2}}+C')}$ on the open sector $\{\omega<\arg(t)<\pi-\omega\}$.
Then $e^{\rmi t(\Delta_\theta^{\nicefrac{1}{2}})}=e^{-\rmi tC'}e^{\rmi t(\Delta_\theta^{\nicefrac{1}{2}}+C')}$ is a holomorphic semigroup on this sector for $\Delta_\theta^{\nicefrac{1}{2}}$.
Since we can make $\omega$ arbitrarily small, we get the holomorphic semigroup $e^{\rmi t(\Delta_\theta^{\nicefrac{1}{2}})}$ on the open upper half-plane.

In case~\eqref{FallA}, we can also define $e^{\rmi t(\Delta_\theta^{\nicefrac{1}{2}})}$ by functional calculus thus extending it to the closed upper half-plane.
The family of functions $e^{\rmi tx}$ is uniformly bounded for $t\in\C$ with $\Im(t)\ge0$ and $x\in\R$ with $x\ge0$ and is continuous in $t$.
Lebesgue's dominated convergence theorem applied to the representation of $e^{\rmi t(\Delta_\theta^{\nicefrac{1}{2}})}u$ for $u\in L^2(\Sigma;\SS)$ as an integral against the spectral measure of $\Delta_\theta^{\nicefrac{1}{2}}$ shows strong continuity.

In case $\eqref{FallB}$, we note that $\Delta_\theta^{\nicefrac{1}{2}}$ is a pseudodifferential operator of order one with diagonal real principal symbol. 
Therefore, for any $u_0 \in L^2(\Sigma;\SS)$ the initial value problem of the symmetric hyperbolic system
$$
 \frac{du}{ds}(s) = \rmi \Delta_\theta^{\nicefrac{1}{2}} u(s), \quad u(0)=u_0,
$$
has a unique solution in $C(\R,L^2(\Sigma;\SS))$, see for example \cite{MR618463}*{Thm.~2.3 on p.~75}. 
Thus, the operator $\Delta_\theta^{\nicefrac{1}{2}}$ is the generator of a $C^0$-group $s\mapsto e^{\rmi s(\Delta_\theta^{\nicefrac{1}{2}})}$ defined for $s \in \R$.
This group commutes with the semigroup $e^{\rmi t(\Delta_\theta^{\nicefrac{1}{2}})}$ defined for $t$ in the open upper half-plane. 
In particular,
$$
 e^{\rmi t(\Delta_\theta^{\nicefrac{1}{2}})} e^{\rmi s(\Delta_\theta^{\nicefrac{1}{2}})} 
 = 
 e^{\rmi s(\Delta_\theta^{\nicefrac{1}{2}})} e^{\rmi t(\Delta_\theta^{\nicefrac{1}{2}})}
 =
 e^{\rmi (t+s)(\Delta_\theta^{\nicefrac{1}{2}})}.
$$
By sequential continuity of multiplication in the strong operator topology, the function $e^{\rmi t(\Delta_\theta^{\nicefrac{1}{2}})} u$ is sequentially norm-continuous on the closed upper half-plane for any $u \in L^2(\Sigma;\SS)$.
This implies that it is strongly continuous.
\end{proof}

\subsection{Dirac-type operators} 
Assume again that $\SS$ is a complex vector bundle and that $\DS$ is a first-order operator of Dirac type acting on sections of $\SS$.
Then $\Delta = \DS^2$ is of Laplace type. 
We will now find analogues of spectral projections in the case of compact $\Sigma$ without any assumption of selfadjointness. 

Consider the operators $\Delta_\theta^{-\nicefrac{1}{2}}$ and $p_0$ constructed above.
Then  $\Delta_\theta^{-\nicefrac{1}{2}}$ commutes with $\DS$ and $\Delta_\theta^{-\nicefrac{1}{2}} \DS$ is a zero-order pseudodifferential operator that squares to $1 - p_0$ by Proposition~\ref{prop:powers}~\eqref{power3}, \eqref{power5}, and \eqref{power7}.
Now define
\begin{align}
 p_{>} &:= \tfrac{1}{2} \left( 1 - p_0 +\Delta_\theta^{-\nicefrac{1}{2}} \DS \right), \label{eq:p>}\\ 
 p_{<} &:=  \tfrac{1}{2} \left( 1 - p_0 -\Delta_\theta^{-\nicefrac{1}{2}} \DS \right).\label{eq:p<}
\end{align}
In this notation for $p_{>}$ and $p_{<}$ we suppress the dependence on the choice of ray of minimal growth. 
We also define the projections 
\begin{equation}
p_\geq := p_> + p_0 \text{ and } p_\leq := p_< + p_0.
\label{eq:p=} 
\end{equation}

In case $\Sigma$ is a noncompact complete Riemannian manifold and there exists a positive definite Hermitian inner product on $\SS$ with respect to which $\DS$ is selfadjoint and $0$ does not lie in the essential spectrum, we are in case~\eqref{FallA}.
We can apply the Borel functional calculus for unbounded selfadjoint operators to $\DS$ and we define
$$
p_0 := \chi_{\{0\}}(\DS), \;  p_> := \chi_{(0,\infty)}(\DS), \; p_\geq := \chi_{[0,\infty)}(\DS),  \;
p_< := \chi_{(-\infty,0)}(\DS),\textrm{ and } \; p_\leq := \chi_{(-\infty,0]}(\DS).
$$
The projections are orthogonal in this case and the equations \eqref{eq:p0} and \eqref{eq:p>}--\eqref{eq:p=} still hold.

The definition of these projections is very natural and apart from the dependence on the ray of minimal growth is independent of the precise choice of functional calculus.
Proposition~\ref{prop:powers} and simple computation yields the following good properties.
\begin{prop}\label{prop:projectors}
Assume \eqref{FallA} or \eqref{FallB} as stated on page~\pageref{FallA}.
Then the following holds:
\begin{enumerate}[(1)]
\item 
$p_>$, $p_<$, and $p_0$ are projections, i.e. $p_>^2=p_>$, $p_<^2=p_<$, $p_0^2=p_0$.
\item 
$p_> + p_< + p_0 = 1$.
\item 
$p_> p_<= p_> p_0 = p_< p_0 =0$.
\item 
$p_>, p_<, p_0$ commute with each other and with $\DS$.
\item\label{projectors5} 
$(p_> - p_<)  \DS =\Delta_\theta^{\nicefrac{1}{2}}$.
\item 
The Schwartz kernel of $p_0$ is smooth.
\end{enumerate}
If $\Sigma$ is compact, we have furthermore:
\begin{enumerate}[(1)]
\setcounter{enumi}{6}
\item\label{projectors7} 
$p_>$ and $p_<$ are pseudodifferential operators of order zero.
\hfill$\Box$
\end{enumerate}
\end{prop}

\section{Distinguished parametrices for normally hyperbolic operators} \label{GanzFein:Section}

In this section we temporarily forget about Dirac-type operators and study special solutions to certain second-order equations.
Later this will be applied to the square of the Dirac operator.

Let $\SS\to X$ be a vector bundle over a globally hyperbolic Lorentzian manifold $X$.
For any distributional section 
$$
F \in \DD'(X \times X;\SS \boxtimes \SS^*)
$$
the associated operator $C_\mathrm{c}^\infty(X;\SS) \to \DD'(X;\SS)$ with Schwartz kernel $F$ will be denoted by $\hat F$.
In other words, for $u\in C_\mathrm{c}^\infty(X;\SS)$ and $v\in C_\mathrm{c}^\infty(X;\SS^*)$ we have
$$
\hat F(u)(v) = F(v\otimes u).
$$

For a distribution $F \in \DD'( X \times X, \SS \boxtimes \SS^*)$ let $\WF(F)\subset\dot T^*(X \times X)$ be its wavefront set.
Here $\dot T^*$ denotes the cotangent bundle with the zero-section removed.
As usual, we put
$$
 \WF'(F) := \{(x,\xi;x',\xi') \in \dot T^*(X \times X) \mid (x,\xi;x',-\xi') \in \WF(F)\}.
$$
Let $\Phi_t$ be the geodesic flow on $\dot T^*X$, as obtained from the geodesic flow on $TX$ by identifying $T^*X$ and $TX$ using the metric.
A causal covector is future directed if it gives a positive number when evaluated on a future directed timelike vector.
Hence, by our choice of signature, the metric identifies a future directed covector with a past directed vector. 
Therefore, somewhat unintuitively, $\Phi_t$ moves future directed covectors to the past and past directed covectors to the future for positive values of the parameter $t$.
We define the following subsets of $\dot T^*(X\times X)$:
\begin{align*}
\Delta^* 
&:= 
\{(x,\xi;x,\xi) \in \dot T^*(X \times X) \mid (x,\xi)\in\dot T^*X\},\notag\\
\Lambda_\Feyn^\pm
&:=
\Delta^* \cup \{(x,\xi;x',\xi') \in \dot T^*(X \times X) \mid \xi \textrm{ is lightlike,}\; \exists\, t>0: \Phi_{\mp t}(x,\xi) = (x',\xi') \},
\\
\Lambda_{\ret}
&:=
\Delta^* \cup \{(x,\xi;x',\xi') \in \dot T^*(X \times X) \mid \xi \textrm{ is lightlike,}\; x \in \J^+(x'),\; \exists\, t \in \R: \Phi_t(x',\xi') = (x,\xi) \}. \notag
 \\
\Lambda_{\adv}
&:=
\Delta^* \cup \{(x,\xi;x',\xi') \in \dot T^*(X \times X) \mid \xi \textrm{ is lightlike,}\; x \in \J^-(x'),\; \exists\, t \in \R: \Phi_t(x',\xi') = (x,\xi) \}. \notag
\end{align*}

Let $P$ be a normally hyperbolic operator on $X$ acting on the sections of $\SS$. 
Let us denote by $\hat G_{\ret/\adv}: C_\mathrm{c}^\infty(X;\SS) \to C^\infty(X;\SS)$ the unique retarded and advanced fundamental solutions of $P$, see \cite{BGP07}*{Sec.~3.4}.
Recall that $F$ is called a \emph{parametrix} for $P$ if the operators $P \hat F - 1$ and $\hat F P -1$ have smooth kernels. 
In particular, the retarded and advanced fundamental solutions are parametrices. 
The theory of distinguished parametrices for normally hyperbolic operators characterizes parametrices by the above wavefront sets, cf.\ \cite{MR0388464}*{Sect.~6.6}: 
the retarded fundamental solution satisfies $\WF'(G_\ret) = \Lambda_{\ret}$ and any other parametrix $G$ with $\WF'(G) \subset \Lambda_{\ret}$ coincides with $G_\ret$ modulo smoothing operators. 
A similar statement holds for the advanced fundamental solution. 
The sets $\Lambda_\Feyn^\pm$ characterize two more types of distinguished parametrices, the Feynman and anti-Feynman parametrices. 
We focus here on Feynman parametrices but similar statements hold for anti-Feynman parametrices.

\begin{definition}
A parametrix $G$ for $P$ is called a \emph{Feynman parametrix} if $\WF'(G) \subset \Lambda_\Feyn^+$. 
A fundamental solution is called a \emph{Feynman propagator} if it is a Feynman parametrix.
\end{definition}

It is shown in \cite{MR0388464}*{Thm.~6.5.3} that Feynman parametrices for \emph{scalar} normally hyperbolic operators are distinguished parametrices and hence are unique modulo smoothing operators. 
Moreover, any Feynman parametrix $G$ satisfies $\WF'(G) = \Lambda_\Feyn^+$. 
These results have been generalized to operators on vector bundles in \cite{strohmaierislam}*{Thm.~1.2}.

\subsection{The Feynman propagator in the product case}

Suppose that $X = \R \times \Sigma$ is equipped with a product metric and that $P$ has product structure over $X$,
$$
 P=\frac{\partial^2}{\partial t^2} + \Delta.
$$
Note that the operator $\Delta$ leaves the range of $p_0$ invariant and that $A := \Delta |_{\textrm{rg}(p_0)}$ is nilpotent.
We can therefore define the operator $A^{-\nicefrac{1}{2}} \sin(tA^{\nicefrac{1}{2}})$ on $\textrm{rg}(p_0)$ as the polynomial 
$$
A^{-\nicefrac{1}{2}} \sin(tA^{\nicefrac{1}{2}})=\sum_{k=0}^\infty \frac{(-1)^k}{(2k+1)!} t^{2k+1} A^{k},
$$ 
slightly abusing notation.
\begin{theorem}\label{thm:FeynmanProductWave}
Assume \eqref{FallA} or \eqref{FallB} as stated on page~\pageref{FallA}.
Suppose furthermore that $P$ has product structure over $X=\R\times\Sigma$.
Then the family of operators
\begin{gather} \label{prodfeyn}
 k(t) := 
 \tfrac{\rmi}{2} \e^{-\rmi | t | \Delta_\theta^{\nicefrac{1}{2}} }   \Delta_\theta^{-\nicefrac{1}{2}} +  \chi_{[0,\infty)} (t) A^{-\nicefrac{1}{2}} \sin(tA^{\nicefrac{1}{2}}) \, p_0
\end{gather}
defines a Feynman propagator $G$ via
\begin{gather} \label{LaplaceHadamard}
 (\hat G u)(t,\cdot) = \int_\R k(t-s) u(s,\cdot) ds .
\end{gather}
\end{theorem}

\begin{proof}
Let $u \in C_\mathrm{c}^\infty(\R \times \Sigma; \SS)$. 
Using Proposition~\ref{prop:powers}, a direct computation gives
$$
\Big(\frac{\partial^2}{\partial t^2} + \Delta\Big)\circ \hat G (u)= \hat G \circ \Big(\frac{\partial^2}{\partial t^2} + \Delta\Big)(u) = u . 
$$
It remains to show the wavefront set relation. 
Because of the above equation, $\WF'(G)$ contains $\Delta^*$, and $\WF'(G) \setminus \Delta^*$ contains only lightlike nonzero covectors.
Moreover, by H\"ormander's propagation of singularities theorem applied to both of the above equations considered on the product manifold $X \times X$, it is invariant under the geodesic flow in the sense that it is invariant under the $\R^2$-action induced by $\Phi_t \times \Phi_s$.
Define the open subsets $Y_+=\{((t,x),(s,y)) \in X \times X \mid t < s\} \subset X \times X$ and $Y_-=\{((t,x),(s,y)) \in X \times X \mid t > s\} \subset X \times X$.
By invariance under $\Phi_t \times \Phi_s$ it is sufficient to show that upon restriction to $Y_+$ only vectors of the form $((x,\xi),(x',\xi'))$ with future directed (possibly zero) $\xi,\xi'$ are in $\WF'(G)$, and that upon restriction to $Y_-$ only vectors of the form $((x,\xi),(x',\xi'))$ with past directed (possibly zero) $\xi,\xi'$  are in $\WF'(G)$. 
Indeed, this implies that $((x,\xi),(x',\xi'))$ can only be in $\WF'(G)$ if an element in $\Delta^*$ is on the same orbit. 
Otherwise, one can apply the flow to move the nonzero covector from $Y_\pm$ to $Y_\mp$, 
which would imply that $\xi$ as well as $\xi'$ are both future and past directed. 
Thus, $\WF'(G)$ can only contain elements of the form $((x,\xi),(x',\xi'))$ with $(x',\xi') = \Phi_t (x,\xi)$ for some $t \in \R$. 
If $t \not= 0$ this means that $(x,x')$ is either in $Y_+$ or in $Y_-$. 
It then follows that $\WF'(G) \subset \Lambda_\Feyn^+$.

We therefore only need to show that the restriction of $G$ to $Y_\pm$ has the claimed structure. 
In $Y_\pm$ the kernel of the second term $\chi_{[0,\infty)} (t-s) A^{-\frac{1}{2}} \sin(A^\frac{1}{2} (t-s)) \, p_0$ is smooth. 
It is therefore sufficient to consider the term $\frac{\rmi}{2} \e^{-\rmi | t-s | \Delta_\theta^{\nicefrac12} }$ only. 

First consider its restriction to $Y_+$.
For $t<s$ this is of the form $\frac{\rmi}{2} \e^{\rmi ( t-s ) \Delta_\theta^{\nicefrac12} }$ and therefore the strong boundary value of a function that is strongly holomorphic in the parameters $t$ and $s$ in the region $\Im(t-s)>0$ and smooth in other variables. 
We now apply Theorem 8.1.6 in \cite{Ho1} in the variables $t$ and $s$ to conclude that $\WF(G)$ can only contain points $((x,\xi),(x',\xi'))$ with 
$\xi(a \partial_t) - \xi'(b \partial_s) \geq 0$ when $a+b>0$. 
This means that $\xi$ must be future directed and $\xi'$ must be past directed.
Note that this theorem is stated for open cones and can therefore not be immediately applied to the kernel $G$. 
However, it can be applied to a subset of the variables as for example shown in \cite{MR1936535}*{Theorem 2.8} for the analytic wavefront set. 
We omit the discussion here and refer for details to \cite{MR1936535}. 
The argument for $Y_-$ is analogous.
\end{proof}

\subsection{Feynman propagators on general globally hyperbolic spacetimes}

Feynman propagators exist in fact on any globally hyperbolic spacetime and various constructions for them are known in the literature. 
The first general proof for scalar wave equations is due to Duistermaat and H\"ormander \cite{MR0388464}. 
It relies on microlocalization and can also be carried out for normally hyperbolic operators acting on vector bundles, see \cite{strohmaierislam}.
The existence of Feynman propagators is essentially equivalent to the existence of Hadamard states, each Hadamard state giving rise to a Feynman propagator. 
They form an important class of states in quantum field theory on curved spacetimes as they allow for renormalization of perturbative theories, see for example \cite{MR1736329}. 
Different types of constructions of such states have been found, for example \cite{MR641893,MR1400751,MR1421547,MR3148100}, and this remains an active research area, see for example \cite{noMRGerard} for a recent review.

For the sake of completeness we give here a direct general argument for the existence of Feynman propagators, similar to the deformation construction of \cite{MR641893}. 
The main observation is that a Feynman propagator defined near a Cauchy hypersurface can be extended uniquely to a Feynman propagator on the entire spacetime.

\begin{prop}\label{prop:FeynmanExtend}
Let $P$ be a normally hyperbolic operator on a globally hyperbolic spacetime $X$.
Let $\Sigma\subset X$ be a Cauchy hypersurface and let $\UU$ be a globally hyperbolic neighborhood of $\Sigma$ such that $\Sigma$ is also a Cauchy hypersurface in $\UU$.
Let $G$ be a Feynman propagator of $P$ on $\UU\times\UU$.
 
Then there is a unique Feynman propagator $\tilde G$ of $P$ on $X\times X$ that extends $G$.
\end{prop}

\begin{proof}
Let $G_\ret$ be the retarded fundamental solution of $P$ on $\UU$ and $\tilde G_\ret$ the one on $X$.
By uniqueness of retarded fundamental solutions, $G_\ret$ is just the restriction of $\tilde G_\ret$.
Since $G'=G-G_\ret$ is a bi-solution near $\Sigma \times \Sigma$ it uniquely extends to a bisolution $\tilde G'$ on $X$.
Then $\tilde G = \tilde G' + \tilde G_\ret$ is the unique extension of $G$. 

It remains to show that $\tilde G$ has wavefront set contained in $\Lambda_\Feyn^+$ over all of $X\times X$.
This follows from propagation of singularities as $\mathrm{WF} \setminus \Delta^*$ is invariant under the geodesic flow and every orbit intersects $T^* (\UU \times \UU)$ because each maximal lightlike geodesic hits $\Sigma$.
\end{proof}
  
\begin{prop}\label{prop:FeynmanExist}
Let $P$ be a normally hyperbolic operator on a globally hyperbolic spacetime $X$. 
Then there exists a Feynman propagator.
\end{prop}

\begin{proof}
By Proposition~\ref{prop:FeynmanExtend}, it suffices to construct a Feynman propagator in a globally hyperbolic neighborhood $\UU$ of a Cauchy hypersurface $\Sigma$. 
The metric can be deformed in the future of $\UU$ to have product structure near another Cauchy hypersurface in such a way that the interpolating manifold is globally hyperbolic. 
To be more precise, there exists another globally hyperbolic spacetime $X'$ which contains a Cauchy hypersurface $\Sigma_1\subset X'$ with a globally hyperbolic neighborhood $\UU_1\subset X'$ and a  Cauchy hypersurface $\Sigma_2\subset X'$ with a globally hyperbolic neighborhood $\UU_2\subset X'$ such that:
\begin{enumerate}[\myicon]
\item
the neighborhood $\UU_1$ is isometric to $\UU$, 
\item
the metric has product structure over $\UU_2$,
\end{enumerate}
see \cite[Thm.~3]{omueller}.

The (isometric copy of) the restriction of $P$ to $\UU$ can be extended to a normally hyperbolic operator $P'$ on $X'$ and then be deformed on $\UU_2$ to be of the form $P'=\frac{\partial^2}{\partial t^2} + \Delta$.
If we are in case~\eqref{FallA}, we can furthermore assume that $\Delta$ is formally selfadjoint with respect to a positive definite Hermitian inner product and bounded from below by 1, say.

Now one constructs a Feynman propagator for $P'$ on $\UU_2$ as in \eqref{LaplaceHadamard} and extends it to a Feynman propagator for $P'$ on $X'$. 
The restriction of this operator to $\UU_1$ then defines a Feynman propagator for $P$ on $\UU$. 
\end{proof}
   
While Feynman propagators are not unique, their singularity structure is. i.e.\ different Feynman propagators differ by smoothing operators. 
In particular, Feynman propagators that are constructed from Cauchy surfaces near which the metric is of product type in the above manner will in general be different. 
This difference, in a way, is the origin of the appearance of a nontrivial index. 
Similarly, the fact that the difference is smoothing reflects the Fredholmness of the problem. 
The singularity structure of Feynman propagators near the diagonal can be described quite explicitly on a symbolic level.
For our purposes we will need only the singularities up to a certain order. 

Suppose that $n=\dim(X)$ is even and that $\UU$ is a sufficiently small neighborhood of the diagonal in $X \times X$. 
We apply Proposition~\ref{prop:ZerlegeG} with $N=2$ and obtain a decomposition
\begin{equation}
G = G^{\loc} + G^{\reg},
\label{eq:DecompositionOfG}
\end{equation}
where
$G^{\reg}$ is $C^2$ near the diagonal and 
\begin{equation}
G^{\loc}|_{\UU} =\sum_{j=0}^{\frac{n+2}{2}} V_j  G^{+,\UU}_{-\frac{n}{2}+1+j}.
\label{eq:ExpansionOfGloc}
\end{equation}
For a definition of the distributions $G^{+,\UU}_\beta$ and the Hadamard coefficients $V_j$ see Appendix~\ref{sec:HadDist}.

\begin{notation}
If $\hat K$ is an operator with continuous Schwartz kernel $K \in C(X \times X; \SS \boxtimes \SS^*)$ then the restriction of $K$ to the diagonal will be denoted by $[\hat K] \in C(X;\SS \otimes \SS^*)$, i.e.\ $[\hat K](x) = K(x,x)$.
We will write $\mathrm{Tr}_x$ for the pointwise trace of an operator $\hat K$ with continuous integral kernel , i.e.
$$
 \mathrm{Tr}_x \left( \hat K\right) = \tr(K(x,x)) =  \tr ([\hat K](x)).
$$
\end{notation}

\section{Feynman propagators for Dirac-type operators} \label{Dirac-Fein}

Now we return to a Dirac-type operator $\dirac$ acting on sections of the bundle $\SS$. 
The operator $\dirac^2$ is normally hyperbolic and also acts on $C_\mathrm{c}^\infty(X;\SS)$.
It therefore has retarded and advanced fundamental solutions $G_{\ret},G_{\adv}\in\DD'(X \times X;\SS \boxtimes \SS^*)$. 

It is easy to see that $\hat D_{\ret/\adv}:=\dirac\circ \hat G_{\ret/\adv} = \hat G_{\ret/\adv} \circ\dirac$ yield the retarded and advanced fundamental solutions for $\dirac$, see \cite{MR3302643}*{Example~3.16}. 

As for normally hyperbolic operators, a distinguished parametrix $D$ for $\dirac$ is called \emph{Feynman parametrix} if $\WF'(D) \subset \Lambda_\Feyn^+$.
The operator $\hat D$ is called a \emph{Feynman propagator} if $D$ is a Feynman parametrix and it in addition satisfies $\dirac \hat D =   \dirac \hat D =\id$ on $C_\mathrm{c}^\infty(X; \SS)$.

\subsection{APS boundary conditions}

If $\Sigma$ is compact, we do not assume that the Dirac operator $\DS$ is selfadjoint with respect to some positive definite inner product.
Thus, its eigenvalues of $\DS$ may fail to be real-valued.
In order to pose APS boundary conditions, one needs to make a choice of which eigenvalues are considered positive and which are considered negative. 
We use the projectors introduced in \eqref{eq:p>}--\eqref{eq:p=} to achieve this.

Let $\Sigma_-, \Sigma_+ \subset X$ be two disjoint smooth spacelike Cauchy hypersurfaces of $X$ with $\Sigma_-$ lying in the past of $\Sigma_+$.
Let $M=\J^+(\Sigma_-)\cap \J^-(\Sigma_+)$ be the spacetime region in between these two Cauchy hypersurfaces.
Then $M$ is a manifold with boundary $\partial M = \Sigma_- \cup \Sigma_+$.
We introduce the Atiyah-Patodi-Singer spaces as
\begin{align*}
C^\infty_{\APS}( M;\SS)
&:=
\{ u\in C^\infty( M;\SS) \mid p_{\geq}(\D_{\Sigma_-})( u|_{\Sigma_-}) = 0 = p_{<}(\D_{\Sigma_+})( u|_{\Sigma_+})\} .
\end{align*}

\subsection{The Feynman propagator in the product case}

Suppose that $X = \R \times \Sigma$ and the Dirac operator $\dirac=\rmi\nS (\tfrac{\partial}{\partial t}+\rmi\DS)$ has product structure over $X$.
Then the normally hyperbolic operator $\dirac^2$ also has product structure over $X$.
As we have seen, a Feynman propagator $G$ for this operator can then be explicitly given by \eqref{prodfeyn}. 
Similarly to the advanced and retarded fundamental solutions, we get a Feynman propagator for $\dirac$ by setting
\begin{alignat*}{2}
 \hat D &:= \dirac \hat G &&= \hat G \dirac .
\end{alignat*}
More explicitly, we have
\begin{theorem}
Assume \eqref{FallA} or \eqref{FallB} as stated on page~\pageref{FallA}.
Suppose furthermore that $\dirac$ has product structure over $X = \R \times \Sigma$.
Then the family of operators
\begin{equation}
k_D(t) := 
\rmi\big(\chi_{[0,\infty)}(t)\, p_\ge - \chi_{(-\infty,0]}(t)\, p_<\big)\e^{-\rmi t \DS} \nS
\label{prodfeyndirac}\end{equation}
defines a Feynman propagator $D$ for $\dirac$ via
\begin{gather}
 (\hat D u)(t,\cdot) = \int_\R k_D(t-s) u(s,\cdot) ds .
 \label{DiracFeyn}
 \end{gather}
\end{theorem}

\begin{proof}
From \eqref{eq:DiracProduct2} we have
$$
\dirac= \rmi\bigg(\frac{\partial}{\partial t} -\rmi\DS\bigg)\nS .
$$
From Theorem~\ref{thm:FeynmanProductWave} we know
$$
(\hat D u)(t,\cdot) = \int_\R (k_1(t-s)+k_2(t-s)) u(s,\cdot) ds 
$$
where 
\begin{align*}
k_1(t) 
&= 
-\frac{1}{2}\bigg(\frac{\partial}{\partial t} -\rmi\DS\bigg)\nS \e^{-\rmi | t | \Delta_\theta^{\nicefrac{1}{2}} } \Delta_\theta^{-\nicefrac{1}{2}} , \\
k_2(t) 
&= 
\rmi\chi_{[0,\infty)} (t)\bigg(\frac{\partial}{\partial t} -\rmi\DS\bigg)\nS A^{-\nicefrac{1}{2}} \sin(tA^{\nicefrac{1}{2}}) \, p_0 .
\end{align*}
Recall that $\Delta=\DS^2$ and that $A$ is the restriction of $\Delta$ to its generalized kernel.
Since $\nS$ anticommutes with $\DS$ it commutes with $\Delta$, $A$, all powers $\Delta_\theta^s$, and $p_0$.
Thus,
\begin{align*}
k_1(t) 
&= 
-\frac{1}{2}\bigg(\frac{\partial}{\partial t} -\rmi\DS\bigg) \e^{-\rmi | t | \Delta_\theta^{\nicefrac{1}{2}} } \Delta_\theta^{-\nicefrac{1}{2}}\nS , \\
k_2(t) 
&= 
\rmi\chi_{[0,\infty)} (t)\bigg(\frac{\partial}{\partial t} -\rmi\DS\bigg) A^{-\nicefrac{1}{2}} \sin(tA^{\nicefrac{1}{2}}) \, p_0\nS .
\end{align*}
We compute, using Proposition~\ref{prop:powers},
\begin{align}
k_1(t)
&=
-\tfrac{1}{2}\Big(-\rmi\mathrm{sign}(t)\Delta_\theta^{\nicefrac{1}{2}} -\rmi \DS \Big)\e^{-\rmi | t | \Delta_\theta^{\nicefrac{1}{2}} }\Delta_\theta^{-\nicefrac{1}{2}}\nS \notag\\
&=
\tfrac{\rmi}{2}\Big(\mathrm{sign}(t)(1-p_0) + \DS \Delta_\theta^{-\nicefrac{1}{2}}\Big)\e^{-\rmi | t | \Delta_\theta^{\nicefrac{1}{2}} }\nS \notag\\
&=
\rmi\Big(\chi_{[0,\infty)}(t)\, p_> - \chi_{(-\infty,0]}(t)\, p_<\Big)\e^{-\rmi | t | \Delta_\theta^{\nicefrac{1}{2}} }\nS .\notag
\end{align}
For $t\ge0$ and on $\rg(p_>)$ we have, by Proposition~\ref{prop:projectors}~\eqref{projectors5}, $\Delta_\theta^{\nicefrac{1}{2}}=\DS$ and hence $\chi_{[0,\infty)}(t)\, p_>\, \e^{-\rmi | t | \Delta_\theta^{\nicefrac{1}{2}} }=\chi_{[0,\infty)}(t)\, p_>\, \e^{-\rmi t \DS}$.
Similarly, for $t\le0$ and on $\rg(p_<)$ we get $\Delta_\theta^{\nicefrac{1}{2}}=-\DS$ and hence $\chi_{(-\infty,0]}(t)\, p_<\, \e^{-\rmi | t | \Delta_\theta^{\nicefrac{1}{2}} }=\chi_{(-\infty,0]}(t)\, p_<\, \e^{-\rmi t \DS}$.
Thus,
\begin{equation}
k_1(t) = \rmi\Big(\chi_{[0,\infty)}(t)\, p_> - \chi_{(-\infty,0]}(t)\, p_<\Big)\e^{-\rmi t \DS}\,\nS .
\label{eq:k1}
\end{equation}
Furthermore, we have on $\rg(p_0)$
\begin{align}
k_2(t) 
&= 
\rmi\chi_{[0,\infty)} (t)\big(\cos(tA^{\nicefrac{1}{2}}) -\rmi A^{-\nicefrac{1}{2}} \sin(tA^{\nicefrac{1}{2}})\DS \big) p_0\,\nS \notag\\
&=
\rmi\chi_{[0,\infty)} (t)\big( \cos(t\DS) - \rmi \sin(t\DS) \big) p_0\,\nS \notag\\
&=
\rmi\chi_{[0,\infty)} (t)\, \e^{-\rmi t\DS} \, p_0\,\nS .
\label{eq:k2}
\end{align}
Adding \eqref{eq:k1} and \eqref{eq:k2} yields \eqref{prodfeyndirac}.
\end{proof}

\begin{remark}
For the advanced and retarded fundamental solutions of $\dirac$ we have similar but simpler formulas.
Namely, 
$$
(\hat G_{\adv/\ret}^\dirac u)(t,\cdot) = \int_R k_{\adv/\ret}(t-s)u(s,\cdot)ds
$$
where
\begin{align*}
k_\adv(t) &= -\rmi \chi_{(-\infty,0]}(t)\e^{-\rmi t \DS} \nS 
\quad\text{ and}\\
k_\ret(t) &=\rmi \chi_{[0,\infty)}(t)\e^{-\rmi t \DS} \nS.
\end{align*}
For later reference we note that this implies
\begin{equation}
\big(\hat D - \tfrac12 (\hat G_\ret^\dirac + \hat G_\adv^\dirac)\big)u(t,\cdot)
=
\tfrac\rmi2 \int_\R (p_\ge - p_<)\e^{-\rmi (t-s) \DS} \nS u(s,\cdot) ds .
\label{eq:FeynmanRegularized}
\end{equation}
\end{remark}

In the product case with compact Cauchy hypersurface, the nonlocal regular part $[G^{\reg}]$ of the Feynman propagator for $\dirac^2$ is related to the $\eta$-invariant $\eta(\DS)$ of the Riemannian Dirac operator $\DS$, as we will see. 

\subsection{\texorpdfstring{The $\pmb{\eta}$-invariant of a Dirac-type operator on a compact Riemannian manifold}{The eta-invariant of a Dirac-type operator on a compact Riemannian manifold}}
Let $\Sigma$ be a compact Riemannian manifold and $\DS$ a Dirac-type operator on $\Sigma$.
Before we treat the general case let us review the classical $\eta$-invariant for selfadjoint Dirac operators.

\subsubsection{\texorpdfstring{The $\eta$-invariant of a selfadjoint Dirac operator in the sense of Gromov and Lawson}{The eta-invariant of a selfadjoint Dirac-type operator}}
Suppose that $\DS$ is selfadjoint with respect to a positive definite inner product.
Then $\DS$  has real and discrete spectrum. 
The $\eta$-invariant can be defined as the value at $s=0$ of the analytic continuation of the $\eta$-function which is defined for $\mathrm{Re}(s)> n$ by
$$
 \eta(s) = \sum_{\mu \not= 0} \mathrm{sign}(\mu) |\mu|^{-s}.
$$
Here the summation is taken over all nonzero eigenvalues $\mu$ of $\DS$, repeated according to their multiplicity.
If $|\DS|^{s}$ is defined by functional calculus as $\chi_s(|\DS|)$, where 
$$
 \chi_s(x) = 
 \begin{cases} 
 x^s, & x >0, \\ 
 0, & x \leq 0,
 \end{cases}
$$ 
then this can also be written as $\eta(s) = \Tr (\mathrm{sign}(\DS) |\DS|^{-s} )$. 
We also have the corresponding local $\eta$-function, that is, the meromorphic continuation of the function defined by 
$$
\eta_y(s) = \Tr_y (\mathrm{sign}(\DS) |\DS|^{-s} ), \qquad \mathrm{Re}(s)> n .
$$
Obviously, the local $\eta$-function integrates to the $\eta$-invariant,
$$
\eta(s) = \int_\Sigma \eta_y(s)\, dA(y).
$$
Using the Mellin transform one easily verifies the relations
\begin{align*}
\eta_y(s) 
&= 
\Tr_y ( (p_> - p_<) |\DS|^{-s} ) \\ 
&=  
\Tr_y \big( \DS |\DS|^{-(s+1)}\big) \\ 
&= 
\Tr_y \big( \DS (\DS^2)^{-\frac{s+1}{2}}\big) \\
&= 
\frac{1}{\Gamma(\frac{s}{2})} \int_0^\infty t^{\frac{s}{2}-1} \Tr_y \big( (p_> - p_<) \e^{-t \DS^2}\big) \, dt\\ 
&=  
\frac{1}{\Gamma(\frac{s+1}{2})} \int_0^\infty t^{\frac{s-1}{2}} \Tr_y \big(\DS \e^{-t \DS^2}\big) \, dt .
\end{align*}
It is well known that the asymptotic expansion of the local trace $\Tr_y (\DS \e^{-t \DS^2})$ as $t \searrow 0$ gives a meromorphic continuation of the local $\eta$-function. 

In the case of a Dirac operator in the sense of Gromov and Lawson, one has $\Tr_y (\DS \e^{-t \DS^2}) = \O(t^{\nicefrac{1}{2}})$ uniformly in $y$ as $t \searrow 0$, see \cite{bismut1986analysis}*{Thm.~2.4}.
In particular, $\eta_y$ is regular at $s=0$. Comparison of expansion coefficients then shows that
$$
 \Tr_y \big( (p_> - p_<) \e^{-t \DS^2} \big) = \O(1),
$$
as $t \searrow 0$ uniformly in $y$.  
The local $\eta$-invariant $\eta_y(\DS)$ is defined for $y \in \Sigma$ to be the value $\eta_y(0)$. 
The above shows that in fact
\begin{gather*} 
 \eta_y(\DS) = \lim_{t \searrow 0}  \Tr_y \big( (p_> - p_<) \e^{-t \DS^2} \big).
\end{gather*}

Similarly, let $h_y(\DS) = \Tr_y p_0(\DS)$. 
This is the trace of the restriction to the diagonal of the integral kernel of $p_0(\DS)$.
Note that $h_y(\DS)$ integrates to $h(\DS) := \dim\ker(\DS)$.

\subsubsection{\texorpdfstring{The $\eta$-invariant of a general Dirac-type operator}{The eta-invariant of a general Dirac-type operator}}

Now we no longer assume that $\DS$ is selfadjoint with respect to a positive definite inner product.
Then we are in Case~\ref{FallB} on page~\pageref{FallB}.
This means that the spectrum of $\DS$ may be nonreal and there may be nontrivial Jordan blocks in the generalized eigenspaces.
In order to measure spectral asymmetry, one needs to specify what ``positive'' and ``negative'' eigenvalues are. 
This is achieved by considering the operators $\Delta_\theta^{-\nicefrac{1}{2}}$, $p_>$, $p_<$, and $p_0$ for $\Delta = \D^2_\Sigma$ as defined in Section~\ref{sec:LaplaceDirac}. 
By Propositions~\ref{prop:powers}~(\ref{power8strich}') and \ref{prop:projectors}~\eqref{projectors7}, $(p_>-p_<) \Delta^{-\nicefrac{s}{2}}_\theta$ is a classical pseudodifferential operator of order $-\Re(s)$. 
In particular, $(p_>-p_<) \Delta^{-\nicefrac{s}{2}}_\theta$ is a trace-class operator for $\Re(s) > n$. 
The function
$$
 \eta(s) := \Tr\big( (p_> - p_<) \Delta_\theta^{-\nicefrac{s}{2}} \big)
$$
is holomorphic and extends to a meromorphic function on the whole complex plane.
Zero is a pole of order at most one. 
The residue at zero is proportional to the Wodzicki residue of $p_> - p_<$, see for example \cite{MR1714485} and references therein.
By a deep result of Gilkey \cite{MR624667} and Wodzicki \cite{MR728144}, pseudodifferential projections have vanishing Wodzicki residues, see also \cite{MR3080490} for an algebraic proof.
Therefore, $\eta(s)$ is regular at $s=0$, and we can define the $\eta$-invariant of $\DS$ by
$$
\eta(\DS) := \eta(0).
$$
By the general theory of residues of  pseudodifferential operators, we have for any zero-order (polyhomogeneous) pseudodifferential operator
$A$ that
$$
 \Tr \left( A  e^{-t \Delta_\theta} \right) \sim \sum_{j=0}^{n-2}  a_{j} t^{-\frac{n-1}{2}+\frac{j}{2}} + b + c \log(t) +\O(t^{\nicefrac12} \log t) \quad \textrm{ as } t \searrow 0,
$$
where $c$ is proportional to the Wodzicki residue of $A$, see for example \cite{MR1714485}*{(1.14) and (1.16)}.
Hence,
\begin{equation} \label{good-equationnr1}
\Tr \left( (p_> - p_<)  e^{-t \Delta_\theta} \right) 
\sim 
\sum_{j=0}^{n-2}  a_{j} t^{-\frac{n-1}{2}+\frac{j}{2}} + \eta(\DS) +\O(t^{\nicefrac12} \log t) \quad \textrm{ as } t \searrow 0.
\end{equation}
There is also a corresponding local expansion
\begin{equation}
\Tr_y \left( (p_> - p_<)  e^{-t \Delta_\theta} \right) 
\sim 
\sum_{j=0}^{n-2}  a_{y,j} t^{-\frac{n-1}{2}+\frac{j}{2}} + b_y + c_y \log(t) +\O(t^{\nicefrac12} \log t) 
\quad \textrm{ as } t \searrow 0
\label{good-equationnr1pointwise}
\end{equation}
and therefore
$$
 a_j = \int_\Sigma a_{y,j} dy, \quad \eta(\DS) = b = \int_\Sigma b_y dy, \quad \int_\Sigma c_y dy =0.
$$
While in the case of selfadjoint Dirac operator in the sense of Gromov and Lawson we have $c_y=0$ and $a_{y,j}=0$ for $0 \leq j\leq n-2$ this is no longer true for general operators of Dirac type.
It was shown in \cite{MR3483832}*{Thm.~1.3} that for a selfadjoint operator of Dirac type, $a_{y,0}=0$ if and only if it is a Dirac operator in the sense of Gromov and Lawson. One can use the results by
Branson and Gilkey \cite{MR1174158} to see that $c_y$ is in general nonzero for Dirac-type operators. 

\subsection{\texorpdfstring{Feynman propagator and $\pmb{\eta}$-invariant}{Feynman propagator and eta-invariant}}

We return to Lorentzian manifolds and will derive a relation between the Feynman propagator of the square of the Dirac operator and the $\eta$-invariant of the induced Riemannian Dirac operator on a Cauchy hypersurface.

\subsubsection{The product case}

We start by considering the product case.
Let $X = \R \times \Sigma$ with product metric $-dt^2+g_\Sigma$ where $\Sigma$ is compact with Riemannian metric $g_\Sigma$.
Let $\dirac =
\begin{pmatrix}
0 & \dirac_R \\
\dirac_L & 0
\end{pmatrix}$
be an odd Dirac-type operator of product type over $X$.
Then $\dirac_L = \rmi\nS (\frac{\partial}{\partial t} +\rmi\DS )$ where $\DS$ is the induced Dirac-type operator on $\Sigma$.
Denote by $h(\DS)$ the dimension of the generalized kernel of $\DS$.

Let $\hat G_L$ be the Feynman propagator of $\dirac_R\dirac_L$ given in Theorem~\ref{thm:FeynmanProductWave}.
We split $\hat G_L$ into a local and a regular part, $\hat G_L=\hat G_L^\loc + \hat G_L^\reg$, as in \eqref{eq:DecompositionOfG}.
Then the regular part $\hat G_L^\reg$ relates to the $\eta$-invariant of $\DS$ as follows:

\begin{prop}\label{prop:etah} 
Let $n\ge2$ be even.
For any $t\in\R$
$$
 \rmi\int_\Sigma  \tr \left( [ \dirac_L \hat G_{L}^{\reg}](t,y) \nS \right)  dy 
 = 
 \tfrac{1}{2} \left( \eta(\DS) + h(\DS) \right)
$$
holds.
If $\DS$ is a Dirac operator in the sense of Gromov and Lawson and selfadjoint with respect to a Hermitian inner product, then this formula holds pointwise, i.e.\ for any $(t,y)\in X=\R\times \Sigma$ we have
$$
 \rmi\tr \left( [ \dirac_L \hat G_{L}^{\reg}](t,y) \nS \right) = \tfrac{1}{2} \left( \eta_y+ h_y \right).
$$
\end{prop}

\begin{proof}
It is no loss of generality to prove the asserted formulas at $t=0$.
Let $G_{L,\ret}$ and $G_{L,\adv}$ be the retarded and the advanced fundamental solution of $\dirac_R\dirac_L$, respectively.
Since $G_L$, $G_{L,\ret}$ and $G_{L,\adv}$ are fundamental solutions, $G_L-\frac12(G_{L,\ret}+G_{L,\adv})$ is a bi-solution.
Its wave front set satisfies 
$$
\WF'(G_L-\tfrac12(G_{L,\ret}+G_{L,\adv})) \subset \Lambda^+_\Feyn \cup \Lambda_\ret \cup \Lambda_\adv .
$$
Away from the diagonal in $X\times X$ this shows that $\WF'(G_L-\frac12(G_{L,\ret}+G_{L,\adv}))$ is contained in the set of pairs of lightlike covectors.
By propagation of singularities, this also holds over the diagonal.
Thus, for fixed $y \in \Sigma$, the map 
$$
\alpha_y\colon \R\to X\times X, \quad
\alpha_y(t)= (t,y,0,y),
$$
is transversal to this wave front set and hence also to the wave front set of the operator $\dirac_L\circ (\hat G_{L}-\frac12(\hat G_{L,\ret}+\hat G_{L,\adv}))\circ\nS$.
We may therefore pull back its Schwartz kernel and obtain an $\mathrm{End}(\SS_{L,y})$-valued distribution on $\R$.
We apply the pointwise trace to obtain a well-defined complex-valued distribution $\phi_y \in \mathcal{D}'(\R)$ given by
$$
\phi_y := \tr\big(\alpha_y^*(\dirac_{L,(1)}(G_{L}-\tfrac12(G_{L,\ret}+G_{L,\adv}))\circ\nS)\big) .
$$
Here $\dirac_{L,(1)}$ means the application of $\dirac_L$ to the first argument of a kernel.
Since $G_L^\reg$ is $C^2$ on a neighborhood of the diagonal, we also have $\phi_y^\reg\in C^1(-\eps,\eps)$ given by
$$
\phi^\reg_y(t) 
:= 
\tr (\dirac_{L,(1)} G^\reg_{L}(t,y,0,y) \nS )
$$
for some $\eps>0$.
Our aim is to evaluate $\phi_y^\reg(0)$ and its integral over $y\in\Sigma$.

In order to understand the singularity structure of $\phi_y$ near $0$, we consider the distribution $\phi^\loc_y := \phi_y-\phi^\reg_y \in \mathcal{D}'(-\eps,\eps)$.
By Lemma~\ref{lem:logpoly} we have
\begin{equation}
\phi_y^\loc(t) = -\rmi\left\{\sum_{\ell=0}^{n-1} \tilde a_{y,\ell} t^{-\ell} + \tilde c_y \log|t| + \tilde r_y(t)\right\}
\label{eq:SingStructure}
\end{equation}
for $|t|<\eps$ and $t\neq0$.
The remainder term $\tilde r_y$ is continuous at $t=0$ and satisfies $\tilde r_y(t)=\O(t\log|t|)$ as $t\to0$, uniformly in $y$.
Hence, $\phi^\reg_y(0) - \rmi\tilde a_{y,0}$ equals the constant term in the singularity expansion of $\phi_y(t)$ at $t=0$.

By \eqref{eq:FeynmanRegularized} we have explicitly
\begin{align*}
\phi_y(t) 
&=
-\tfrac{\rmi}{2}\Tr_y\left( (p_\geq - p_<) \; \e^{-\rmi t  \DS}  \right) \\
&=
-\tfrac{\rmi}{2}\Tr_y\left( p_> \; \e^{-\rmi t  \Delta_\theta^{\nicefrac{1}{2}}}  - p_< \; \e^{\rmi t    \Delta_\theta^{\nicefrac{1}{2}}}  + \e^{-\rmi t  \DS}p_0 \right) \\
&=
-\tfrac{\rmi}{2}\Tr_y\left( p_> \; \e^{-\rmi t  \Delta_\theta^{\nicefrac{1}{2}}}  - p_< \; \e^{\rmi t    \Delta_\theta^{\nicefrac{1}{2}}}\right) -\tfrac{\rmi}{2} h_y + t \, q_y(t),
\end{align*}
where $q_y(t)$ is a polynomial in $t$ with coefficients that are smooth functions of $y$.
We have used here that $\DS$ is nilpotent on its generalized kernel.

For the moment let $Q$ be any zero-order pseudodifferential operator.
Since $\e^{-\rmi t  \Delta_\theta^{\nicefrac{1}{2}}}$ is a strongly continuous group of bounded operators, we have $\| Q \; \e^{-\rmi t  \Delta_\theta^{\nicefrac{1}{2}}} \|_{L^2 \to L^2} = O(\e^{c |t|})$ for some $c>0$.
Given any $s \in \R$ we then have $\| Q \; \e^{-\rmi t  \Delta_\theta^{\nicefrac{1}{2}}} \|_{H^s \to H^s} = O(\e^{c |t|})$ for some $c>0$.
For sufficiently large $N$, the operator $Q \Delta_\theta^{-N} \e^{-\rmi t  \Delta_\theta^{\nicefrac{1}{2}}}$ has a continuous kernel and the Sobolev embedding theorem then implies the pointwise bound
 $$
  \Tr_y\left( Q \; \Delta_\theta^{-N}  \e^{-\rmi t  \Delta_\theta^{\nicefrac{1}{2}}} \right) = O(\e^{c |t|})
 $$
for some $c>0$.
It follows that $\phi_y$ is the distributional derivative of a continuous function that is bounded by an exponential.
This implies that $\phi_y$ extends as a linear functional continuously to a larger test function space.
In particular, the pairing with a Gaussian test function is well defined.
The pairing of $\phi_y$ with the family of Gaussian test functions $g_s(t)=\frac{\e^{-\nicefrac{t^2}{4s}}}{\sqrt{4\pi s}} $ defines a function $\psi_y\colon\R^+\to\C$.
Thus, formally with the integral interpreted as a dual pairing,
$$
  \psi_y(s)=\int_{-\infty}^{\infty} \phi_y(t) \frac{\e^{-\nicefrac{t^2}{4s}}}{\sqrt{4 \pi s}} \, dt.
$$
The expansion \eqref{eq:SingStructure} is valid for $t \not=0$ and therefore determines the distribution  $\phi_y$ up to a distribution with support at the origin.
We therefore get 
\begin{equation}
\psi_y(s) 
= 
\sum_{j=1}^{\frac{n-2}{2}} \rmi (-1)^{j+1}  \frac{j!\tilde a_{y,2j}}{(2j)!}s^{-j} +p_y(s^{-1})s^{-\frac{1}{2}} - \tfrac{\rmi\tilde c_y}{2}(-\gamma + \log(s))  + \phi_y^\reg(0) - \rmi\tilde a_{y,0}  + \O(s^{\nicefrac12} \log s),\label{eq:psiyasymptotics}
\end{equation}
$\textrm{ as } s \searrow 0$
where $\gamma$ is the Euler-Mascheroni constant and $p_y$ is a polynomial.

Since the Fourier transform of $\frac{\e^{-\nicefrac{t^2}{4s}}}{\sqrt{4 \pi s}}$ in the $t$-variable equals $\e^{-\xi^2 s}$, the functional calculus for sectorial operators yields
\begin{equation}
\psi_y(s) 
= 
-\tfrac{\rmi}{2} \Tr_y\left( (p_> - p_<) \e^{- s \Delta_\theta}    \right) - \tfrac{\rmi}{2} h_y + s \,\tilde q_y(s),
\label{eq:psiformel}
\end{equation}
where $\tilde q_y(s)$ is a polynomial in $s$ with coefficients depending smoothly on $y$.
In case $\DS$ is a selfadjoint Dirac operator in the sense of Gromov and Lawson, \eqref{eq:psiyasymptotics} with \eqref{good-equationnr1pointwise} shows that $c_y=0$.
Since $\tilde a_{y,0}$ is a multiple of $\tilde c_y$ (see Lemma~\ref{lem:logpoly}) we also have $a_{y,0} = 0$. 
Thus, the constant term in the expansion of $\psi_y(s)$ equals $\phi_y^\reg(0)$ by \eqref{eq:psiyasymptotics} and $-\frac{\rmi}{2}(\eta_y + h_y)$ by \eqref{good-equationnr1pointwise} and \eqref{eq:psiformel}.
Hence, 
$$
\tr \left( [ \dirac_L \hat G_{L}^{\reg}](0,y) \nS \right) 
=
\phi_y^\reg(0)
=
-\tfrac{\rmi}{2} \left( \eta_y+ h_y \right) .
$$
If $\DS$ is a general Dirac-type operator, the same reasoning still works after integrating over $y\in\Sigma$.
\end{proof}

\subsubsection{\texorpdfstring{The local $\xi$- and $\eta$-invariants of a Feynman propagator}{The local xi- and eta-invariants of a Feynman propagator}}

Proposition~\ref{prop:etah} and the Atiyah-Patodi-Singer index theorem show that it is more natural to consider the sum
$$
 \Xi_y(\DS):=\tfrac{1}{2} \left( \eta_y + h_y \right)
$$
rather than the $\eta$-invariant itself.
This sum is well defined if $\Sigma$ is compact, and it is sometimes called the \emph{local $\xi$-invariant}. 
In order to avoid confusion with vectors and covectors that are also denoted by $\xi$ in this article, we will denote the $\xi$-invariant by the boldface version of the letter $\xi$. 

We will now give a more general definition of the local $\xi$-invariant associated with any Feynman propagator even when we drop the assumption that $\Sigma$ is compact or that $\DS$ is selfadjoint. 

\begin{definition} \label{better}
The local $\xi$-invariant $\Xi_y(G_L)$ of a Feynman propagator $G_L$ is defined as
$$
 \Xi_y(G_L)=\tr \left( [ \dirac_L G_{L}^{\reg}](y) \nS \right).
$$
If the local $\xi$-invariant is integrable then the $\xi$-invariant $\Xi(G_L)$  is defined as
$$
 \Xi(G_L) = \int_\Sigma \Xi_y(G_L) \dA(y).
$$
\end{definition}

Proposition~\ref{prop:etah} guarantees that in case $\Sigma$ is compact and the Feynman propagator is given by  \eqref{prodfeyndirac} and \eqref{DiracFeyn}, then this definition of the global $\xi$-invariant coincides with the usual one.
If moreover $\D_\Sigma$ is a Dirac operator in the sense of Gromov and Lawson, this definition gives the local $\xi$-invariant as defined before. 
Note that our Definition~\ref{better} makes sense also for general Dirac-type operators or on noncompact manifolds if there is an essential spectral gap. 
We will show in the next section that it is this local invariant that appears in the index theorem.

\section{The index theorem} \label{index:Section}

From now on suppose that $X$ is an $n$-dimensional globally hyperbolic manifold with $n$ even and that $\dirac$ is an odd Dirac-type operator over $X$.
Let $\Sigma\subset X$ be a Cauchy hypersurface.

\begin{lemma}\label{lem:DiracGTausch}
Assume that $\dirac$ has product structure near $\Sigma$.
Let $\hat G_R$ and $\hat G_L$ be the Feynman propagators of $\dirac_L\dirac_R$ and $\dirac_R\dirac_L$, respectively, given by Theorem~\ref{thm:FeynmanProductWave} near $\Sigma$ and extended to all of $X$ by Proposition~\ref{prop:FeynmanExtend}.
Then we have, globally on $X$,
\begin{align}
\hat G_{R} \dirac_L &= \dirac_L \hat G_{L},
\label{eq:DiracGTausch1}\\
\hat G_{L} \dirac_R &= \dirac_R \hat G_{R}.
\label{eq:DiracGTausch2}
\end{align}
\end{lemma}

\begin{proof}
These relations hold near $\Sigma$ as one can see from the explicit formulas \eqref{prodfeyn} and \eqref{LaplaceHadamard}.
Now the operator $\hat K := \hat G_{R} \dirac_L - \dirac_L \hat G_{L}$ satisfies $\dirac_L\dirac_R\hat K=0$ and $\hat K\dirac_R\dirac_L=0$.
Hence, its Schwartz kernel $K$ is a bi-solution.
Since $K$ vanishes near $\Sigma\times\Sigma$, it vanishes on all of $X\times X$.
Therefore, $\hat K=0$ which shows \eqref{eq:DiracGTausch1}.
Relation \eqref{eq:DiracGTausch2} is proved in the same way.
\end{proof}

Let $x\in X$.
We construct a linear map $\theta:\mathrm{End}(\SS_{L,x}) \to \mathrm{End}(\SS_{R,x})$ as follows:
choose a basis $b_1,\ldots,b_n$ of $T_xX$ and put $g_{k\ell} := g(b_k,b_\ell)$.
Denote the inverse matrix by $g^{k\ell}$, as usual.
Now put 
$$
\theta(A) := \tfrac{1}{n}\sum_{k,\ell=1}^n g^{k\ell} \,\slashed b_{k} \circ A \circ \slashed b_\ell 
$$
for $A\in \mathrm{End}(\SS_{L,x})$.
The definition does not depend on the choice of basis.
This map preserves the trace because
\begin{align*}
\tr(\theta(A))
&=
\tfrac{1}{n}\sum_{k,\ell=1}^n g^{k\ell} \,\tr(\slashed b_{k} \circ A \circ \slashed b_\ell )
=
\tfrac{1}{n}\sum_{k,\ell=1}^n g^{k\ell} \,\tr(\slashed b_\ell\circ \slashed b_{k} \circ A ) \\
&=
\tfrac{1}{2n}\sum_{k,\ell=1}^n g^{k\ell} \,\tr((\slashed b_\ell\circ \slashed b_{k} + \slashed b_k\circ \slashed b_{\ell})\circ A )
=
\tfrac{1}{n}\sum_{k,\ell=1}^n g^{k\ell} \,\tr(g_{k\ell} A )
=
\tr(A).
\end{align*}

\begin{prop} \label{niceprop}
Let $V_{R,k}$ be the Hadamard coefficients of $\dirac_L\dirac_R$ and $V_{L,k}$ those of $\dirac_R\dirac_L$.
Then we have
$$
[ \dirac_L \hat G^\loc_{L} \dirac_R - \dirac_L\dirac_R \hat G^\loc_{R} ] 
= 
-\tfrac{\rmi}{(4 \pi)^{\nicefrac{n}{2}} (\frac{n}{2})!} \big(  V_{R,\frac{n}{2}}  -\theta(V_{L,\frac{n}{2}})\big).
$$
In particular,
$$
\tr [ \dirac_L \hat G^\loc_{L} \dirac_R - \dirac_L\dirac_R \hat G^\loc_{R} ] 
=
-\tfrac{\rmi}{(4 \pi)^{\nicefrac{n}{2}} (\frac{n}{2})!} \big(\tr (V_{R,\frac{n}{2}}) - \tr ( V_{L,\frac{n}{2}}) \big).
$$
\end{prop}

\begin{proof}
By Lemma~\ref{lem:DiracGTausch} we have $\hat G^\loc_{L} \dirac_R - \dirac_R \hat G^\loc_{R} = -\hat G^\reg_{L} \dirac_R + \dirac_R \hat G^\reg_{R}$, hence the kernel of the difference is $C^1$ on a neighborhood $\UU$ of the diagonal.
Therefore the kernel of $\dirac_L\hat G^\loc_{L} \dirac_R - \dirac_L\dirac_R \hat G^\loc_{R}$ is continuous on $\UU$ so that $[\dirac_L\hat G^\loc_{L} \dirac_R - \dirac_L\dirac_R \hat G^\loc_{R}]$ is a well-defined continuous section of the bundle $\mathrm{End}(\SS_R)\to X$.

The strategy of the proof is now based on a comparison of expansion coefficients. 
Since the Hadamard expansion does not have an external parameter, the Hadamard coefficients are not a priori determined by the singularity structure of the propagator.
This problem will be circumvented below by taking the product with the real line, thus increasing the dimension by one and providing the external parameter.

We will freely use the notation and the results of the appendix, in particular we will be using extensively the families of distributions  $G^{+,\UU}_\beta$ and $F^{+,\UU}_\beta$, as well as the forward and backward Riesz distributions $R^{\pm,\UU}_\beta$.
Note that the superscript $+$ stands for the forward light cone, however in the case of the distributions $G^{+,\UU}_\beta$ and $F^{+,\UU}_\beta$ this refers to the propagation direction of the wavefront set.

Near the diagonal we have, by Proposition~\ref{prop:ZerlegeG},
$$
G^\loc_{R}
=
\sum\limits_{k=0}^{\frac{n+2}{2}} V_{R,k} \cdot G^{+,\UU}_{1-\frac{n}2+k} .
$$
Recall here that for a differential operator $P$ on $X$ acting on a kernel on $X \times X$ we use $P_{(1)}$ and $P_{(2)}$ to refer to the variable it acts on.
We compute the Schwartz kernel of $\dirac_R\circ \hat G^\loc_{R}$, using \eqref{eq:DiracLeibnitz}, Proposition~\ref{FGOmega}, and Lemma~\ref{lem:QH}:
\begin{align*}
\sum\limits_{k=0}^{\frac{n+2}{2}} &\dirac_{R,(1)} \big(V_{R,k} \cdot G^{+,\UU}_{1-\frac{n}2+k}\big)
=
\sum\limits_{k=0}^{\frac{n+2}{2}} \big(\dirac_{R,(1)}V_{R,k} \cdot G^{+,\UU}_{1-\frac{n}2+k} -\rmi \grades G^{+,\UU}_{1-\frac{n}2+k}\circ V_{R,k}  \big) \\
&=
-\tfrac{\rmi}{4} \grades \Gamma \circ V_{R,0} \cdot H^{+,\UU}_{-\frac{n}2}
+ \sum\limits_{k=0}^{\frac{n}{2}} \big(\dirac_{R,(1)}V_{R,k}  -\tfrac{\rmi}{4(k+1)} \grades \Gamma\circ V_{R,k+1}  \big)\cdot G^{+,\UU}_{1-\frac{n}2+k} \\
&\quad + 
\dirac_{R,(1)}V_{R,\tfrac{n+2}{2}} \cdot G^{+,\UU}_{2} 
+ \tfrac{1}{n^2\pi} \grades\Gamma\circ V_{R,\frac{n}{2}}\cdot F^{+,\UU}_{0}
+ \tfrac{1}{(n+2)^2\pi} \grades\Gamma\circ V_{R,\frac{n+2}{2}}\cdot F^{+,\UU}_{1} .
\end{align*}
Similarly, the Schwartz kernel of $\hat G^\loc_{L}\circ \dirac_R$ is given by 
\begin{align*}
\sum\limits_{k=0}^{\frac{n+2}{2}} &\dirac_{R,(2)}^* \big(V_{L,k} \cdot G^{+,\UU}_{1-\frac{n}2+k}\big)
\\
&=
-\tfrac{\rmi}{4} V_{L,0}\circ \gradzs \Gamma  \cdot H^{+,\UU}_{-\frac{n}2}
+ \sum\limits_{k=0}^{\frac{n}{2}} \big(\dirac_{R,(2)}^* V_{L,k}  -\tfrac{\rmi}{4(k+1)} V_{L,k+1}\circ \gradzs \Gamma \big)\cdot G^{+,\UU}_{1-\frac{n}2+k} \\
&\quad + 
\dirac_{R,(2)}^*V_{L,\tfrac{n+2}{2}} \cdot G^{+,\UU}_{2} 
+ \tfrac{1}{n^2\pi} V_{L,\frac{n}{2}}\circ \gradzs\Gamma\cdot F^{+,\UU}_{0}
+\tfrac{1}{(n+2)^2\pi} V_{L,\frac{n+2}{2}}\circ \gradzs\Gamma\cdot F^{+,\UU}_{1} .
\end{align*}
Thus, near the diagonal, the Schwartz kernel of $\hat G^\loc_{L} \dirac_R - \dirac_R \hat G^\loc_{R}$ takes the form
\begin{align}
\sum\limits_{k=0}^{\frac{n+2}{2}} \Big(\dirac_{R,(2)}^* \big(V_{L,k} \cdot G^{+,\UU}_{1-\frac{n}2+k}\big) &- \dirac_{R,(1)} \big(V_{R,k} \cdot G^{+,\UU}_{1-\frac{n}2+k}\big)\Big) \notag\\
&= 
A_{-1}\cdot H^{+,\UU}_{-\frac{n}2} + \sum\limits_{k=0}^{\frac{n+2}{2}} A_k\cdot G^{+,\UU}_{1-\frac{n}2+k} + B_0\cdot F^{+,\UU}_{0} + B_1\cdot F^{+,\UU}_{1} 
\label{eq:structure1}
\end{align}
where
\begin{align*}
A_{-1} &=
\tfrac{\rmi}{4} (\grades \Gamma \circ V_{R,0}  - V_{L,0}\circ \gradzs \Gamma) ,\\
A_k &=
\dirac_{R,(2)}^* V_{L,k} - \dirac_{R,(1)}V_{R,k}  -\tfrac{\rmi}{4(k+1)} \big(V_{L,k+1}\circ \gradzs \Gamma - \grades \Gamma\circ V_{R,k+1} \big),\quad k=0,\ldots,\tfrac{n}{2},\\
A_{\frac{n+2}{2}} &=
\dirac_{R,(2)}^*V_{L,\tfrac{n+2}{2}} - \dirac_{R,(1)}V_{R,\tfrac{n+2}{2}},\\
B_k &=-
\tfrac{1}{(n+2k)^2\pi} \Big( \grades\Gamma\circ V_{R,\frac{n+2k}{2}} - V_{L,\frac{n+2k}{2}}\circ \gradzs\Gamma \Big) , \quad k=0,1 .
\end{align*}
Note that the coefficients $A_k$ and $B_k$ are smooth.
In order to understand the $A_k$ better, we introduce the product manifold $\tilde X = X \times \R$ with metric $\tilde g = g + ds^2$, where $s$ is the additional spacelike variable.
Pulling the bundle $\SS$ back to $\tilde X$ along the obvious projection, we obtain a bundle $\tilde \SS\to\tilde X$ which is a direct sum $\tilde \SS = \tilde \SS_L \oplus \tilde \SS_R$. On $\tilde \SS$  we define the operator $\tilde \dirac$ by
$$
 \tilde \dirac = \left( \begin{matrix} \rmi \partial_s &  \dirac_{R} \\ \dirac_{L} & -\rmi \partial_s \end{matrix} \right).
$$
Its square $\tilde P = \tilde \dirac^2$ is a normally hyperbolic operator of the form $\tilde P_L \oplus \tilde P_R$, where $\tilde P_L = \dirac_R \dirac_L - \partial_s^2$ and $\tilde P_R = \dirac_L \dirac_R - \partial_s^2$. 
It follows that the retarded fundamental solution $\tilde G_\ret$ is of the form $\tilde G_\ret = \tilde G_{L,\ret} \oplus \tilde G_{R,\ret}$.
It has an expansion into a Hadamard series in the sense that
$$
 \tilde G_\ret - \sum_{k=0}^{\frac{n+2}{2}} \tilde V_k \cdot \tilde R^+_{1-\frac{n+1}{2} +k}
$$
is $C^1$ near the diagonal by \cite{BGP07}*{Prop.~2.5.1}.
Here $\tilde R^+_j$ denote the Riesz distributions on $\tilde X$.
By Lemma~\ref{lem:HadamardProduktSpace}, the Hadamard coefficients are independent of the extra variables $s_1,s_2$, and we have the relation
$$
  \tilde V_k =  \begin{pmatrix} V_{L,k} & 0 \\ 0 & V_{R,k} \end{pmatrix} .
$$
From $\tilde \Gamma (x_1,s_1,x_2,s_2) = \Gamma(x_1,x_2) - (s_1-s_2)^2$ we find
$$  
\grades \tilde \Gamma = \begin{pmatrix} -2 (s_1-s_2) & \grades \Gamma \\ \grades\Gamma & -2 (s_1-s_2) \end{pmatrix}
\quad\mbox{ and }\quad
\gradzs \tilde \Gamma = \begin{pmatrix} 2 (s_1-s_2) & \gradzs \Gamma \\ \gradzs\Gamma & 2 (s_1-s_2) \end{pmatrix}.
$$
We now use the fact that $ \tilde \dirac_{(1)} \tilde G_\ret -  \tilde \dirac^*_{(2)} \tilde G_\ret=0$, which follows from uniqueness of the retarded fundamental solution of $\tilde\dirac$. 
A similar computation as above, using Proposition~\ref{FGOmega} and Lemma~\ref{lem:QH}, shows that
\begin{align}
\sum\limits_{k=0}^{\frac{n+2}{2}} \Big(\tilde\dirac_{(2)}^* \big(\tilde V_{k} \cdot \tilde R^{+,\UU}_{1-\frac{n+1}2+k}\big) - \tilde \dirac_{(1)} \big(\tilde V_{k} \cdot \tilde R^{+,\UU}_{1-\frac{n+1}2+k}\big)\Big) 
= 
\tilde A_{-1}\cdot \tilde Q^{+,\UU}_{-\frac{n+1}2} + \sum\limits_{k=0}^{\frac{n+2}{2}} \tilde A_k\cdot \tilde R^{+,\UU}_{1-\frac{n+1}2+k} 
\label{eq:structure2}
\end{align}
where
\begin{align*}
\tilde A_{-1} &=
\tfrac{\rmi}{4} (\grades \tilde\Gamma \circ \tilde V_{0}  - \tilde V_{0}\circ \gradzs \tilde\Gamma) \\
&=
\tfrac{\rmi}{4} \left\{\begin{pmatrix} -2 (s_1-s_2) & \grades \Gamma \\ \grades\Gamma & -2 (s_1-s_2) \end{pmatrix}
                \begin{pmatrix} V_{L,0} & 0 \\ 0 & V_{R,0} \end{pmatrix}
               -\begin{pmatrix} V_{L,0} & 0 \\ 0 & V_{R,0} \end{pmatrix}
                \begin{pmatrix} 2 (s_1-s_2) & \gradzs \Gamma \\ \gradzs\Gamma & 2 (s_1-s_2) \end{pmatrix}\right\} \\
&=
\begin{pmatrix}* & A_{-1} \\ * & *\end{pmatrix}, \\
\tilde A_k &=
\tilde\dirac_{(2)}^* \tilde V_{k} - \tilde\dirac_{(1)}\tilde V_{k}  -\tfrac{\rmi}{4(k+1)} \big(\tilde V_{k+1}\circ \gradzs \tilde\Gamma - \grades \tilde\Gamma\circ \tilde V_{k+1} \big)
=
\begin{pmatrix}* & A_{k} \\ * & *\end{pmatrix},\quad k=0,\ldots,\tfrac{n}{2} ,\\
\tilde A_{\frac{n+2}{2}} &=
\tilde \dirac_{(2)}^* \tilde V_{\tfrac{n+2}{2}} - \tilde\dirac_{(1)}\tilde V_{\tfrac{n+2}{2}}
=
\begin{pmatrix}* & A_{\frac{n+2}{2}} \\ * & *\end{pmatrix}.
\end{align*}
Here the symbol $*$ in the matrices stands for expressions which we do not need to calculate.

Put $\UU^+ := \{(x_1,x_2)\in\UU \mid x_2\in \J^+(x_1)\}$.
Then $\Gamma>0$ on $\UU^+$.
For $(x_1,x_2)\in\UU^+$ and $|s|<\sqrt{\Gamma(x_1,x_2)}$ we have $\tilde \Gamma(x_1,s,x_2,0)>0$.
Moreover, at these points $(x_1,s,x_2,0)$ the Riesz distributions $\tilde R^{+,\UU}_\beta$ coincide with the smooth functions $2C(\beta,n+1)\tilde\Gamma^\beta$.
Furthermore, by Lemma~\ref{lem:QH}~\eqref{QH1}, $\tilde Q^{+,\UU}_{-\frac{n+1}{2}}$ coincides with the smooth function
$$
2(2-(n+1))\tilde R^{+,\UU}_{1-\frac{n+1}{2}}\cdot\tilde\Gamma^{-1}
=
4(1-n)C(1-\tfrac{n+1}{2},n+1)\tilde\Gamma^{-\frac{n+1}{2}}
=
\tfrac{2\pi^{\frac{1-n}{2}}}{(-\tfrac{n+1}{2})!}\tilde\Gamma^{-\frac{n+1}{2}} .
$$
We find that the $C^1$-Schwartz kernel in \eqref{eq:structure2} takes the following form on $\UU^+$:
$$
\tilde A_{-1} \tfrac{2\pi^{\frac{1-n}{2}}}{(-\tfrac{n+1}{2})!}(\Gamma-s^2)^{-\frac{n+1}{2}}
+ 2\sum\limits_{k=0}^{\frac{n+2}{2}} \tilde A_k \cdot C(1-\tfrac{n+1}{2}+k,n+1) \cdot (\Gamma-s^2)^{1-\frac{n+1}{2}+k} .
$$
Thus, the smooth function
\begin{align*}
(x_1,s,x_2)\mapsto A_{-1}&(x_1,x_2) \frac{2\pi^{\frac{1-n}{2}}}{(-\tfrac{n+1}{2})!}(\Gamma(x_1,x_2)-s^2)^{-\frac{n+1}{2}} \\
&+ 2\sum\limits_{k=0}^{\frac{n+2}{2}} A_k(x_1,x_2) \cdot C(1-\tfrac{n+1}{2}+k,n+1) \cdot (\Gamma(x_1,x_2)-s^2)^{1-\frac{n+1}{2}+k}
\end{align*}
on $\UU^+\times\R$ is the restriction of a $C^1$-function on $\UU\times\R$.
Taking the limit $s^2\to\Gamma(x_1,x_2)$ shows that $A_k\equiv0$ on $\UU^+$ for $k=-1,0,\ldots,\frac{n+2}{2}$.

Returning to Feynman propagators, \eqref{eq:structure1} simplifies on $\UU^+$ to
\begin{align*}
\sum\limits_{k=0}^{\frac{n+2}{2}} \Big(\dirac_{R,(2)}^* \big(V_{L,k} \cdot G^{+,\UU}_{1-\frac{n}2+k}\big) &- \dirac_{R,(1)} \big(V_{R,k} \cdot G^{+,\UU}_{1-\frac{n}2+k}\big)\Big)
=
B_0\cdot F^{+,\UU}_{0} + B_1\cdot F^{+,\UU}_{1} .
\end{align*}
Applying $\dirac_L$ to the first argument, we find for the continuous Schwartz kernel of $\dirac_L\hat G^\loc_{L} \dirac_R - \dirac_L\dirac_R\hat G^\loc_{R}$ on $\UU^+$:
\begin{align}
\sum\limits_{k=0}^{\frac{n+2}{2}} \Big(\dirac_{L,(1)}\dirac_{R,(2)}^* \big(V_{L,k} \cdot G^{+,\UU}_{1-\frac{n}2+k}\big) &- \dirac_{L,(1)}\dirac_{R,(1)} \big(V_{R,k} \cdot G^{+,\UU}_{1-\frac{n}2+k}\big)\Big)
=
\dirac_{L,(1)}\big(B_0 F^{+,\UU}_{0} + B_1 F^{+,\UU}_{1}\big) .
\label{eq:structure3}
\end{align}
All terms of the RHS are smooth.
Since the diagonal of $X\times X$ is contained in the closure of $\UU^+$, continuity of the kernel implies that \eqref{eq:structure3} also holds along the diagonal.
Since $\grades\Gamma$ and $\gradzs\Gamma$ vanish along the diagonal, so do $B_0$ and $B_1$.
Moreover, $F^{+,\UU}_1$ vanishes along the diagonal too.
Thus restricting to the diagonal yields
\begin{align}
[\dirac_L\hat G^\loc_{L} \dirac_R - \dirac_L\dirac_R \hat G^\loc_{R}]
&=
\dirac_{L,(1)}B_0\cdot F^{+,\UU}_{0} - \rmi \grades F^{+,\UU}_{0} \circ B_0 +\dirac_{L,(1)}B_1\cdot F^{+,\UU}_{1} \notag\\
&\quad - \rmi \grades F^{+,\UU}_{1} \circ B_1 \notag\\
&=
\dirac_{L,(1)}B_0\cdot F^{+,\UU}_{0} \notag\\
&=
-\tfrac{1}{n^2\pi}\dirac_{L,(1)} \Big( \grades\Gamma\circ V_{R,\frac{n}{2}} - V_{L,\frac{n}{2}}\circ \gradzs\Gamma \Big) \cdot C(0,n)\notag\\
&=
-\tfrac{1}{n^2\pi}\dirac_{L,(1)} \Big( \grades\Gamma\circ V_{R,\frac{n}{2}} - V_{L,\frac{n}{2}}\circ \gradzs\Gamma \Big) \cdot \frac{\pi^{\frac{2-n}{2}}}{2^n (\frac{n-2}{2})!}\notag\\
&=
-\frac{1}{2n (4\pi)^{\nicefrac{n}{2}}(\frac{n}{2})!}\dirac_{L,(1)} \Big( \grades\Gamma\circ V_{R,\frac{n}{2}} - V_{L,\frac{n}{2}}\circ \gradzs\Gamma \Big) .
\label{eq:fastamziel}
\end{align}
Let $e_1,\ldots,e_n$ be a local Lorentz-orthonormal tangent frame near a point on the diagonal.
Then $g(e_k,e_\ell) = \eps_k\delta_{k\ell}$ with $\eps_k=\pm 1$.
Since $\gradzs\Gamma$ vanishes along the diagonal, using \eqref{eq:DiracLeibnitz} we find along the diagonal:
\begin{align*}
\dirac_{L,(1)} \Big( V_{L,\frac{n}{2}}\circ \gradzs\Gamma\Big)
&=
\dirac_{L,(1)} \sum_k \eps_k (\partial_{e_k,(2)}\Gamma) V_{L,\frac{n}{2}}\circ \slashed e_k \\
&=
-\rmi \sum_k \eps_k \grades(\partial_{e_k,(2)}\Gamma) \circ V_{L,\frac{n}{2}}\circ \slashed e_k \\
&=
-\rmi \sum_{k,\ell} \eps_\ell\eps_k (\partial_{e_\ell,(1)}\partial_{e_k,(2)}\Gamma) \slashed e_\ell \circ V_{L,\frac{n}{2}}\circ \slashed e_k \\
&=
2\rmi \sum_{k} \eps_k  \slashed e_k \circ V_{L,\frac{n}{2}}\circ \slashed e_k \\
&=
2n\rmi \theta(V_{L,\frac{n}{2}}) .
\end{align*}
The verification of
$$
\dirac_{L,(1)} \Big( \grades\Gamma\circ V_{R,\frac{n}{2}}\Big)
=
2n\rmi V_{R,\frac{n}{2}}
$$
is even simpler.
Inserting this into \eqref{eq:fastamziel} concludes the proof.
\end{proof}

We define the \emph{Dirac currents} $J^{{\pm}}$ by
\begin{equation}
 J^{\pm}(\xi)(x) :=  \rmi\mathrm{tr}(\slashed \xi \circ[\dirac_L \hat G_{\pm,L}^{\reg}](x))
\label{eq:DefJ+-}
\end{equation}
for any $x\in X$ and $\xi\in T_xX$.
Then $J^+$ and $J^-$ are $1$-forms of $C^1$-regularity.
The difference $J= J^+ - J^-$ is given by
$$
 J(\xi)(x)
=  
\rmi\mathrm{tr}(\slashed \xi \circ[\dirac_L (\hat G^{\reg}_{+,L}- \hat G^{\reg}_{-,L})](x))
=  
\rmi\mathrm{tr}(\slashed \xi \circ[\dirac_L (\hat G^+_{L}- \hat G^-_{L})](x)).
$$

\begin{lemma}
\label{lem:Jclosed}
The $1$-form $J$ is co-closed,
$$
\delta J=0.
$$
\end{lemma}

\begin{proof}
Since $\delta J$ is continuous, it suffices to show that $\delta J=0$ in the distributional sense in the interior of $X$.
Let $u\in C^\infty_c(\mathring X)$.
Then, using \eqref{eq:DiracLeibnitz}, $\dirac_R \dirac_L (\hat G^+_{L}- \hat G^-_{L}) =\dirac_L \dirac_R (\hat G^+_{R}- \hat G^-_{R}) =0$ and Lemma~\ref{lem:DiracGTausch}, we find
\begin{align*}
\int_X \delta J\cdot u \dV
&=
\int_X J(\grad u) \dV \\
&=
\rmi\int_X \mathrm{tr}(\grads u\circ[\dirac_L (\hat G^+_{L}- \hat G^-_{L})]) \dV \\
&=
\rmi\int_X \mathrm{tr}[\grads u\dirac_L (\hat G^+_{L}- \hat G^-_{L})] \dV \\
&=
-\int_X \mathrm{tr}[(\dirac_R\circ u-u\dirac_R)\dirac_L (\hat G^+_{L}- \hat G^-_{L})] \dV \\
&=
-\Tr(\dirac_R\circ u\circ \dirac_L\circ (\hat G^+_{L}- \hat G^-_{L})) \\
&=
-\Tr(u\circ \dirac_L\circ (\hat G^+_{L}- \hat G^-_{L})\circ \dirac_R) \\
&=
-\Tr(u\circ \dirac_L\circ \dirac_R\circ (\hat G^+_{R}- \hat G^-_{R})) \\
&=
0.
\qedhere
\end{align*}
\end{proof}

Let $U= U_{\Sigma_+,\Sigma_-}: L^2_\mathrm{c}(\Sigma_-;\SS_R)\to L^2_\mathrm{c}(\Sigma_+;\SS_R)$ be the time-evolution operator as defined in~\eqref{eq:U}.

\begin{theorem}
 Suppose that $X$ is spatially compact and that $\dirac_L$ has product structure near $\Sigma_+$ and $\Sigma_-$.
 Let $P_+ = p_\geq(\D_{\Sigma_+})$ and $P_- = U p_\geq(\D_{\Sigma_-}) U^{-1}$. 
 Then $P_+ -P_-$ has a smooth integral kernel. 
 In particular, $(P_+,P_-)$ is a Fredholm pair and 
 $$
  \mathrm{ind}(P_+,P_-) = \Tr(P_+ -P_-) = \int_\Sigma J(n_{\Sigma}) \dA,
 $$
 where integration is over any smooth spacelike Cauchy hypersurface $\Sigma$.
\end{theorem}

\begin{proof}
The theorem can be inferred from the analysis in \cite{Baer:2015aa} and \cite{baerstroh2015chiral}. 
For the sake of completeness we give a direct proof.
By \eqref{eq:FeynmanRegularized}, the integral kernel of
 $$
  -\rmi\dirac_L (\hat G^+_{L} - \tfrac{1}{2}(\hat G_{L,\ret}+\hat G_{L,\adv}))
  =
  -\rmi\big(\hat D - \tfrac12 (\hat G_\ret^{\dirac_L} + \hat G_\adv^{\dirac_L})\big)
 $$
 restricted to $\Sigma_+ \times \Sigma_+$ coincides with the integral kernel of $\frac{1}{2} (p_\geq(\D_{\Sigma_+}) - p_<(\D_{\Sigma_+}))\nSp = (P_+ - \frac{1}{2} \id)\nSp$. 
 Similarly, the Schwartz kernel of $-\rmi\dirac_L (\hat G^-_{L} - \tfrac{1}{2}(\hat G_{L,\ret}+\hat G_{L,\adv}))$ restricts on $\Sigma_-\times\Sigma_-$ to that of $(p_\geq(\D_{\Sigma_-}) - \frac{1}{2} \id)\nSm$.

Temporarily denote the Schwartz kernel of $-\rmi \dirac_L (\hat G^-_{L} - \tfrac{1}{2}(\hat G_{L,\ret}+\hat G_{L,\adv}))$ by $H$ and that of $p_\geq(\D_{\Sigma_-}) - \tfrac{1}{2} \id$ by $Q$.
Then $(p_\geq(\D_{\Sigma_-}) - \tfrac{1}{2} \id)\nSm$ has the kernel $(\id\otimes(\nSm)^*)Q$.
Let $U = U_{\Sigma_+,\Sigma_-}\colon L^2_\mathrm{c}(\Sigma_-;\SS)\to L^2_\mathrm{c}(\Sigma_+;\SS)$ be the time-evolution operator for $\dirac$ and $\tilde U = \tilde U_{\Sigma_+,\Sigma_-}\colon L^2_\mathrm{c}(\Sigma_-;\SS^*)\to L^2_\mathrm{c}(\Sigma_+;\SS^*)$ that of $\dirac^*$.
From $\dirac_{R,(1)}H=0$ we get
$$
(U\otimes\id) (H|_{\Sigma_-\times X}) = H|_{\Sigma_+\times X}
$$
and from $\dirac_{R,(2)}^*H=0$ we find
$$
(\id\otimes\tilde U)(H|_{X\times\Sigma_-}) = H|_{X\times\Sigma_+}.
$$
Using the commutative diagram \eqref{eq:diagram-U} we find
\begin{align*}
H|_{\Sigma_+\times\Sigma_+}
&=
(\id\otimes\tilde U)(H|_{\Sigma_+\times\Sigma_-}) \\
&=
(U\otimes\tilde U)(H|_{\Sigma_-\times\Sigma_-}) \\
&=
(U\otimes\tilde U)(\id\otimes (\nSm)^*)Q \\
&=
(U\otimes (\nSm\tilde U^*)^*)Q\\
&=
(U\otimes (U^{-1}\nSp)^*)Q\\
&=
(\id\otimes(\nSp)^*)(U\otimes (U^{-1})^*)Q .
\end{align*}
This says that $H|_{\Sigma_+\times\Sigma_+}$ is the kernel of $U(p_\geq(\D_{\Sigma_-}) - \tfrac{1}{2} \id)U^{-1}\nSp=(P_- - \tfrac{1}{2} \id)\nSp$.
It follows that the integral kernel of $P_+-P_-$ is given by the restriction to $\Sigma_+ \times \Sigma_+$ of the smooth kernel of
$$
- \rmi \dirac_L (\hat G^+_{L} -\hat G^-_{L})\nSp^{-1}
=
\rmi \dirac_L (\hat G^+_{L} -\hat G^-_{L})\nSp .
$$
In particular, $P_+ - P_-$ has a smooth integral kernel.
By Mercer's theorem, the trace of $P_+ - P_-$ is given by the integral
$$
\int_{\Sigma_+} \mathrm{tr}[\rmi \dirac_L (\hat G^+_{L}- \hat G^-_{L})\nSp ](y) dA(y)
=
\int_{\Sigma_+} J(n_{\Sigma_+}) \dA.
$$
Since $J$ is co-closed, the integral over any other Cauchy hypersurface gives the same result by Gauss' divergence theorem.
\end{proof}

\begin{theorem}
\label{thm:indexJ}
Let $X$ be an  even dimensional compact globally hyperbolic manifold with boundary $\partial X=\Sigma_- \sqcup \Sigma_+$.
Let $\dirac=
\begin{pmatrix}
0 & \dirac_R \\
\dirac_L & 0
\end{pmatrix}$ be an odd Dirac-type operator over $X$.
Suppose that $X$ and $\dirac$ have product structure near $\Sigma_\pm$.
Then under $\APS$-boundary conditions, the Dirac operator
$$
\dirac_L:C^\infty_{\APS}( X;\SS_L)\to C^\infty( X;\SS_R)
$$
as a continuous linear map between Fr\'echet spaces is Fredholm and its index is given by
$$
\mathrm{ind}(\dirac_L)
=
\Xi(\D_{\Sigma_+}) - \Xi(\D_{\Sigma_-}) - \int_{X} \delta J^{-}  \dV.
$$
\end{theorem}

\begin{proof}
The Fredholm property of Dirac-type operators on spatially compact globally hyperbolic Lorentzian manifolds subject to APS-boundary conditions was shown in \cite{Baer:2015aa}.
Moreover, it was shown there that $\mathrm{ind}(\dirac_L)=  \mathrm{ind}(P_+,P_-)$.
We therefore can apply the above theorem to compute 
\begin{equation}\label{eq:IndexLadung}
\mathrm{ind}(\dirac_L) = \int_{\Sigma_+} \big(J_{}^{{+}}(n_{\Sigma_+}) - J_{}^{{-}}(n_{\Sigma_+})\big)\dA.
\end{equation}
Gauss' divergence theorem yields
\begin{equation*} 
  \int_{\Sigma_+} J^{{-}}(n_{\Sigma_+})   \dA -  \int_{\Sigma_-} J^{{-}}(n_{\Sigma_-})    \dA  =
  \int_{X} \delta J^{{-}}    \dV .
\end{equation*}
Recall from Proposition~\ref{prop:etah} that
\begin{equation}
 \int_{\Sigma_\pm} J^{{\pm}}(n_{\Sigma_\pm}) \dA
 =
 \Xi(\D_{\Sigma_\pm}).
 \label{eq:StdRandterm}
 \end{equation}
Combining equations \eqref{eq:IndexLadung}--\eqref{eq:StdRandterm} we find
\begin{align*}
\mathrm{ind}(\dirac_L)
&=  
\int_{\Sigma_+} \left( J^{{+}}(n_{\Sigma_+}) - J^{{-}}(n_{\Sigma_+}) \right) \dA \\
&= 
\int_{\Sigma_+} J^{{+}}(n_{\Sigma_+})  \dA  -  \int_{\Sigma_-} J^{{-}}(n_{\Sigma_-})  \dA  -  \int_{X} \delta J^{{-}}   \dV \\
&=
\Xi(\D_{\Sigma_+}) - \Xi(\D_{\Sigma_-})  - \int_{X} \delta J^{{-}}  \dV.\qedhere
\end{align*}
\end{proof}

By the above, the following statement can be interpreted as a local version of the index theorem which holds irrespective of spatial compactness.

\begin{theorem}
\label{thm:deltaJHadamard}
Let $X$ be an $n$-dimensional globally hyperbolic manifold with boundary $\partial X=\Sigma_-\sqcup\Sigma_+$.
Let $n$ be even and let $\dirac=
\begin{pmatrix}
0 & \dirac_R \\
\dirac_L & 0
\end{pmatrix}$ be an odd Dirac-type operator over $X$.
Let $V_{R,k}$ be the Hadamard coefficients of $\dirac_L\dirac_R$ and $V_{L,k}$ those of $\dirac_R\dirac_L$.
Let $J^-$ be the Dirac current for the Feynman propagator of $\dirac_R\dirac_L$, as defined in \eqref{eq:DefJ+-}.

Then the index density $\delta J^{{-}}$ can be expressed in terms of the Hadamard coefficients via
$$
\delta J^{{-}}
=
\frac{\tr(V_{L,\frac{n}{2}})  -\tr(V_{R,\frac{n}{2}})}{(4 \pi)^\frac{n}{2} (\frac{n}{2})!}.
$$
\end{theorem}

\begin{proof}
The same computation as in the proof of Lemma~\ref{lem:Jclosed} with $J^-$ instead of $J$ yields for an arbitrary test function $u \in C_\mathrm{c}^\infty(\mathring{X})$:
\begin{align*}
\int_X u\delta J^{-}\dV
&=
\rmi \int_X u \{ \tr[\dirac_L\hat{G}^\reg_{-,L}\dirac_R]-\tr[\dirac_R\dirac_L\hat{G}^\reg_{-,L}] \}\dV .\end{align*}
Proposition~\ref{prop:ZerlegeG} implies that the kernels of $\dirac_R\dirac_L\hat{G}^\reg_{-,L}$ and $\dirac_L\dirac_R\hat{G}^\reg_{-,R}$ vanish along the diagonal.
Thus, 
\begin{align*}
\delta J^{-}
&=
\rmi \{ \tr[\dirac_L\hat{G}^\reg_{-,L}\dirac_R]-\tr[\dirac_R\dirac_L\hat{G}^\reg_{-,L}] \} \\
&=
\rmi \{ \tr[\dirac_L\hat{G}^\reg_{-,L}\dirac_R]-\tr[\dirac_L\dirac_R\hat{G}^\reg_{-,R}] \} \\
&=
-\rmi \{ \tr[\dirac_L\hat{G}^\loc_{-,L}\dirac_R]-\tr[\dirac_L\dirac_R\hat{G}^\loc_{-,R}] \},
\end{align*}
by Lemma~\ref{lem:DiracGTausch}.
Proposition~\ref{niceprop} concludes the proof.
\end{proof}

\begin{corollary}
\label{cor:charforms}
Let $X$ be spin$^c$ with associated determinant line bundle $(L,\nabla^L)$ and let $\dirac$ be a twisted Dirac operator with twisting bundle $(E,\nabla^E)$.
Then
$$
 \delta J^{{-}} \dV
 =
 \Big[-\Adach(\nabla^g) \wedge \exp\big(\tfrac12 c_1(\nabla^L)\big) \wedge \ch(\nabla^E)\Big]_n .
$$
Here $[\ldots]_n$ denotes the homogeneous part of degree $n$ of a mixed form, $\Adach(\nabla^g)$ the $\Adach$-form of the tangent bundle equipped with the Levi-Civita connection, $c_1(\nabla^L)$ the first Chern form of the determinant line bundle $(L,\nabla^L)$, and $\ch(\nabla^E)$ the Chern character form of the twist bundle $(E,\nabla^E)$.
\end{corollary}

\begin{proof}
On the diagonal, the Hadamard coefficients are given by essentially the same universal polynomials of the curvatures and their covariant derivatives as the heat coefficients on Riemannian manifolds, see \eqref{eq:Vkak}.
Together with the local index theorem in the elliptic case this implies
\begin{equation*}
\frac{\tr(V_{R,\frac{n}{2}})  -\tr(V_{L,\frac{n}{2}})}{(4 \pi)^\frac{n}{2} (\frac{n}{2})!}
=
\Big[\Adach(\nabla^g) \wedge \exp\big(\tfrac12 c_1(\nabla^L)\big) \wedge \ch(\nabla^E)\Big]_n .
\qedhere
\end{equation*}
\end{proof}

Theorems~\ref{thm:indexJ} and \ref{thm:deltaJHadamard} and Corollary~\ref{cor:charforms} combine to give

\begin{corollary}
\label{cor:index}
Suppose that $X$ is spatially compact and $\D$ has product structure near $\Sigma_\pm$.
Then under $\APS$-boundary conditions, the Dirac-type operator
$$
\dirac_L:C^\infty_{\APS}( X;\SS_L)\to C^\infty( X;\SS_R)
$$
as a continuous linear map between Fr\'echet spaces is Fredholm and its index is given by
$$
\mathrm{ind}(\dirac_L)
=
\Xi(\D_{\Sigma_+}) - \Xi(\D_{\Sigma_-}) + \int_{X}  \frac{\tr(V_{R,\frac{n}{2}})  -\tr(V_{L,\frac{n}{2}})}{(4 \pi)^\frac{n}{2} (\frac{n}{2})!} \dV.
$$
If $\dirac$ is a twisted spin$^c$-Dirac operator with twisting bundle $(E,\nabla^E)$, then
\begin{equation}
\mathrm{ind}(\dirac_L)
=
\Xi(\D_{\Sigma_+}) - \Xi(\D_{\Sigma_-})
+ \int_{X} \Adach(\nabla^g) \wedge \exp\big(\tfrac12 c_1(\nabla^L)\big) \wedge \ch(\nabla^E) .
\tag*{$\qed$}
\end{equation}
\end{corollary}

\appendix

\section{Expansions of distinguished parametrices}
\label{sec:HadDist}

In this appendix we give a new construction of Feynman parametrices based on a family of distributions which is similar to 
that of Riesz distributions.
We give a precise description of their singularity structures which looks different depending on the parity of the dimension of the underlying manifold.
The Hadamard coefficients which occur are central to our derivation of the local index density, but the appendix is self contained and should be of independent interest.

\subsection{\texorpdfstring{Two families of homogeneous distributions on $\pmb{\R}$}{Two families of homogeneous distributions on R}}

Let $\log$ be the branch of the logarithm defined on $\C\setminus (-\infty,0]$ which coincides with the standard logarithm on positive real numbers.
Then the imaginary part of $\log$ takes values in $(-\pi,\pi)$.
For $\beta\in\C$ the function $z\mapsto z^{\beta} := \exp(\beta\log(z))$ is holomorphic on $\C\setminus (-\infty,0]$.
By \cite{Ho1}*{Thm.~3.1.11}, the function $z\mapsto z^{\beta}$, restricted to the strip $\{z\in\C \mid 0<\Im(z)<1\}$, has a distributional limit $f_\beta^+\in\DD'(\R)$ as $\Im(z)\searrow 0$.

Then $f^+_\beta$ coincides with the function $z\mapsto z^{\beta}$ on positive real numbers and with $z\mapsto |z|^\beta\cdot\exp(i\beta\pi)$ on negative real numbers.
In particular, $f^+_\beta$ is smooth on $\R\setminus\{0\}$, hence $\singsupp (f^+_\beta)\subset\{0\}$.
By \cite{Ho1}*{Thm.~8.1.6} the wavefront set satisfies $\WF(f^+_\beta)\subset (0,\infty)dt \subset T^*_0\R$ where $t$ is the standard coordinate on $\R$.
For $\beta\in \N_0$ the distribution is just the smooth function $t\mapsto t^\beta$.
Otherwise, the wavefront set is nonempty and hence $\WF(f^+_\beta) = (0,\infty)dt \subset T^*_0\R$. 
We also define $f_{\beta}^-:=\overline {f_{\overline \beta}^+}$ and note $\WF(f^-_\beta) = (-\infty,0)dt \subset T^*_0\R$ if $\beta\notin\N_0$.

If $k\in\N_0$ and $\Re(\beta)>k$ then $f^\pm_\beta$ is $C^k$ and vanishes to $k^\mathrm{th}$ order at $0$.

Note that $z^\beta=\exp(\beta\log(z))$ depends holomorphically on $\beta$.
Since, as $\Im(z) \searrow 0$, the limit is linear and continuous, the limit distribution $f^+_\beta$ depends holomorphically on $\beta$ in the sense that for each test function $\phi\in C_\mathrm{c}^\infty(\R)$, the map $\C\to\C$, $\beta\mapsto f^\pm_\beta[\phi]$, is holomorphic.
The same holds for the family $f^-_\beta$.

We differentiate the $f_\beta^\pm$ in $\beta$ to obtain another holomorphic family of distributions,
$$
 h_{\beta}^\pm := \tfrac{d}{d\beta} f_\beta^\pm.
$$
If $\Re(\beta)>0$ then $h_{\beta}^+$ is given by the function $z^\beta\, \LOG(z)$ where $\LOG$ is a branch of the logarithm with cut in the lower half plane and which coincides with the usual logarithm on positive real numbers.

Moreover, $h_{\beta}^-= \overline{h_{\overline{\beta}}^+}$. 
By the same argument as before, we get $\WF(h^\pm_\beta) = \pm(0,\infty)dt \subset T^*_0 \R$. 
Again, if $k\in\N_0$ and $\Re(\beta)>k$ then $h^\pm_\beta$ is $C^k$ and vanishes to $k^\mathrm{th}$ order at $0$.

Yet another family of homogeneous distributions on $\R$ is defined by $t_\pm^\beta := \chi_{[0,\infty)}(\pm t) | t |^\beta$ for $\Re(\beta)> -1$ and then by meromorphic continuation for arbitrary $\beta\in\C$ by means of the formula $\frac{d}{dt} t_\pm^\beta = \pm \beta  t_\pm^{\beta-1}$ where $\beta\in\C\setminus\{0,-1,-2,\ldots\}$. 
The only poles of $\beta \mapsto t_\pm^\beta$ are simple poles at negative integers.
Their residues are given in terms of derivatives of the delta distribution $\delta_0$, see for example \cite{Ho1}, Section~3.2.
Note that
\begin{align}
 f_\beta^+ + f_\beta^- &= 2\, t_+^\beta + 2 \cos(\beta \pi)  t_-^\beta, \label{beziehung1}\\
 -\rmi ( f_\beta^+ - f_\beta^-) &= 2 \sin(\beta \pi)  t_-^\beta \label{beziehung2}.
\end{align}
One checks this directly for large $\Re(\beta)$ and then concludes by unique continuation that the relations hold for all $\beta\in\C$.

\subsection{Lorentz invariant homogeneous distributions on Minkowski space}

In this appendix we construct a distinguished family of distributions that in some ways is similar to the Riesz distributions on $n$-dimensional Minkowski space, except that their wavefront set is that of the Feynman propagator.
In the same way as the Riesz distributions are used in the Hadamard construction of retarded and advanced fundamental solutions (see for example \cites{paul2014huygens, baum1996normally, BGP07} and references therein in the geometric setting), this family of distributions can be used directly to construct a Feynman parametrix.

Let $n\in\N$.
We equip $\R^n$ with the Minkowski metric $\ml x,y\mr =-x_1y_1+x_2y_2+\ldots+x_ny_n$.
We denote the negative Minkowski square by $\gamma:\R^n\to\R$, $\gamma(X)=-\ml X,X\mr$.
We provide Minkowski space with the time orientation for which $\frac{\partial}{\partial x_1}$ is future directed.
Let $\OO(n-1,1)$ be the group of Lorentz transformations, i.e.\ the group of linear automorphisms of $\R^n$ which preserve $\ml\cdot,\cdot\mr$.
Since $\gamma:\R^n\setminus\{0\} \to \R$ is a submersion we can pull back $f^\pm_\beta$ and obtain $\gamma^* f^\pm_\beta \in \DD'(\R^n\setminus\{0\})$.
By \cite{GS}*{Prop.~VI.3.2}, the wavefront set satisfies $\WF(\gamma^* f^\pm_\beta)\subset \{\lambda d\gamma(x) \mid x\in \CC\setminus\{0\},\, \pm\lambda>0\}$.
Note that the covector $d\gamma(x)=2(x_1dx_1-x_2dx_2-\ldots-x_ndx_n)$ is lightlike for $x\in\CC\setminus\{0\}$.
For $\beta\notin\N_0$ the distribution $\gamma^* f^\pm_\beta$ is not smooth at $\CC$ and hence $\WF(\gamma^* f^\pm_\beta) = \{\lambda d\gamma(x) \mid x\in \CC\setminus\{0\},\, \pm\lambda>0\}$. 
Since the $f^\pm_\beta$ depend holomorphically on $\beta$ so do the pullbacks $\gamma^* f^\pm_\beta$.

Since $\gamma$ is quadratic, the distributions $\gamma^* f^\pm_\beta$ are homogeneous of degree $2\beta$.
If $2\beta$ is not an integer $\le -n$ then $\gamma^* f^\pm_\beta$ has a unique extension to a homogeneous distribution on $\R^n$ by \cite{Ho1}*{Thm.~3.2.3}.
We denote the extension of $\gamma^* f^\pm_\beta$ by $(\gamma\pm i0)^\beta\in\DD'(\R^n)$. 
Then 
\begin{gather} \label{giwavefront}
\WF((\gamma\pm i0)^\beta) \subset \{\lambda d\gamma(x) \mid x\in \CC\setminus\{0\},\, \pm\lambda>0\} \cup \dot{T}_0^*\R^n. 
\end{gather}
The extension mapping in \cite{Ho1}*{Thm.~3.2.3} is linear and continuous, hence $(\gamma\pm i0)^\beta$ form a holomorphic family of distributions on $\R^n$ for $\beta\in\C\setminus\{-\frac{n}{2}, -\frac{n+1}{2},-\frac{n+2}{2},\ldots\}$.

Similarly, the derivative of $(\gamma \pm \rmi 0)^\beta$ with respect to $\beta$, will be denoted by $(\gamma \pm \rmi 0)^\beta \log(\gamma \pm \rmi 0)$. 
Hence, $(\gamma \pm \rmi 0)^\beta \log(\gamma \pm \rmi 0)$ is a holomorphic family of distributions defined on the set of $\beta$ for which $2 \beta$ is not a negative integer. 
On $\R^n\setminus\{0\}$ this distribution equals the pullback $\gamma^* h^\pm_\beta$, and we therefore have the inclusion
 \begin{gather} \label{giwavefrontlog}
\WF((\gamma\pm \rmi 0)^\beta\log(\gamma\pm \rmi 0)^\beta) \subset \{\lambda d\gamma(x) \mid x\in \CC\setminus\{0\},\, \pm\lambda>0\} \cup \dot{T}_0^*\R^n. 
\end{gather}
We put 
\begin{equation}
C(\beta,n) := \frac{2^{-n-2\beta}\pi^{\frac{2-n}{2}}}{(\beta+\frac{n-2}{2})! \beta!}
\label{eq:defC}
\end{equation}
and define
\begin{align*}
F^+_\beta &:= C(\beta,n)(\gamma - i0)^\beta,\\
F^-_\beta &:=   C(\beta,n)(\gamma + i0)^\beta = \overline{F^+_{\overline \beta}}.
\end{align*}%
If $\Re(\beta)>k\in\N_0$ then $F^\pm_\beta$ is $C^k$ and vanishes to $k^\mathrm{th}$ order along $\CC$.
Since $\beta\mapsto C(\beta,n)$ is an entire function, the $F^\pm_\beta$ are again holomorphic for $\beta\in\C\setminus\{-\frac{n}{2}, -\frac{n+1}{2},-\frac{n+2}{2},\ldots\}$.

Recall also the holomorphic family of \emph{Riesz distributions} $R^\pm_\beta$ on $\R^n$. 
This family can be introduced in a very general setting of arbitrary signature, and we refer to \cite{kolk1991riesz} and references therein for the general theory. 
They relate to the Riesz distributions in \cite{BGP07}*{Ch.~1} by $R^\pm_\beta=R_\pm(\alpha)$ where $\alpha=n+2\beta$. 
For $\Re(\beta)>-1$ they are defined as 
$$
R^\pm_\beta[\phi] = 2C(\beta,n)\int_{J^\pm(0)}\gamma^\beta \phi\, dx,
$$
where $J^\pm(0)$ is the future or past light cone, i.e.\ the set of future or past directed causal vectors, respectively. 
Note that for $\Re(\beta)>-1$ we have $R^+_\beta + R^-_\beta = 2 C(\beta,n) \gamma^*(t_+^\beta)$ on $\R^n\setminus\{0\}$.

As an auxiliary tool we also define the \emph{complementary Riesz distributions} for $\Re(\beta)>-1$ by
\begin{align*}
\tilde R_\beta[\phi] &:= 2C(\beta,n)\int_{\R^n \setminus J(0)}(-\gamma)^\beta \phi\, dx,\mbox{ i.e.}\\
\tilde R_\beta       &= 2 C(\beta,n) \gamma^*(t_-^\beta) \textrm{ on }\R^n\setminus\{0\}.
\end{align*}
In the same way as for the Riesz distributions one shows that $\tilde R_\beta$ extends to the entire complex plane as a holomorphic family of distributions.
Note that we have $\Box \tilde R_{\beta+1} = -\tilde R_\beta$, see e.g.\ \cite{kolk1991riesz}*{Eq.~(2.6)}, and 
\begin{gather*} 
\tilde R_{-\frac{n}{2}} = 
\begin{cases}
  (-1)^{\frac{n}{2}-1} 2\delta_0, & \textrm{for } n \textrm{ even}, \\ 
  0, &  \textrm{for } n \textrm{ odd},
\end{cases}
\end{gather*}
see \cite{kolk1991riesz}*{Eq.~(2.10)}.

\begin{prop} \label{F:properties}
The family $\beta\mapsto F^\pm_\beta$ of distributions on $\R^n$, $n\ge2$, extends uniquely to a holomorphic family of distributions defined for all $\beta\in\C$. 
The following holds for all $\beta\in\C$:
\begin{enumerate}[(a)]
\item\label{F:mult}
$\gamma\cdot F^\pm_\beta = (2\beta+2)(2\beta+n) F^\pm_{\beta+1}$;
\item\label{F:grad}
$(\grad\gamma)\cdot F^\pm_{\beta} = 2(2\beta+n)\grad (F^\pm_{\beta+1})$;
\item\label{F:invariance}
$F^\pm_\beta$ is $\OO(n-1,1)$-invariant, i.e.\ $A^*F^\pm_\beta=F^\pm_\beta$ for all $A\in\OO(n-1,1)$;
\item\label{F:realpart}
$
 F^+_\beta + F^-_\beta = \left(  R^+_\beta + R^-_\beta \right) + \cos(\pi \beta)  \tilde R_\beta ;
$
\item\label{F:impart} 
$
\rmi (F^+_\beta - F^-_\beta) = \sin(\pi \beta) \tilde R_\beta ;
$
\item\label{F:box}
$\Box F^\pm_{\beta+1}=F^\pm_{\beta}$;
\item\label{F:traeger1}
for every $\beta\in\C\setminus\{-1,-2,\ldots\}\cup\{-\frac{n}{2},-\frac{n}{2}-1,\ldots\}$ we have
\begin{enumerate}[(i)]
\item\label{F:traeger1-1} 
$\supp(F^\pm_\beta)=\R^n$;
\item\label{F:traeger1-2}
$\singsupp(F^\pm_\beta)=\emptyset$ if $\beta\in\N_0$ and $\singsupp(F^\pm_\beta)=\CC$ if $\beta\notin\N_0$;
\end{enumerate}
\item\label{F:Ck}
if $k\in\N_0$ and $\Re(\beta)>k$ then $F^\pm_\beta$ is $C^k$ and vanishes to $k^\mathrm{th}$ order along $\CC$;
\item\label{F:nullsein}
if $-\beta \in \N$ we have $F^\pm_{\beta}=0$;
\item\label{F:traeger2}
for $\beta\in\{-1,-2,\ldots\}\cup\{-\frac{n}{2},-\frac{n}{2}-1,\ldots\}$ we have
$$
\supp(F^\pm_\beta)=\singsupp(F^\pm_\beta)=\CC \quad\mbox{or}\quad\supp(F^\pm_\beta)=\singsupp(F^\pm_\beta)=\{0\} \quad\mbox{or}\quad F^\pm_\beta=0;
$$
\item \label{F:wavefront}
$\WF(F^\pm_\beta)\subset \{\lambda d\gamma(x) \mid x\in \CC\setminus\{0\},\, \mp\lambda>0\} \cup \dot{T}_0^*\R^n$;
\item\label{F:delta}
for the value $\beta=-\frac{n}{2}$ we have
$$
F^\pm_{-n/2}
=
\begin{cases}
\delta_0,& \mbox{ if $n$ is odd,}\\
0, & \mbox{ if $n$ is even.}
\end{cases}
$$
\end{enumerate}
\end{prop}

\begin{proof}
First assume $\Re(\beta)>1$.
Then the distributions $F^\pm_\beta$ and $F^\pm_{\beta+1}$ are $C^1$-functions, and we simply compute
$$
\gamma\cdot F^\pm_\beta
=
C(\beta,n)\gamma^{\beta+1}
=
\frac{C(\beta,n)}{C(\beta+1,n)} F^\pm_{\beta+1}
=
(2\beta+2)(2\beta+n) F^\pm_{\beta+1}
$$
and
\begin{align*}
\grad(F^\pm_{\beta+1})
&=
C(\beta+1,n)(\beta+1)\gamma^\beta \grad(\gamma)\\
&=
(\beta+1)\frac{C(\beta+1,n)}{C(\beta,n)}\grad(\gamma)\cdot F^\pm_\beta \\
&=
\frac{1}{2(2\beta+n)}\grad(\gamma)\cdot F^\pm_\beta.
\end{align*}
Moreover, for $A\in\OO(n-1,1)$, we find
$$
A^*F^\pm_\beta = A^*\gamma^* f^\pm_\beta = (\gamma\circ A)^* f^\pm_\beta = \gamma^* f^\pm_\beta = F^\pm_\beta.
$$
Furthermore, for $\Re(\beta)>2$, using these relations, $\div(\grad(\gamma))=-2n$, and $\ml\grad\gamma,\grad\gamma\mr = -4\gamma$ we find
\begin{align*}
\Box F^\pm_{\beta+1} 
&=
-\div(\grad F^\pm_{\beta+1}) \\
&=
-\frac{1}{2(2\beta+n)}\div(\grad(\gamma)F^\pm_{\beta}) \\
&=
-\frac{1}{2(2\beta+n)}(\div(\grad(\gamma))F^\pm_{\beta}+\ml\grad\gamma,\grad F^\pm_\beta\mr) \\
&=
\frac{1}{2(2\beta+n)}\big(2nF^\pm_{\beta}-\big\ml\grad\gamma,\tfrac{1}{2(2\beta-2+n)}\grad(\gamma)\cdot F^\pm_{\beta-1}\big\mr\big) \\
&=
\frac{1}{2(2\beta+n)}(2nF^\pm_{\beta}+\tfrac{2}{2\beta-2+n}\gamma\cdot F^\pm_{\beta-1}) \\
&=
\frac{1}{2(2\beta+n)}(2nF^\pm_{\beta}+2\cdot 2\beta F^\pm_{\beta}) \\
&=
F^\pm_{\beta}.
\end{align*}
Using \eqref{beziehung1} one obtains
\begin{align}
F^+_\beta + F^-_\beta 
&= C(\beta,n)(\gamma - i0)^\beta + C(\beta,n)(\gamma + i0)^\beta  \notag\\
&= 2 C(\beta,n)   \gamma^* t_+^\beta +  2 C(\beta,n) \cos(\pi \beta)  \gamma^* t_-^\beta  \notag\\
&= \left( R_\beta^+ +  R_\beta^- \right) +  \cos(\pi \beta)  \tilde R_\beta .
\label{eq:FFRRR}
\end{align}
Similarly, using \eqref{beziehung2} gives
\begin{align*}
\rmi \left( F^+_\beta - F^-_\beta \right) &= \rmi \left( C(\beta,n)(\gamma - i0)^\beta - C(\beta,n)(\gamma + i0)^\beta \right) \\
&=  2  C(\beta,n) \cos(\pi \beta)  \gamma^* t_-^\beta  = \sin(\pi \beta)  \tilde R_\beta.
\end{align*}

Thus, relations \eqref{F:mult}--\eqref{F:box} hold for $\beta$ in the half plane given by $\Re(\beta)>2$.
Relation~\eqref{F:box} can be used to extend the family $\beta\mapsto F^\pm_\beta$ to a holomorphic family defined for all $\beta\in\C$.
By unique continuation, this extension is unique and relations \eqref{F:mult}--\eqref{F:box} hold on all of $\C$.
As to relation~\eqref{F:invariance}, note that $\Box$ commutes with Lorentz transformations.

The set of zeros of the entire function $\beta\mapsto C(\beta,n)$ is given by 
$$
\{-1,-2,\ldots\}\cup\{-\tfrac{n}{2},-\tfrac{n}{2}-1,\ldots\}.
$$
For $\beta$ outside this set, the supports, singular supports and wavefront sets of $F^\pm_\beta$ coincide with those of $(\gamma\mp i0)^\beta$.
This shows \eqref{F:traeger1}.

If $k\in\N_0$ and $\Re(\beta)>k$ then $f^\pm_\beta$ is $C^k$ and vanishes to $k^\mathrm{th}$ order at $0$.
Thus, $\gamma^*f^\pm_\beta=f^\pm_\beta\circ\gamma$ is $C^k$ and vanishes to $k^\mathrm{th}$ order along $\gamma^{-1}(0)=\CC$ since $\gamma$ is smooth.
This proves \eqref{F:Ck}.

Since $F^\pm_0$ equals the constant function $1$, (\ref{F:box}) implies that $F^\pm_{\beta} = \Box^{-\beta} 1 = 0$ if $\beta$ is a negative integer. 
This proves (\ref{F:nullsein}). 

On $\R^n\setminus\CC$ the distributions $(\gamma\pm i0)^\beta$ are smooth.
Thus, whenever $C(\beta,n)$ vanishes then $F^\pm_\beta$ vanishes on $\R^n\setminus\CC$.
Hence, $\supp(F^\pm_\beta)\subset\CC$.
Since $\gamma$ is invariant under pull back by a Lorentz transformation of $\R^n$ so is $F^\pm_\beta$.
The only nonempty Lorentz-invariant closed subsets of $\CC$ are $\{0\}$ and $\CC$ itself.
Thus, $\supp(F^\pm_\beta)=\{0\}$ or $\supp(F^\pm_\beta)=\CC$ and \eqref{F:traeger2} follows.

The inclusion \eqref{F:wavefront} for the wavefront set follows from the inclusion \eqref{giwavefront} for the wavefront set of $(\gamma \pm \rmi 0)^\beta$.

In order to show \eqref{F:delta} note first that $R^\pm_{-\frac{n}{2}}=\delta_0$.
If $n$ is odd, then $\tilde R_{-\frac{n}{2}}=0$.
In this case we immediately obtain from \eqref{F:realpart} and \eqref{F:impart} that $F^\pm_{-\frac{n}{2}}=\delta_0$.
If $n$ is even, we have $\tilde R_{-\frac{n}{2}}=(-1)^{\frac{n}{2}-1} \delta_0$.
Then, by   \eqref{F:realpart}, we have $F^+_{-\frac{n}{2}} + F^-_{-\frac{n}{2}} = 2 \delta_0 + 2 (-1)^{\frac{n}{2}-1} \cos( \frac{ \pi n}{2}) \delta_0 =0$.
Relation~\eqref{F:impart} becomes $\rmi \left( F^+_{-\frac{n}{2}}-F^+_{-\frac{n}{2}}\right)=0$.
Both equations together imply $F^\pm_{-\frac{n}{2}}=0$.
\end{proof}

Let $\Lambda\in\R$ be an arbitrary but fixed constant.
We put 
\begin{equation}
G_\beta^{\pm}:= \mp \tfrac{\rmi}{\pi} \tfrac{d}{d\beta} F^\pm_\beta  + \Lambda F^\pm_\beta.
\label{eq:DefGbeta}
\end{equation}
Note that the explicit formula of $G_\beta^{\pm}$ involves derivatives $\tfrac{d}{d\beta} C(\beta,n)$. 
When $\beta$ is a negative integer, this can be computed more explicitly:

\begin{lemma}\label{lem:dCdbeta}
Let $n$ be an even positive integer.
\begin{enumerate}[(a)]
\item\label{dCbetanegative}
If $\beta$ is a negative integer then 
$$
\frac{d}{d \beta} C(\beta,n) 
= 
\frac{(-1)^{\beta +1} 2^{-n-2\beta} \pi^{\frac{2-n}{2}} (-\beta -1)! }{\left(\beta+\frac{n-2}{2}\right)!}.
$$
\item\label{dCbetanonnegative}
If $\beta$ is a nonnegative integer then
$$
\frac{d}{d \beta} C(\beta,n) 
= 
\left(2\gamma - \log(4) - \sum_{\ell=1}^{\beta+\tfrac{n}2-1} \frac1\ell  - \sum_{\ell=1}^{\beta} \frac1\ell\right)\cdot \frac{2^{-n-2\beta}\pi^{\frac{2-n}{2}}}{(\beta+\frac{n-2}{2})!\beta!}
$$
where $\gamma$ is the Euler-Mascheroni constant.
\end{enumerate}
\end{lemma}

\begin{proof}
First observe
\begin{align*}
\frac{d}{d \beta} C(\beta,n)
&=
\frac{d}{d \beta} \log(C(\beta,n)) \cdot C(\beta,n) \\
&=
-(\log(4) + \Psi(\beta+\tfrac{n}2) + \Psi(\beta+1))\cdot \frac{2^{-n-2\beta}\pi^{\frac{2-n}{2}}}{\Gamma(\beta+\frac{n}{2})\Gamma(\beta+1)}.
\end{align*}
Here $\Psi$ denotes the logarithmic derivative of the $\Gamma$-function, also known as digamma function.
At nonnegative integers, all terms can be evaluated explicitly which leads to \eqref{dCbetanonnegative}.

At negative integers, $\Gamma(\beta+1)$ has a pole so that the fraction vanishes.
Hence, the $\log(4)$-term does not contribute.
If $k\in\N_0$ then
$$
\lim_{\beta\to -k} \frac{\Psi(\beta)}{\Gamma(\beta)} 
=
\lim_{\beta\to -k} \frac{\Gamma'(\beta)}{\Gamma(\beta)^2} 
=
\frac{-1}{\res_{\beta=-k}\Gamma(\beta)}
= 
(-1)^{k+1} k!.
$$
Moreover,
$$
\lim_{\beta\to -k}\frac{\Psi(\beta+\tfrac{n}2)}{\Gamma(\beta+\tfrac{n}2)\Gamma(\beta+1)} = 0
$$
because $\frac{\Psi(\beta+\tfrac{n}2)}{\Gamma(\beta+\tfrac{n}2)}$ has a finite limit.
So the $\Psi(\beta+\tfrac{n}2)$-term does not contribute either.
Finally, 
\begin{align*}
\lim_{\beta\to -k}\frac{\Psi(\beta+1)}{\Gamma(\beta+\tfrac{n}2)\Gamma(\beta+1)}
&=
\frac{(-1)^k(k-1)!}{(-k+\tfrac{n-2}2)!}.
\end{align*}
This proves \eqref{dCbetanegative}.
\end{proof}

By construction, $G_\beta^{\pm}$ is a holomorphic family of distributions on $\R^n$ defined on the complex plane.

\begin{prop}\label{G:properties}
Let $n\in\N$ be even.
Then the holomorphic family $\beta \mapsto G^\pm_\beta$ of distributions on $\R^n$ satisfies
\begin{enumerate}[(a)]
\item\label{G:mult}
$\gamma\cdot G^\pm_\beta = (2\beta+2)(2\beta+n) G^\pm_{\beta+1} \mp \frac{\rmi }{\pi}(8 \beta + 4 + 2n) F^\pm_{\beta+1}$;
\item\label{G:grad}
$(\grad\gamma)\cdot G^\pm_{\beta} = 2(2\beta+n)\grad (G^\pm_{\beta+1}) \mp \frac{4\rmi}{\pi} \grad (F^\pm_{\beta+1})$;
\item\label{G:invariance}
$G^\pm_\beta$ is $\OO(n-1,1)$-invariant;
\item\label{G:box}
$\Box G^\pm_{\beta+1}=G^\pm_{\beta}$;
\item\label{G:traeger1}
for every $\beta\in\C\setminus\{-1,-2,\ldots\}$ we have
\begin{enumerate}[(i)]
\item 
$\supp(G^\pm_\beta)=\R^n$;
\item
 $\singsupp(G^\pm_\beta)=\CC$;
\end{enumerate}
\item\label{G:traeger2}
for $\beta\in\{-1,-2,\ldots,-\frac{n}{2}+1\}$ we have 
$$
\singsupp(G^\pm_\beta)=\CC;
$$
\item\label{G:traeger3}
for $\beta\in\{-\frac{n}{2},-\frac{n}{2}-1,-\frac{n}{2}-2,\ldots\}$ we have 
$$
\singsupp(G^\pm_\beta)=\{0\};
$$
\item\label{G:Ck}
if $k\in\N_0$ and $\Re(\beta)>k$ then $G^\pm_\beta$ is $C^k$ and vanishes to $k^\mathrm{th}$ order along $\CC$;
\item \label{G:wavefront}
$\WF(G^\pm_\beta)\subset \{\lambda d\gamma(x) \mid x\in \CC\setminus\{0\},\, \mp\lambda>0\} \cup \dot{T}_0^*\R^n$;
\item \label{G:sum}
if $\beta \in \mathbb{Z}$ we have
$ G_\beta^+ + G_\beta^- =  R^+_\beta + R^-_\beta$;
\item\label{G:delta}
for the value $\beta=-\frac{n}{2}$ we have $G^\pm_{-n/2}=\delta_0$.
\end{enumerate}
\end{prop}

\begin{proof}
Relations \eqref{G:mult}--\eqref{G:box} follow by simply differentiating \eqref{F:mult}--\eqref{F:box} in Proposition~\ref{F:properties} with respect to $\beta$.
 
For $\beta$ outside the zero set of $C(\beta,n)$, the supports and singular supports of $G^\pm_\beta$ coincide with those of certain linear combinations of $(\gamma \mp i0)^\beta \log(\gamma \mp i0)$ and $(\gamma \mp i0)^\beta $.
This shows \eqref{G:traeger1}. 

On the zero set of $C(\beta,n)$ the support is contained in the light cone. 
In the same way as in the proof of Proposition~\ref{F:properties} Lorentz invariance now implies \eqref{G:traeger2}. 

Relation~\eqref{G:traeger3} is a simple consequence of \eqref{G:delta} and \eqref{G:box}.

Assertion~\eqref{G:Ck} follows as in the proof of Proposition~\ref{F:properties}.

In case $\beta$ is not a negative integer, the inclusion \eqref{G:wavefront} follows from the inclusions \eqref{giwavefront} and  \eqref{giwavefrontlog} for the wavefront set of $(\gamma \pm \rmi 0)^\beta$ and $(\gamma \pm \rmi 0)^\beta \log(\gamma \pm \rmi 0)$, respectively. 
Relation~\eqref{G:box} now implies that $\WF(G^\pm_\beta) \subset \WF(G^\pm_{\beta+1})$ and this shows the inclusion for general $\beta$.

To show \eqref{G:sum} we differentiate Proposition~\ref{F:properties}~\eqref{F:impart} and get for $\beta \in \mathbb{Z}$ the formula
$$
F_\beta^++F_\beta^- - ( G^+_\beta + G^-_\beta) = (-1)^\beta \tilde R_\beta.
$$
Together with \eqref{eq:FFRRR} this implies \eqref{G:sum}.

It remains to show \eqref{G:delta}.
By \eqref{G:grad} and Proposition~\ref{F:properties}~\eqref{F:nullsein}, we have $x_j G^\pm_{-n/2} =0$.
This implies that $G^\pm_{-n/2}$ is a multiple $c^\pm_n \delta_0$ of the delta distribution $\delta_0$. 
By \eqref{G:sum} we have 
\begin{align*}
 G^+_{-n/2} + G^-_{-n/2} =  2 \delta_0.
\end{align*}
This implies that $\Re(c^\pm_n)=1$. In order to show that the imaginary parts vanish we differentiate  \eqref{F:realpart} of Proposition~\ref{F:properties} and find that if $\alpha$ is a nonnegative integer that
\begin{align*}
 \rmi\left( G^+_{\alpha} - G^-_{\alpha} \right)
 =
 \tfrac{2}{\pi}\gamma^\alpha \big( C(\alpha,n)  \log(| \gamma |)  +  \tfrac{d}{d\beta}C(\beta,n)|_{\beta=\alpha} \big) ;
\end{align*}
In particular, for $\alpha=0$ we get 
\begin{align*}
  \rmi\left( G^+_{0} - G^-_{0} \right) =   c_1 \left(  \log(| \gamma |) \right) + c_2,
\end{align*}
and we therefore conclude that $2\, \Im(c^\pm_n) \, \delta_0 =\rmi( G^+_{-n/2} - G^-_{-n/2}) = c_1 \Box^{\frac{n}{2}}(\log(| \gamma |))$.
It is known that  in even dimensions we have $\Box^{\frac{n}{2}} \left(  \log(| \gamma |) \right) =0$, see \cite{rham1958solution}*{p.~366}, implying that $\Im(c^\pm_n)=0$ and therefore $c^\pm_n=1$. 
\end{proof}

\subsection{Turning to manifolds} 
\label{Hadamardappendix1}

Now let $(X,g)$ be an $n$-dimensional oriented and time-oriented Lorentzian manifold. 
Denote by $\pi_{TX}:TX\to X$ be tangent bundle.
Let $\pi_P:P\to X$ be the bundle of Lorentz-orthonormal tangent frames.
This is an $\OO(n-1,1)$-principal bundle.
Let $\omega$ be the solder form on $P$.
This is a horizontal $\R^n$-valued $1$-form on the total space of $P$, see \cite{MR1852066}*{Sec.~6.1.1}.

Now let $Q\in\DD'(\R^n)$ be an $\OO(n-1,1)$-invariant distribution such as $F^\pm_\beta$, $G^\pm_\beta$, or $\tilde R_\beta$.
Let $\Omega\subset X$ be an open subset and $s:\Omega\to P$ a local section.
Then $\omega\circ ds:T\Omega\to \R^n$ maps each tangent $T_xX=T_x\Omega$ for $x\in\Omega$ isometrically onto Minkowski space $\R^n$.
In particular, $\omega\circ ds$ is a submersion, and we can pull back $Q$ to obtain a distribution $Q_\Omega := (\omega\circ ds)^*Q$ defined on $T\Omega=TX|_\Omega$.

Now $Q_\Omega$ is independent of the particular choice of section $s$.
Namely, if $s':\Omega\to P$ is another section then there is a unique smooth map $a:\Omega\to \OO(n-1,1)$ such that $s'=a\circ s$.
From the equivariance property of the solder form, $\omega\circ dR_a= a^{-1}\circ\omega$ where $R_a$ denotes the right action of $a\in\OO(n-1,1)$ on $P$, and the invariance of $Q$ we find
 \begin{align*}
(\omega\circ ds')^*Q
=
(\omega\circ dR_a\circ ds)^*Q
=
(a^{-1}\circ\omega\circ ds)^*Q
=
(\omega\circ ds)^*(a^{-1})^*Q
=
(\omega\circ ds)^*Q.
\end{align*}
In particular, two such pullbacks $Q_\Omega$ and $Q_{\Omega'}$ coincide on $\Omega\cap\Omega'$.
Thus, there is a unique distribution $\tilde Q$ on $TX$ such that $\tilde Q|_\Omega=Q_\Omega$ for all local sections of $P$ defined on $\Omega$.

Denote the Riemannian exponential map by $\exp$.
Then $(\exp,\pi_{TX}):TX\to X\times X$ restricts to a diffeomorphism from a neighborhood $\UU'$ of the zero-section of $TX$ onto a neighborhood $\UU$ of the diagonal in $X\times X$.
W.l.o.g.\ we assume that each ``slice'' $\UU'\cap T_xX$ is a convex neighborhood of the origin in $T_xX$.

There is a corresponding smooth ``distortion function'' $\mu:\UU\to\R^+$ characterized by $\vol_{X\times X}=\mu\cdot(\exp,\pi_{TX})_*\vol_{TX}$.
Here $\vol_{X\times X}$ denotes the volume measure on $X\times X$ and $\vol_{TX}$ the one on $TX$, both induced by the Lorentzian metric on $X$.

Denote by $\CC_\UU$ the image of the set of lightlike vectors in $\UU'$ under $(\exp,\pi_{TX})$. 
The elements are pairs of points $(x,y)$ which can be joined by a lightlike geodesic.
More precisely, $y=\exp_x(v)$ for some lightlike tangent vector $v\in T_xX$ and the convexity assumption on $\UU'$ yields $(x,\exp_x(tv))\in\CC_\UU$ for every $t\in[0,1]$.

We now define a distribution $Q^\UU$ on $\UU$ by $Q^\UU := \mu\cdot((\exp,\pi_{TX})|_{\UU'})_*(\tilde Q|_{\UU'})$.
The appearance of the distortion factor $\mu$ is explained by the following observation:
if $\tilde Q$ is given by an $L^1_\loc$-function $\tilde q$ then we find for any test function $\phi\in C_\mathrm{c}^\infty(\UU)$:
\begin{align*}
Q^\UU[\phi]
&=
(\exp,\pi_{TX})|_{\UU'})_*(\tilde Q|_{\UU'})[\mu\cdot\phi] \\
&=
\tilde Q[(\mu\circ(\exp,\pi_{TX}))\cdot(\phi\circ(\exp,\pi_{TX})))] \\
&=
\int_{\UU'} \tilde q \cdot (\phi\circ(\exp,\pi_{TX})) \cdot (\mu\circ(\exp,\pi_{TX})) \, d\vol_{TX} \\
&=
\int_\UU (\tilde q\circ(\exp,\pi_{TX})^{-1}) \cdot \phi \cdot \mu \, d((\exp,\pi_{TX})_*\vol_{TX}) \\
&=
\int_\UU (\tilde q\circ(\exp,\pi_{TX})^{-1}) \cdot \phi \, d\vol_{X\times X}.
\end{align*}
Thus, $Q^\UU$ is then given by the function $\tilde q\circ(\exp,\pi_{TX})^{-1} = ((\exp,\pi_{TX})^{-1})^*\tilde q$, without any distortion factor.
If we do not want to distinguish between a function and the corresponding distribution, we will also write $Q^\UU = ((\exp,\pi_{TX})^{-1})^*\tilde Q|_{\UU'}$.

We define $\Gamma:\UU\to\R$ by $\Gamma(x,y) := \gamma(\exp_y^{-1}(x))$.
Then one checks, see \cite{BGP07}*{Lemma~1.3.19}:
\begin{align}
\<\grade\Gamma,\grade\Gamma\> &= -4\Gamma, \label{eq:gradGamma}\\
\Boxz\Gamma - 2n = -\<\grade\Gamma, &\,\grade(\log\circ\mu)\>. \label{eq:BoxGamma}
\end{align}
Here we have used the notation $\<\cdot,\cdot\>=g$.
The subindex $(1)$ indicates that the differential operator is applied to the first variable of $\Gamma$ while keeping the second one fixed.

\begin{prop}\label{FGOmega}
The families $\beta\mapsto R^{\pm,\UU}_\beta$, $\beta\mapsto F^{\pm,\UU}_\beta$ and $\beta\mapsto G^{\pm,\UU}_\beta$ are holomorphic families of distributions on $\UU$, defined for all $\beta\in\C$.
The following holds for all $\beta\in\C$:
\begin{enumerate}[(a)]
\item\label{FG}\abovedisplayskip=-10pt
\begin{align*}
G_\beta^{\pm,\UU} = \mp \tfrac{\rmi}{\pi} \tfrac{d}{d\beta} F^{\pm,\UU}_\beta  + \Lambda F^{\pm,\UU}_\beta;
\end{align*}
\item\label{GammaFG}
\begin{align*}
\Gamma\cdot R^{\pm,\UU}_\beta &= (2\beta+2)(2\beta+n) R^{\pm,\UU}_{\beta+1}, \\
\Gamma\cdot F^{\pm,\UU}_\beta &= (2\beta+2)(2\beta+n) F^{\pm,\UU}_{\beta+1}, \\
\Gamma\cdot G^{\pm,\UU}_\beta &= (2\beta+2)(2\beta+n) G^{\pm,\UU}_{\beta+1} \mp \tfrac{\rmi}{\pi}(8 \beta + 4 + 2n) F^{\pm,\UU}_{\beta+1};
\end{align*}
\item\label{gradFG}
\begin{align*}
(\grade\Gamma)\cdot R^{\pm,\UU}_{\beta} &= 2(2\beta+n)\grade (R^{\pm,\UU}_{\beta+1}), \\
(\grade\Gamma)\cdot F^{\pm,\UU}_{\beta} &= 2(2\beta+n)\grade (F^{\pm,\UU}_{\beta+1}), \\
(\grade\Gamma)\cdot G^{\pm,\UU}_{\beta} &= 2(2\beta+n)\grade (G^{\pm,\UU}_{\beta+1}) \mp \tfrac{4\rmi}{\pi} \grade (F^{\pm,\UU}_{\beta+1});
\end{align*}
\item\label{boxFG}
if $\beta\not=-\tfrac{n}{2}$ then

\begin{align*}
\Boxz R^{\pm,\UU}_{\beta+1} 
&= 
\left( \tfrac{\Boxz\Gamma-2n}{2( n+2\beta)} +1 \right)R^{\pm,\UU}_\beta, \\
\Boxz F^{\pm,\UU}_{\beta+1} 
&= 
\left( \tfrac{\Boxz\Gamma-2n}{2( n+2\beta)} +1 \right)F^{\pm,\UU}_\beta, \\
\Boxz G^{\pm,\UU}_{\beta+1} 
&= 
\left( \tfrac{\Boxz\Gamma-2n}{2 (n+2\beta)} +1 \right) G^{\pm,\UU}_\beta  \pm \tfrac{\rmi}{\pi} \tfrac{\Boxz\Gamma-2n}{(n+2\beta)^2} F^{\pm,\UU}_\beta ;
\end{align*}
\item\label{FGFglatt}
if $\beta\in \Z$ then the distribution $F^{\pm,\UU}_{\beta}$ is a smooth function and if $\beta\in\{-1,-2,\ldots\}$ then $F^{\pm,\UU}_{\beta}=0$;
\item\label{CkFG}
if $k\in\N_0$ and $\Re(\beta)>k$, then $R^{\pm,\UU}_{\beta}$, $F^{\pm,\UU}_{\beta}$ and $G^{\pm,\UU}_{\beta}$ are $C^k$ and vanish to $k^\mathrm{th}$ order along $\CC_\UU$;
\item\label{Rsupp}
if $n$ is even and $\beta\in\{-1,-2,\ldots,1-\frac{n}{2}\}$ then

$$
\supp(R^{+,\UU}_\beta + R^{-,\UU}_\beta) = \singsupp(R^{+,\UU}_\beta + R^{-,\UU}_\beta) = \CC_\UU;
$$
\item\label{FGWF}\abovedisplayskip=-12pt
\begin{align*}
\WF'(F^{\pm,\UU}_\beta) &\subset \Lambda_\Feyn^\pm \quad\mbox{ and }\quad
\WF'(G^{\pm,\UU}_\beta) \subset \Lambda_\Feyn^\pm;
\end{align*}
\item\label{FG0}
for the diagonal embedding $\diag:X\to X\times X$ we have

$$
\diag_*\vol_X = 
\begin{cases}
R^{\pm,\UU}_{-\frac{n}{2}} & \text{ for all $n$,} \\
F^{\pm,\UU}_{-\frac{n}{2}} & \text{ if $n$ is odd,} \\
G^{\pm,\UU}_{-\frac{n}{2}} & \text{ if $n$ is even;}
\end{cases}
$$
\item\label{FGRgamma}\abovedisplayskip=0pt
over $\UU\setminus\diag(X)$ we have

\begin{gather*}
R^{+,\UU}_\beta + R^{-,\UU}_\beta = 2C(\beta,n)\mu\Gamma^*(t_+^\beta), \\
F^{\pm,\UU}_\beta = C(\beta,n) \mu\Gamma^*(f^\mp_\beta) , \\
G^{\pm,\UU}_\beta 
=
\mu\cdot \Gamma^*\big((\mp\tfrac{\rmi}{\pi}\tfrac{dC(\beta,n)}{d\beta}+\Lambda C(\beta,n))f^\mp_\beta \mp\tfrac{\rmi}{\pi}C(\beta,n)h^\pm_\beta\big).
\end{gather*}
\end{enumerate}
\end{prop}

\begin{proof}
Since the statements on the Riesz distributions $R^{\pm,\UU}_\beta$ can be found in \cite{BGP07}*{Prop.~1.4.2} we only do the proof for the families $F^{\pm,\UU}_\beta$ and $G^{\pm,\UU}_\beta$.

Assertion~\eqref{FG} follows directly from the definition in \eqref{eq:DefGbeta}.

Relation~\eqref{GammaFG} follows from Proposition~\ref{F:properties}~\eqref{F:mult} and Proposition~\ref{G:properties}~\eqref{G:mult}, respectively, together with $\gamma\circ\omega\circ ds = \Gamma\circ (\exp,\pi_{TX})$.

As to in \eqref{gradFG}, we assume that $\beta\in\mathbb{C}$ with $\Re(\beta)>0$.
Since $F^{\pm,\UU}_{\beta+1}$ is then $C^1$, we can compute $dF^{\pm,\UU}_{\beta+1}$ classically.
The formula will then follow for all $\beta\neq -\frac{n}{2}$ by analyticity.

For the second equation in \eqref{gradFG}, we rewrite Proposition~\ref{F:properties}~\eqref{F:grad} as $d\gamma\cdot F^\pm_\beta=2(2\beta+n)d F^\pm_{\beta+1}$.
Applying $(\omega\circ ds)^*$ and using that pullback commutes with exterior differentiation yields
$$
d(\gamma\circ\omega\circ ds)\cdot \tilde F^\pm_\beta
=
2(2\beta+n)d \tilde F^\pm_{\beta+1} .
$$
The left-hand side can be rewritten as 
$$
d(\Gamma\circ (\exp,\pi_{TX}))\cdot (\exp,\pi_{TX})^* F^{\pm,\UU}_\beta
=
(\exp,\pi_{TX})^*(d\Gamma\cdot F^{\pm,\UU}_\beta)
$$
while the right-hand side yields
$$
2(2\beta+n)(\exp,\pi_{TX})^*dF^{\pm,\UU}_{\beta+1}.
$$
Therefore, we have
$$
d\Gamma\cdot F^{\pm,\UU}_\beta=2(2\beta+n)dF^{\pm,\UU}_{\beta+1}.
$$
Restricting to the first argument and converting the differentials into gradients using the musical isomorphism given by the Lorentzian metric proves the second equation in \eqref{gradFG}.
The third equation in \eqref{gradFG} follows by differentiating the second one with respect to $\beta$.

We show the second equation in \eqref{boxFG}.
Let $\beta\in\mathbb{C}$ with $\Re(\beta)>1$.
Since $F^{\pm,\UU}_{\beta+1}$ is then $C^2$, we can compute $\Box F^{\pm,\UU}_{\beta+1}$ classically.
Again, the formula will then follow for all $\beta\neq -\frac{n}{2}$ by analyticity.
Indeed, using \eqref{gradFG}, \eqref{eq:gradGamma}, and \eqref{GammaFG}, we get
\begin{align*}
\Boxz F^{\pm,\UU}_{\beta+1}
&=
-\divz\big(\grade F^{\pm,\UU}_{\beta+1}\big)\\
&=
-\tfrac{1}{2(2\beta+n)}\,\divz\big(F^{\pm,\UU}_\beta \cdot\grade\Gamma\big)\\
&=
\tfrac{1}{2(2\beta+n)}\,\Boxz \Gamma\cdot F^{\pm,\UU}_\beta-\tfrac{1}{2(2\beta+n)}\, \langle\grade\Gamma,\grade F^{\pm,\UU}_\beta\rangle\\
&=
\tfrac{1}{2(2\beta+n)}\,\Boxz \Gamma\cdot F^{\pm,\UU}_\beta- \tfrac{1}{2(2\beta+n)\cdot 2(2\beta-2+n)}\,\langle\grade\Gamma,\grade\Gamma\cdot F^{\pm,\UU}_{\beta-1}\rangle \\
&=
\tfrac{1}{2(2\beta+n)}\,\Boxz \Gamma\cdot F^{\pm,\UU}_\beta+\tfrac{1}{(2\beta+n)(2\beta-2+n)}\,\Gamma\cdot F^{\pm,\UU}_{\beta-1}\\
&=
\tfrac{1}{2(2\beta+n)}\,\Boxz \Gamma\cdot F^{\pm,\UU}_\beta+\tfrac{2\beta(2\beta-2+n)}{(2\beta+n)(2\beta-2+n)} \,F^{\pm,\UU}_\beta \\
&=
\left(\tfrac{\Boxz\Gamma-2n}{2(2\beta+n)}+1\right)F^{\pm,\UU}_\beta.
\end{align*}
The third equation in \eqref{boxFG} again follows by differentiating the second one with respect to $\beta$.

Assertion~\eqref{FGFglatt} follows directly from Proposition~\ref{F:properties}~\eqref{F:traeger1} and \eqref{F:nullsein}.

Statement~\eqref{CkFG} follows from Proposition~\ref{F:properties}~\eqref{F:Ck} and Proposition~\ref{G:properties}~\eqref{G:Ck} together with $(\omega\circ ds \circ (\exp,\pi_{TX})^{-1})^{-1}(\CC)=\CC_\UU$.

As to the wavefront set, recall from Proposition~\ref{F:properties}~\eqref{F:wavefront} that $\WF(F^\pm_\beta)\subset \{\lambda d\gamma(x) \mid x\in \CC\setminus\{0\},\, \mp\lambda>0\} \cup \dot{T}_0^*\R^n$.
It is easily checked that covectors in $(\omega\circ ds)^*\dot{T}_0^*\R^n$ vanish on tangent vectors to $T\Omega$ which are tangent to the zero-section.
For dimensional reasons we must then have for each $x\in\Omega$
$$
(\omega\circ ds(x))^*\dot{T}_0^*\R^n
=
\{\xi\in \dot{T}_{\oo(x)}^*TX \mid  \xi|_{T_{\oo(x)}\oo(X)}=0\}.
$$
Similarly, let $\xi\in T^*X$ and let $v\in T_{\oo(x)}TX$ be tangent to the zero-section.
Then we can write $v=\dot{c}(0)$ with $c(t)=\oo(\pi_{TX}(c(t))$.
Thus, $\exp(c(t))=\pi_{TX}(c(t))$ and hence $d\exp(v))=d\pi_{TX}(v)$
Therefore
$$
(\exp,\pi_{TX})^*(\xi,-\xi)(v)
=
\xi(d\exp(v)) - \xi(d\pi_{TX}(v))
=
0.
$$
Thus, the pullback of the anti-diagonal in $\dot{T}(X\times X)$ along $(\exp,\pi_{TX})$ is also contained in $\{\xi\in \dot{T}_{\oo(x)}^*TX \mid  \xi|_{T_{\oo(x)}\oo(X)}=0\}$ and, again for dimensional reasons, we have equality.
This shows
$$
(\omega\circ ds \circ (\exp,\pi_{TX})^{-1}(x,x))^*\dot{T}_0^*\R^n 
=
\{(\xi,-\xi) \mid \xi\in \dot{T}^*_xX\}.
$$
The vector $\omega\circ ds  \circ (\exp,\pi_{TX})^{-1}(x,y)$ being lightlike means that $x$ and $y$ can be joined by lightlike geodesic.
From $\Gamma=\gamma\circ \omega\circ ds \circ (\exp,\pi_{TX})^{-1}$ we see that 
$$
(\omega\circ ds \circ (\exp,\pi_{TX})^{-1})^*(\lambda d\gamma) 
=
\lambda d\Gamma
=
(\lambda d_{(1)}\Gamma, \lambda d_{(2)}\Gamma).
$$
By the Gauss lemma, $-d_{(2)}\Gamma(x,y)$ is the parallel translate of $d_{(1)}\Gamma(x,y)$ along the geodesic joining $x$ and $y$, see \cite{ON}*{Cor.~3 on p.~128}.
In terms of the geodesic flow $\Phi_t$ on the cotangent bundle this means for $(\xi,\xi')=(\lambda d_{(1)}\Gamma, \lambda d_{(2)}\Gamma)$:
$$
\Phi_{\frac{1}{2\lambda}}(\xi)
=
-\Phi_{\frac{1}{2\lambda}}(-\lambda d_{(1)}\Gamma)
=
-2\lambda \Phi_1(-\tfrac12d_{(1)}\Gamma)
=
-2\lambda \tfrac12d_{(2)}\Gamma
=
-\xi' .
$$
Summarizing, this yields
\begin{align*}
\WF(F^{\pm,\UU}_\beta)
&=
\WF((\omega\circ ds \circ (\exp,\pi_{TX})^{-1})^* F^{\pm}_\beta) \\
&\subset
(\omega\circ ds \circ (\exp,\pi_{TX})^{-1})^* \WF(F^\pm_\beta) \\
&\subset
(\omega\circ ds \circ (\exp,\pi_{TX})^{-1})^* (\dot{T}_0^*\R^n \cup  \{\lambda d\gamma(x) \mid x\in \CC\setminus\{0\},\, \mp\lambda>0\}) \\
&\subset
\{(\xi,-\xi) \mid \xi\in \dot{T}^*_xX\} \cup \{(\xi,\xi') \in \dot T^*(X \times X) \mid \xi \textrm{ is lightlike,}\; \exists\, t>0: \Phi_{\mp t}(\xi) = -\xi' \}.
\end{align*}
This proves the first statement in \eqref{FGWF}.
The second one follows in the same way.

We show \eqref{FG0}.
From Proposition~\ref{F:properties}~\eqref{F:delta} and \ref{G:properties}~\eqref{G:delta} we know that $F^\pm_{-n/2}=\delta_0$ if $n$ is odd and $G^\pm_{-n/2}=\delta_0$ if $n$ is even.
The preimage of $0$ under the solder form $\omega$ is the set of vertical vectors $V=\{w\in TP\mid d\pi_P(w)=0\}$.
Thus, for any local section $s:\Omega\to P|_\Omega$ the preimage $(\omega\circ ds)^{-1}(0)$ coincides with the image $\oo(\Omega)$ where $\oo:X\to TX$ is the zero-section.
Since the pullback of a distribution is given by fiber-integration of the test functions, we find for any test function $\phi$ on $TP|_\Omega$:
$$
(\delta_0)_\Omega [\phi]
=
((\omega\circ ds)^*\delta_0) [\phi]
=
\int_{(\omega\circ ds)^{-1}(0)} \phi\; d(\oo_*\vol_X)
=
\int_\Omega (\phi\circ\oo)\; d\vol_X .
$$
Thus, $\tilde\delta_0=\oo_*\vol_X$.
Since $\mu\equiv1$ along the diagonal, we get
$$
\delta_0^\UU
=
((\exp,\pi_{TX})|_{\UU'})_*\tilde\delta_0
=
((\exp,\pi_{TX})|_{\UU'})_*\oo_*\vol_X
=
((\exp,\pi_{TX})\circ\oo)_*\vol_X
=
\diag_*\vol_X.
$$
This proves \eqref{FG0}.

As to \eqref{FGRgamma}, note that $\Gamma|_{\UU\setminus\diag(X)}\colon \UU\setminus\diag(X) \to\R$ is a submersion so that the pull back of distributions makes sense.
We observe
$$
\Gamma\circ (\exp,\pi_{TX}) = \gamma\circ\omega\circ ds
$$
which follows from the definitions.
Now
\begin{align*}
R^{+,\UU}_\beta + R^{-,\UU}_\beta 
&=
\mu\cdot ((\exp,\pi|_{TX})^{-1})^* (\omega\circ ds)^*(R^+_\beta + R^-_\beta) \\
&=
2C(\beta,n)\mu\cdot ((\exp,\pi|_{TX})^{-1})^* (\omega\circ ds)^*\gamma^*(t^\beta_+) \\
&=
2C(\beta,n)\mu\cdot (\gamma\circ\omega\circ ds \circ (\exp,\pi|_{TX})^{-1})^*(t^\beta_+) \\
&=
2C(\beta,n)\mu\cdot \Gamma^*(t^\beta_+).
\end{align*}
Similarly, 
\begin{align*}
F^{\pm,\UU}_\beta 
&=
\mu\cdot ((\exp,\pi|_{TX})^{-1})^* (\omega\circ ds)^*(F^\pm_\beta)
=
C(\beta,n)\mu\Gamma^*(f^\pm_\beta)
\end{align*}
and
\begin{align*}
G^{\pm,\UU}_\beta 
&=
\mu\cdot ((\exp,\pi|_{TX})^{-1})^* (\omega\circ ds)^*(G^\pm_\beta) \\
&=
\mu\cdot (\omega\circ ds\circ(\exp,\pi|_{TX})^{-1})^*\big((\mp\tfrac{\rmi}{\pi}\tfrac{d}{d\beta}+\Lambda)F^\pm_\beta\big) \\
&=
\mu\cdot \Gamma^*\big((\mp\tfrac{\rmi}{\pi}\tfrac{d}{d\beta}+\Lambda)(C(\beta,n)f^\mp_\beta)\big) \\
&=
\mu\cdot \Gamma^*\big((\mp\tfrac{\rmi}{\pi}\tfrac{dC(\beta,n)}{d\beta}+\Lambda C(\beta,n))f^\mp_\beta \mp\tfrac{\rmi}{\pi}C(\beta,n)h^\pm_\beta\big).
\qedhere
\end{align*}
\end{proof}

Note that $\diag_*\vol_X$ is the Schwartz kernel of the identity operator.

\begin{example}
Let $X$ be Minkowski space.
Then we can choose $\UU=X\times X=\R^n\times \R^n$ and $\UU'=TX=\R^n\times \R^n$.
We have $\pi_{TX}(v,w)=v$ and $\exp(v,w)=v+w$.
Let $s$ be the global frame that assigns the standard orthonormal basis to each point.
Then $(\omega\circ ds)(v,w)=w$.
We conclude $(\exp,\pi_{TX})^{-1}(x,y)=(y,x-y)$ and $\omega\circ ds\circ(\exp,\pi_{TX})^{-1}(x,y)=x-y$. 

For $\Re(\beta)>0$, when our distributions are continuous functions, we can write $F^{\pm,X\times X}_\beta(x,y) = C(\beta,n)(\gamma(x-y)\mp \rmi 0)^\beta$.
The $\mp \rmi 0$ indicates how the $\beta^\mathrm{th}$ power is to be interpreted when $\gamma(x-y)$ is negative, i.e.\ when $x-y$ is spacelike.

On Minkowski space $\Boxz\Gamma=2n$ so that Proposition~\ref{FGOmega}~\eqref{boxFG} reduces to
$$
\Boxz F^{\pm,X\times X}_{\beta+1} = F^{\pm,X\times X}_\beta
\quad\mbox{ and }\quad
\Boxz G^{\pm,X\times X}_{\beta+1} = G^{\pm,X\times X}_\beta
$$
which now makes sense and also holds for $\beta=-\tfrac{n}{2}$.
Thus, $F^{+,X\times X}_{1-\frac{n}{2}}$ is a Feynman propagator if $n$ is odd and $G^{+,X\times X}_{1-\frac{n}{2}}$ is one if $n$ is even.
\end{example}

We also need the distributions 
\begin{equation*}
Q^{\pm,\UU}_\beta := \frac{d}{d\beta}R^{\pm,\UU}_\beta
\quad\mbox{ and }\quad
H^{\pm,\UU}_\beta := \frac12\frac{d}{d\beta}G^{\pm,\UU}_\beta ,
\end{equation*}
but only for $\beta=-\frac{n}{2}$.

\begin{lemma}\label{lem:QH}
The following holds:
\begin{enumerate}[(a)]
\item \label{QH1}\abovedisplayskip=-14pt
\begin{align*}
\Gamma\cdot Q^{\pm,\UU}_{-\frac{n}{2}} &= 2(2-n)R^{\pm,\UU}_{1-\frac{n}{2}} ,\\
\Gamma\cdot H^{\pm,\UU}_{-\frac{n}{2}} &= 2(2-n)G^{\pm,\UU}_{1-\frac{n}{2}}\quad\mbox{ if $n\ge4$ is even};
\end{align*}
\item \label{QH2}
\begin{align*}
4\,\grade R^{\pm,\UU}_{1-\frac{n}2} &= \grade\Gamma\cdot Q^{\pm,\UU}_{-\frac{n}2} ,\\
4\,\grade G^{\pm,\UU}_{1-\frac{n}2} &= \grade\Gamma\cdot H^{\pm,\UU}_{-\frac{n}2} \quad\mbox{ if $n\ge2$ is even}.
\end{align*}
\end{enumerate}
\end{lemma}

\begin{proof}
Assertion~\eqref{QH1} follows from differentiating Proposition~\ref{FGOmega}~\eqref{GammaFG} at $\beta=-\frac{n}{2}$, together with Proposition~\ref{FGOmega}~\eqref{FG} and \eqref{FGFglatt} for the second equation.
Similarly, assertion~\eqref{QH2} follows from differentiating Proposition~\ref{FGOmega}~\eqref{gradFG} at $\beta=-\frac{n}{2}$.
\end{proof}

\subsection{Hadamard series of the Feynman parametrix} 
\label{Hadamardappendix2}

Let $\SS\to X$ be a real or complex vector bundle and let $P$ be a normally hyperbolic operator acting on sections of $\SS$.
We write
$$
P = \Box^\nabla +B
$$
where $\nabla$ is the connection induced by $P$, see Remark~\ref{rem:dAlembert}.

Using the distributions $F^{\pm,\UU}_\beta$ and $G^{\pm,\UU}_\beta$ we construct formal Feynman parametrices over $\UU$.
We make the following formal series ansatz:
$$
\FpmU \coloneq 
\begin{cases}
\sum\limits_{k=0}^{\infty} V_{k} \cdot F^{\pm,\UU}_{1-\frac{n}2+k} & \quad\text{if $n$ is odd,} \\[10pt]
\sum\limits_{k=0}^{\infty} V_{k} \cdot G^{\pm,\UU}_{1-\frac{n}2+k} & \quad\text{if $n$ is even,} 
\end{cases}
$$
where $V_k\in C^\infty(\UU,\SS\boxtimes \SS^*)$ are smooth sections yet to be found. 
They are called \emph{Hadamard coefficients}.

\subsubsection{\texorpdfstring{Fundamental solution up to $C^N$-errors in odd dimensions}{Fundamental solution up to CN-errors in odd dimensions}}
We start with $\sum_{k=0}^{\infty} V_{k} \cdot F^{\pm,\UU}_{1-\frac{n}2+k}$ and $n$ odd.
We formally apply $P_{(1)}$ termwise and, using Proposition~\ref{FGOmega}~\eqref{GammaFG}--\eqref{boxFG}, we solve 
$$
P_{(1)}\Big(\sum_{k=0}^{\infty} V_{k} \cdot F^{\pm,\UU}_{1-\frac{n}2+k}\Big)=\diag_*\vol_X=F^{\pm,\UU}_{-\frac{n}2}
$$ 
by grouping together all terms containing the same $F^{\pm,\UU}_\beta$.
This leads to the transport equations
\begin{equation}
 \nabla_{\grade \Gamma} V_k - \left( \tfrac{1}{2}\Boxz \Gamma - n +2k \right) V_k = 2 k P_{(1)} V_{k-1}
\label{eq:transport}
\end{equation}
together with the initial condition $V_0(x,x) = \mathrm{id}_{\SS_x}$. 
Since the relations for $F^{\pm,\UU}_\beta$ in Proposition~\ref{FGOmega}~\eqref{GammaFG}--\eqref{boxFG} are identical with those for the Riesz distributions $R^\pm_\beta$, we get the same transport equations as in \cite{BGP07}*{Ch.~2}.
Hence, the Hadamard coefficients $V_k$ exist and are uniquely determined by the transport equations and coincide with those in \cite{BGP07}*{Ch.~2}.

Let $N\in\N_0$.
For $k>N+\tfrac{n}{2}-1$ the distribution $F^{\pm,\UU}_{1-\frac{n}{2}+k}$ is a $C^N$-function by Proposition~\ref{FGOmega}~\eqref{CkFG}.
Since $n$ is odd, we have $\lfloor N+\frac{n}{2}-1\rfloor=N+\frac{n-3}{2}$.
The transport equations yield
\begin{align*}
P_{(1)}\bigg(\sum_{k=0}^{N+\frac{n-3}{2}} V_{k} \cdot F^{\pm,\UU}_{1-\frac{n}2+k}\bigg) - \diag_*\vol_X
&=
\big(P_{(1)}V_{N+\frac{n-1}{2}}\big)\cdot F^{\pm,\UU}_{N+\frac{1}{2}} .
\end{align*}
Therefore, the finite sum $\sum_{k=0}^{N+\frac{n-3}{2}} V_{k} \cdot F^{\pm,\UU}_{1-\frac{n}2+k}$ defines a right fundamental solution up to $C^N$-errors.

\subsubsection{\texorpdfstring{Fundamental solution up to $C^N$-errors in even dimensions}{Fundamental solution up to CN-errors in even dimensions}}
Now let $n$ be even.
The relations in Proposition~\ref{FGOmega}~\eqref{GammaFG}--\eqref{boxFG} for the $G^{\pm,\UU}_\beta$ are again the same as for $F^{\pm,\UU}_\beta$ and $R^\pm_\beta$ up to extra terms involving $F^{\pm,\UU}_\beta$.
Since $n$ is even, these extra terms are smooth by Proposition~\ref{FGOmega}~\eqref{FGFglatt}.
Thus, again taking the same Hadamard coefficients as in \cite{BGP07}*{Ch.~2}, the same reasoning as above yields
\begin{align}
P_{(1)}\bigg(\sum_{k=0}^{N+\frac{n-2}{2}} & V_{k} \cdot G^{\pm,\UU}_{1-\frac{n}2+k}\bigg) - \diag_*\vol_X \notag\\
&=
\big(P_{(1)}V_{N+\frac{n}{2}}\big)\cdot G^{\pm,\UU}_{N+1} 
\pm \tfrac{\rmi}{\pi}\sum_{k=\frac{n}{2}}^{N+\frac{n}{2}} \bigg\{\frac{\Boxz\Gamma-2n}{4k^2}\cdot V_k\cdot F_{k-\frac{n}{2}}^{\pm,\UU} - \frac{2}{k}\nabla_{\grade F_{k+1-\frac{n}{2}}^{\pm,\UU}}V_k\bigg\} .
\label{eq:Hadamard-odd-right}
\end{align}
Note that the $V_k$ and the $F_{k-\frac{n}{2}}^{\pm,\UU}$ occurring in the sum are smooth.
Thus, $\sum_{k=0}^{N+\frac{n-2}{2}} V_{k} \cdot G^{\pm,\UU}_{1-\frac{n}2+k}$ defines a right fundamental solution up to $C^N$-errors in the even-dimensional case.

\begin{remark} \label{theremark}
The transport equations \eqref{eq:transport} imply that the Hadamard coefficients $V_k(x,x)$ on the diagonal are given by universal algebraic expressions of the curvature of the manifold and the coefficients of the operator and a finite number (increasing with $k$) of their derivatives evaluated at $x$.
The same is true for the coefficients appearing in the short time asymptotic expansion of the heat kernel of a Laplace-type operator on a Riemannian manifold.
Remarkably, the transport equations for the heat coefficients are, up to constants, identical to those of the Hadamard coefficients and therefore both, the Hadamard coefficients and the heat coefficients, are given by essentially the same expressions.
Comparing the recursion relations of the Hadamard coefficients (e.g. \cite{BGP07}*{Prop.~2.3.1}) and the heat coefficients (e.g. \cite{BGV}*{p. 84}) one obtains
\begin{equation}
  V_k(x,x) = k!\, a_k(x,x)
\label{eq:Vkak}
\end{equation}
where $a_k$ is the $k$-th coefficient in the heat expansion
$$
 k_t(x,x)\sim (4\pi t)^{-\nicefrac{n}{2}}\sum_{k=0}^\infty a_k(x,x) t^{k} \quad\text{ as }t\searrow 0.
$$
This is true both for even and odd $n$ and is of importance in the proof of Corollary~\ref{cor:charforms}.
\end{remark}

\subsubsection{Local Feynman parametrix}
Using Borel summation, we use this to construct a parametrix near the diagonal. 
We fix a cutoff function $\chi \in C^\infty(\R)$ such that $\chi\equiv1$ on $[-\frac12,\frac12]$, $\chi\equiv0$ outside $[-1,1]$ and $0\le\chi\le1$ everywhere. 

\begin{lemma}
\label{lem:HadamardSeries}
There exist $\eps_k>0$ such that for every $N \in \N$ the series 
\begin{equation}
\begin{cases}
\sum\limits_{k=N+\frac{n-1}{2}}^{\infty} \chi(\nicefrac{\Gamma}{\eps_k}) \cdot V_{k} \cdot F^{\pm,\UU}_{1-\frac{n}2+k} & \mbox{ if $n$ is odd,} \\[10pt]
\sum\limits_{k=N+\frac{n}{2}}^{\infty} \chi(\nicefrac{\Gamma}{\eps_k}) \cdot V_{k} \cdot G^{\pm,\UU}_{1-\frac{n}2+k} & \mbox{ if $n$ is even,}
\end{cases}
\label{eq:FeynmanRest}
\end{equation}
converges in $C^N(\UU)$. 
The series 
\begin{equation}
\begin{cases}
\sum\limits_{k=0}^{\infty} \chi(\nicefrac{\Gamma}{\eps_k}) \cdot V_{k} \cdot F^{\pm,\UU}_{1-\frac{n}2+k} & \mbox{ if $n$ is odd,} \\[10pt]
\sum\limits_{k=0}^{\infty} \chi(\nicefrac{\Gamma}{\eps_k}) \cdot V_{k} \cdot G^{\pm,\UU}_{1-\frac{n}2+k} & \mbox{ if $n$ is even,}
\end{cases}
\label{eq:FeynmanReihe}
\end{equation}
defines a Feynman parametrix on $\UU$.
\end{lemma}

\begin{proof}
We only do the even-dimensional case, the case of odd $n$ being simpler.
For any compact subset $K\subset X$ and $k\in\N$ define the $C^k$-norm of a section of $u\in C^k(X,\SS)$ by 
$$
\|u\|_{C^k(K)} := \max_{x\in K} \max_{j=0,\ldots,k} |\nabla^j u| .
$$
Here $\nabla$ is an auxiliary connection and $|\cdot|$ an auxiliary fiber norm on $\SS$.

For $m>\ell+1$ and $0<\eps\le1$ one easily checks
\begin{gather*}
\bigg\|\frac{d^\ell}{dt^\ell}(\chi(t/\eps)t^m)\bigg\|
\le
c_1(\ell,m)\cdot\eps\cdot \|\chi\|_{C^\ell(\R)}\quad\mbox{ and}\\
\bigg\|\frac{d^\ell}{dt^\ell}(\chi(t/\eps)\LOG(t)t^m)\bigg\|
\le
c_1(\ell,m)\cdot\eps\cdot \|\chi\|_{C^\ell(\R)},
\end{gather*}
compare \cite{BGP07}*{Lemma~2.4.1}.
Here $\LOG$ is a branch of the logarithm with cut in the lower half plane and which coincides with the usual logarithm on positive real numbers.
This implies
\begin{align*}
\|\chi(\nicefrac{\Gamma}{\eps}) \Gamma^m\|_{C^N(K)}
&\le
c_2(m,N) \cdot\eps\cdot \|\chi\|_{C^k(\R)} \cdot \|\Gamma\|_{C^N(K)}
\end{align*}
and similarly for the expression with the $\LOG$-term.
Thus,
\begin{equation*}
\|\chi(\nicefrac{\Gamma}{\eps}) G_m^{\pm,\UU}\|_{C^N(K)}
\le
c_3(m,N,n) \cdot\eps\cdot \|\chi\|_{C^N(\R)} \cdot \|\Gamma\|_{C^N(K)} .
\end{equation*}
Hence, for $j>N+\frac{n}{2}$, we find 
\begin{align}
\Big\| \chi\big(\nicefrac{\Gamma}{\eps}\big) \cdot V_{j} \cdot G^{\pm,\UU}_{1-\frac{n}2+j} \Big\|_{C^N(K)}
&\le
c_4(j,N,n)\cdot \eps \cdot \|\chi\|_{C^k(\R)} \cdot  \|\Gamma\|_{C^N(K)} \cdot \|V_j\|_{C^N(K)}.
\label{eq:SchrankeVjGj}
\end{align}
Let $K_1 \subset K_2 \subset K_3 \subset \ldots$ be an exhaustion of $\UU$ by compact sets.
By \eqref{eq:SchrankeVjGj} we can choose $\eps_j>0$ such that
\begin{align*}
\Big\| \chi\big(\nicefrac{\Gamma}{\eps_j}\big) \cdot V_{j} \cdot G^{\pm,\UU}_{1-\frac{n}2+j} \Big\|_{C^N(K_j)}
&\le
2^{-j}
\end{align*}
provided $j>N+\frac{n}{2}$.
This shows that $\sum_{k=N+\frac{n+2}{2}}^{\infty} \chi\big(\nicefrac{\Gamma}{\eps_k}\big) \cdot V_{k} \cdot G^{\pm,\UU}_{1-\frac{n}2+k}$ converges absolutely in $C^N(K_j)$ for each $j$.
Since the summand with $k=N+\frac{n}{2}$ also has $C^N$-regularity, the first statement of the lemma follows.

Now we know that the series \eqref{eq:FeynmanReihe} is a well-defined distribution on $\UU$, and we need to show that it is a parametrix.
It follows from \eqref{eq:Hadamard-odd-right} that
\begin{align*}
P_{(1)}&\bigg(\sum_{k=0}^{\infty} \chi(\nicefrac{\Gamma}{\eps_k}) \cdot V_{k} \cdot G^{\pm,\UU}_{1-\frac{n}2+k}\bigg) - \diag_*\vol_X
= \\
&P_{(1)}\bigg(\sum_{k=0}^{N+\frac{n-2}{2}} (\chi(\nicefrac{\Gamma}{\eps_k})-1) \cdot V_{k} \cdot G^{\pm,\UU}_{1-\frac{n}2+k}\bigg) 
+ P_{(1)}\bigg(\sum_{k=N+\frac{n}{2}}^{\infty} \chi(\nicefrac{\Gamma}{\eps_k}) \cdot V_{k} \cdot G^{\pm,\UU}_{1-\frac{n}2+k}\bigg) \\
&+ \big(P_{(1)}V_{N+\frac{n}{2}}\big)\cdot G^{\pm,\UU}_{N+1} 
\pm \tfrac{\rmi}{\pi}\sum_{k=\frac{n}{2}}^{N+\frac{n}{2}} \bigg\{\frac{\Boxz\Gamma-2n}{4k^2}\cdot V_k\cdot F_{k-\frac{n}{2}}^{\pm,\UU} - \frac{2}{k}\nabla_{\grade F_{k+1-\frac{n}{2}}^{\pm,\UU}}V_k\bigg\} .
\end{align*}
The first term of the right-hand side is smooth since the coefficients $\chi(\nicefrac{\Gamma}{\eps_k})-1$ vanish on a neighborhood of $\{\Gamma=0\}$ which contains the singular support of $G^{\pm,\UU}_{1-\frac{n}2+k}$.
The second term is $C^{N-2}$ on $K_N$ and the third term is $C^N$ on $\UU$.
The last term is smooth on $\UU$.
Hence, $P_{(1)}\bigg(\sum_{k=0}^{\infty} \chi(\nicefrac{\Gamma}{\eps_k}) \cdot V_{k} \cdot G^{\pm,\UU}_{1-\frac{n}2+k}\bigg) - \diag_*\vol_X$ is $C^{N-2}$ on $K_N$ for every $N$.
Thus, it is smooth on $\UU$ and the series \eqref{eq:FeynmanReihe} defines a right parametrix.

By Proposition~\ref{FGOmega}~\eqref{FGWF}, $\WF'(\chi(\nicefrac{\Gamma}{\eps_k}) \cdot V_{k} \cdot G^{\pm,\UU}_{1-\frac{n}2+k})\subset \Lambda_\Feyn^\pm$ for each $k$.
Thus, for each $N$, $\WF'(\sum_{k=0}^{N+\frac{n-2}{2}}\chi(\nicefrac{\Gamma}{\eps_k}) \cdot V_{k} \cdot G^{\pm,\UU}_{1-\frac{n}2+k})\subset \Lambda_\Feyn^\pm$.
Since the remainder term \eqref{eq:FeynmanRest} is $C^N$, the Sobolev $H^N$-wavefront set of the series is contained in $\Lambda_\Feyn^\pm$ (see \cite{junker-schrohe}*{App.~B} for a nice exposition of Sobolev wavefront sets).
This holds for every $N$ and hence the usual smooth wavefront set of the series is contained in $\Lambda_\Feyn^\pm$.
Thus, the series in \eqref{eq:FeynmanReihe} defines a Feynman right parametrix.

Since Feynman parametrices are distinguished parametrices in the sense of Duistermaat and Hörmander, it follows from \cite{MR0388464}*{Thm.~6.5.3} that \eqref{eq:FeynmanReihe} is also a left parametrix.
\end{proof}

\subsubsection{Structure of Feynman parametrices}
Having constructed Feynman parametrices, we now study their singularity structure.

\begin{prop}\label{prop:ZerlegeG}
Let $P$ be a normally hyperbolic operator acting on sections of a vector bundle $\SS$ over a globally hyperbolic manifold $X$ of dimension $n$.
Let $N\in\N$.

Then every Feynman parametrix $G$ of $P$ is of the form 
$$
G = G^\loc + G^\reg
$$%
where $G^\reg$ is a distribution on $X\times X$ which is $C^N$ on a neighborhood of the diagonal and near the diagonal we have
\begin{equation*}
G^\loc =
\begin{cases}
\sum\limits_{k=0}^{N+\frac{n-3}{2}} V_{k} \cdot F^{+,\UU}_{1-\frac{n}2+k}& \text{ if $n$ is odd,}\\[10pt]
\sum\limits_{k=0}^{N+\frac{n-2}{2}} V_{k} \cdot G^{+,\UU}_{1-\frac{n}2+k}& \text{ if $n$ is even.}
\end{cases}
\end{equation*}
Moreover, $P_{(1)}G^\loc - \diag_*\vol_X$ is $C^N$ on a neighborhood of the diagonal and vanishes along the diagonal.
\end{prop}

\begin{proof}
Again, we only do the even-dimensional case since the case of odd $n$ is simpler.
Let $G$ be a Feynman parametrix of $P$.
Let $G_H$ be the Feynman parametrix on $\UU$ defined in \eqref{eq:FeynmanReihe}.
Put $\tilde G^\loc := \sum_{k=0}^{N+\frac{n-2}{2}}\chi(\nicefrac{\Gamma}{\eps_k}) \cdot  V_{k} \cdot G^{+,\UU}_{1-\frac{n}2+k}$ as in Lemma~\ref{lem:HadamardSeries}.

Let $x\in X$. 
Choose a globally hyperbolic open neighborhood $\Omega$ of $x$ such that $\Omega\times\Omega\subset\UU$.
On $\Omega$ we have the two Feynman parametrices $G|_{\Omega\times\Omega}$ and $G_H|_{\Omega\times\Omega}$.
By uniqueness of distinguished parametrices the difference $G|_{\Omega\times\Omega}-G_H|_{\Omega\times\Omega}$ is smooth.
Hence, $G-G_H$ is smooth on a neighborhood $\UU'\subset\UU$ of the diagonal.
After shrinking $\UU'$ if necessary, we can assume that $\tilde G^\loc = \sum_{k=0}^{N+\frac{n-2}{2}} V_{k} \cdot G^{+,\UU}_{1-\frac{n}2+k}$ on $\UU'$.

Let $\tilde\chi\in C^\infty(X\times X)$ be a cutoff function supported in $\UU'$ and such that $\tilde\chi\equiv1$ on a neighborhood $\UU''\subset\UU'$ of the diagonal.
We put $G^\loc := \tilde\chi\cdot\tilde G^\loc$ and $G^\reg := G - G^\loc = (1-\tilde\chi)\cdot G + \tilde\chi\cdot(G-G_H+G_H-\tilde G^\loc)$.
On $\UU''$ we find that $1-\tilde\chi$ vanishes, $G-G_H$ is smooth and $G_H-\tilde G^\loc$ is $C^N$.
Thus, $G^\reg$ is $C^N$ on $\UU''$.

From \eqref{eq:Hadamard-odd-right} we see that $P_{(1)}G^\loc - \diag_*\vol_X$ is $C^N$ on $\UU''$.
Moreover, $P_{(1)}G^\loc - \diag_*\vol_X$ vanishes along the diagonal because $G^{+,\UU}_{N+1}$, $\Boxz\Gamma-2n$, and $\grade F_{j}^{+,\UU}$ ($j\ge1$) do so.
\end{proof}

\subsubsection{The product case}
Let $X=I\times\Sigma$ carry a product metric and let $n=\dim(X)$ be even.
Assume $0\in I$ and let $y\in\Sigma$ be fixed.
Choose $\eps>0$ so that $(t,y,0,y)\in\UU$ whenever $|t|<\eps$.
Consider the embedding
$$
\alpha_y:(-\eps,\eps)\setminus\{0\}\to X\times X,
\quad \alpha_y(t) = (t,y,0,y).
$$
Over $\UU\setminus\diag(X)$ the wave front sets of $R^{\pm,\UU}_\beta$, $F^{\pm,\UU}_\beta$ and $G^{\pm,\UU}_\beta$ are contained in the set of pairs of lightlike covectors, see Proposition~\ref{FGOmega}~\eqref{FGWF}.
Thus the map $\alpha_y$ is transversal to these wave front sets and we can pull back the distributions along $\alpha_y$.
Put
\begin{equation}
\tilde C(\beta,n,\Lambda) := -\tfrac{\rmi}{\pi}\tfrac{dC(\beta,n)}{d\beta}+(\Lambda-1) C(\beta,n) .
\label{eq:defCtilde}
\end{equation}
If $\beta$ is a negative integer then $C(\beta,n)=0$ and Lemma~\ref{lem:dCdbeta} yields
\begin{equation}
\tilde C(\beta,n,\Lambda) 
= 
-\frac{\rmi(-1)^{\beta +1} 2^{-n-2\beta} \pi^{-\frac{n}{2}} (-\beta -1)! }{\left(\beta+\frac{n-2}{2}\right)!}.
\label{eq:Ctildeganz}
\end{equation}
For $\beta=0$ we find
$$
\tilde C(0,n,\Lambda)
=
-\rmi \bigg( 2\gamma - \log(4) - \sum_{\ell=1}^{\tfrac{n-2}2} \frac1\ell + \rmi\pi(\Lambda-1)\bigg)
\frac{2^{-n}\pi^{-\frac{n}{2}}}{\frac{n-2}{2}!} .
$$

\begin{lemma}\label{lem:polylog}
The following pullbacks are smooth functions on $(-\eps,\eps)\setminus\{0\}$ and are given by
\begin{enumerate}[(i)]
\item\label{polylog1}
$\alpha_y^*(R^{+,\UU}_\beta+R^{-,\UU}_\beta)(t) = 2C(\beta,n)\, \,|t|^{2\beta}$,
\item\label{polylog2}
$\alpha_y^*(F^{\pm,\UU}_\beta)(t) = C(\beta,n)\, \,|t|^{2\beta}$,
\item\label{polylog3}
$\alpha_y^*(G^{\pm,\UU}_\beta)(t) 
= 
  \big(\mp\tfrac{\rmi}{\pi}\tfrac{dC(\beta,n)}{d\beta}+\Lambda C(\beta,n) \mp\tfrac{2\rmi}{\pi}C(\beta,n)\log(|t|)\big)\,|t|^{2\beta}
$,
\item\label{polylog4}
$\alpha_y^*((G^{+,\UU}_\beta-\frac12(R^{+,\UU}_\beta+R^{-,\UU}_\beta))(t) 
= 
  (\tilde C(\beta,n,\Lambda) - \frac{2\rmi}{\pi}C(\beta,n)\log|t|)|t|^{2\beta}
$.
\item\label{polylog5}
$\tfrac{\partial}{\partial\beta}\alpha_y^*\left(G^{+,\UU}_\beta+\Lambda F^\UU_\beta - R^{+,\UU}_\beta - R^{-,\UU}_\beta\right)|_{\beta=-\frac{n}{2}} 
= 
  \tfrac{\partial}{\partial\beta}\tilde C(\beta,n,2\Lambda)|_{\beta=-\frac{n}{2}}|t|^{-n}
$.
\end{enumerate}
\end{lemma}

\begin{proof}
Putting $q:=\Gamma\circ\alpha_y$ we observe that $q(t) = t^2$.
Using Proposition~\ref{FGOmega}~\eqref{FGRgamma} we find
\begin{align*}
\alpha_y^*(R^{+,\UU}_\beta+R^{-,\UU}_\beta) 
&= 
2C(\beta,n)\alpha_y^*(\mu\Gamma^*(t_+^\beta)) \\
&=
2C(\beta,n)(\mu\circ\alpha_y)q^*(t_+^\beta) \\
&=
2C(\beta,n)(\mu\circ\alpha_y)(t^{2\beta}_++t^{2\beta}_-)
\end{align*}
and 
\begin{align*}
\alpha_y^*(F^{\pm,\UU}_\beta)
=
C(\beta,n)(\mu\circ\alpha_y)q^*(f^\pm_\beta)
=
C(\beta,n)(\mu\circ\alpha_y)(t^{2\beta}_++t^{2\beta}_-) .
\end{align*}
Note that in the product case $\mu(\alpha_y(t))=\mu(\alpha_y(0))=1$.
Thus, we can drop the factor $\mu\circ\alpha_y$ everywhere.
Furthermore,
\begin{align*}
\alpha_y^*(G^{\pm,\UU}_\beta)
&=
q^*\big((\mp\tfrac{\rmi}{\pi}\tfrac{dC(\beta,n)}{d\beta}+\Lambda C(\beta,n))f^\pm_\beta \mp\tfrac{\rmi}{\pi}C(\beta,n)h^\pm_\beta\big) \\
&=
\big((\mp\tfrac{\rmi}{\pi}\tfrac{dC(\beta,n)}{d\beta}+\Lambda C(\beta,n))(t_+^{2\beta}+t_-^{2\beta}) \mp\tfrac{\rmi}{\pi}C(\beta,n)(t_+^{2\beta}+t_-^{2\beta})\log(t^2)\big) .
\end{align*}
We have shown \eqref{polylog1}--\eqref{polylog3}.
Assertion \eqref{polylog4} follows from \eqref{polylog1} and \eqref{polylog3}.
As to \eqref{polylog5}, we observe
$$
\alpha_y^*\left(G^{+,\UU}_\beta+\Lambda F^\UU_\beta - R^{+,\UU}_\beta - R^{-,\UU}_\beta\right)
=
  (\tilde C(\beta,n,2\Lambda) - \tfrac{2\rmi}{\pi}C(\beta,n)\log|t|)|t|^{2\beta} .
$$
Since $\beta\mapsto C(\beta,n)$ has a zero of order $2$ at $\beta=-\frac{n}{2}$, $\tilde C(\beta,n,2\Lambda)$ also vanishes at $\beta=-\frac{n}{2}$.
Thus
\begin{align*}
\tfrac{\partial}{\partial\beta}\alpha_y^*\left(G^{+,\UU}_\beta+\Lambda F^\UU_\beta - R^{+,\UU}_\beta - R^{-,\UU}_\beta\right)|_{\beta=-\frac{n}{2}}
&=
  \tfrac{\partial}{\partial\beta}(\tilde C(\beta,n,2\Lambda) - \tfrac{2\rmi}{\pi}C(\beta,n)\log|t|)|_{\beta=-\frac{n}{2}}|t|^{-n} \\
&=
  \tfrac{\partial}{\partial\beta}\tilde C(\beta,n,2\Lambda)|_{\beta=-\frac{n}{2}}|t|^{-n} .
\qedhere
\end{align*}
\end{proof}

We observe that the Hadamard coefficients of an operator of product type are independent of the time variables.

\begin{lemma}\label{lem:HadamardProdukt}
 Suppose that $P$ is a normally hyperbolic operator which has product structure over $X=I\times \Sigma$ where $I\subset \R$ is an open interval.
 Then the Hadamard coefficients $V_k(t_1,y_1,t_2,y_2)$ are independent of $t_1, t_2\in I$.
\end{lemma}
\begin{proof}
 The $V_k$ solve the transport equations \eqref{eq:transport} with initial condition $V_0(x,x) = \id_{\SS_x}$.
 In the product case we have 
 $$
 \Box = \frac{\partial^2}{\partial t^2} + \Delta
 \quad\mbox{ and }\quad
 P = \frac{\partial^2}{\partial t^2} + \Delta_P
 $$
 where $\Delta$ is the Laplace-Beltrami operator on $\Sigma$ and $\Delta_P$ is a Laplace-type operator acting on sections of $\SS$.
 Moreover,
 $$
 \Gamma(t_1,y_1,t_2,y_2) = (t_1-t_2)^2 - \tilde{\Gamma}(y_1,y_2)
 $$ 
 where $\tilde \Gamma$ is the square of the Riemannian distance function on $\Sigma$. 
 Then 
 $$
 \grade \Gamma = 2(t_1-t_2) \partial_t - \grade \tilde \Gamma.
 $$
 Now let $W_k$ be the solution of the Riemannian transport equation on $\Sigma$,
 $$ 
   \nabla_{\grad \tilde \Gamma} W_k + \left( -\tfrac{1}{2}\Delta \tilde \Gamma - (n-1) +2 k \right) W_k = - 2 k \Delta_P W_{k-1},
 $$
 with initial condition $W_0(y,y) = \id_{\SS_x}$. 
 One easily checks that $\tilde V_k(t_1,y_1,t_2,y_2):=W_k(y_1,y_2)$ solve the transport equations \eqref{eq:transport}.
 By uniqueness of the Hadamard coefficients we have  $V_k(t_1,y_1,t_2,y_2)=\tilde V_k(t_1,y_1,t_2,y_2)=W_k(y_1,y_2)$.
\end{proof}

We can also replace the manifold $X$ by the product $\tilde X = X \times \R$ with metric $g + ds^2$, where $s$ is the additional (spacelike) variable. 
The bundle $\SS$ and the associated connection are pulled back to $X \times \R$ and the operator $\tilde P = P - \partial_s^2$ is then a well-defined normally hyperbolic operator on $\tilde X$. 

\begin{lemma}\label{lem:HadamardProduktSpace}
 The Hadamard coefficients $\tilde V_k(x_1,s_1,x_2,s_2)$ for the operator $\tilde P$ are independent of $s_1,s_2\in \R$, and we have $\tilde V_k(x_1,s_1,x_2,s_2) = V_k(x_1,x_2)$.
\end{lemma}
\begin{proof}
 The proof is exactly the same as the proof of Lemma~\ref{lem:HadamardProdukt} with the only difference that $\tilde \Gamma(x_1,s_1,x_2,s_2) = \Gamma(x_1,x_2) - (s_1-s_2)^2$. 
 One again checks that the $V_k$ satisfy transport equations and initial condition and uses uniqueness of the Hadamard coefficients.
\end{proof}

Let $\hat G_L$ be the Feynman propagator of $\dirac_R\dirac_L$ given in Theorem~\ref{thm:FeynmanProductWave}.
We split $\hat G_L$ into a local and a regular part, $\hat G_L=\hat G_L^\loc + \hat G_L^\reg$, as in \eqref{eq:DecompositionOfG}.
Let $G_{L,\ret}$ and $G_{L,\adv}$ be the retarded and the advanced fundamental solution of $\dirac_R\dirac_L$, respectively.

\begin{lemma}\label{lem:logpoly}
There exists $\eps>0$ such that 
$$
\tr\Big(\dirac_{L,(1)} \big(G^\loc_{L} -\tfrac{1}{2}( G_{L,\ret} + G_{L,\adv})\big)(t,y,0,y)\circ\nS\Big) 
=
-\rmi\left\{\sum_{\ell=0}^{n-1} \tilde a_{y,\ell} t^{-\ell} + \tilde c_y \log|t| + \tilde r_y(t)\right\}
$$
for all $|t|<\eps$, $t\neq0$, where
\begin{align*}
\tilde c_y 
&=  
\tfrac{2^{1-n}\pi^{-\frac{n}{2}}}{(\frac{n-2}{2})!}\cdot\tr(\dirac_{L,(1)} V_{\frac{n-2}{2}}(0,y,0,y)\circ\nS) , 
\\
\tilde a_{y,0} 
&= 
\bigg(\gamma - \log(2) - \tfrac12\sum_{\ell=1}^{\tfrac{n-2}2} \tfrac1\ell + \frac{\rmi\pi(\Lambda-1)}{2} \bigg)\tilde c_y , 
\\
\tilde a_{y,\ell} 
&=  
\tfrac{(-1)^{j-\frac{n}{2}}\pi^{-\frac{n}{2}}(\frac{n}{2}-2-j)! }{2^{2j+2}j!} \tr(\dirac_{L,(1)} V_j(0,y,0,y)\circ\nS)   
\quad\text{ if }\ell = n-2j-2, \quad j=0,\ldots,\tfrac{n-4}{2} ,
\\
\tilde a_{y,\ell} 
&=   
\tfrac{(-1)^{1+j-\frac{n}{2}} \pi^{-\frac{n}{2}} (\frac{n}{2}-1-j)! }{2^{2j+1}\cdot j!}\tr(V_j(0,y,0,y))
\quad\text{ if }\ell = n-2j-1, \quad j=1,\ldots,\tfrac{n-2}{2} ,
\\
\tilde a_{y,n-1} 
&=  
\tfrac{\pi^{-\frac{n}{2}}}{4}  \tr(V_0(0,y,0,y)) \cdot \tfrac{\partial}{\partial\beta} \tilde C(\beta,n,2\Lambda)|_{\beta=-\frac{n}{2}}
\quad\text{ and}
\\
\tilde r_y &\text{ is continuous at }t=0\text{ and }\tilde r_y(t)=\O(t\log|t|)\text{ as }t\to0\text{ uniformly in }y.
\end{align*}
\end{lemma}

\begin{proof}
The Hadamard expansion of $G^\loc_{L} -\frac{1}{2}( G_{L,\ret} +  G_{L,\adv})$ takes the form
\begin{equation}
 G^\loc_{L} -\tfrac{1}{2}( G_{L,\ret} +  G_{L,\adv}) 
 = 
 \sum_{j=0}^{\frac{n+2}{2}}  V_j \left(  G^{+,\UU}_{-\frac{n}{2}+1+j} -  \tfrac{1}{2} ( R^{+,\UU}_{-\frac{n}{2}+1+j} +   R^{-,\UU}_{-\frac{n}{2}+1+j})\right) + r, 
\label{eq:Hadhilf0}
\end{equation}
where $r$ is $C^2$ and vanishes on the diagonal to second order.
We compute for $j\ge1$, using Proposition~\ref{FGOmega}~\eqref{gradFG}:
\begin{align*}
\dirac_{L,(1)}\left( V_j  G^{+,\UU}_{-\frac{n}{2}+1+j}\right)
&=
\left(\dirac_{L,(1)} V_j \right) G^{+,\UU}_{-\frac{n}{2}+1+j} + \grades G^{+,\UU}_{-\frac{n}{2}+1+j} \cdot V_j\notag\\
&=
\left(\dirac_{L,(1)} V_j \right) G^{+,\UU}_{-\frac{n}{2}+1+j} + \grades\Gamma\cdot V_j \cdot\left(\tfrac{1}{4j} G^{+,\UU}_{-\frac{n}{2}+j}-\tfrac{\rmi}{4j^2\pi} F^{+,\UU}_{-\frac{n}{2}+j}\right)
\end{align*}
and similarly, using Proposition~1.4.2~(4) in \cite{BGP07},
\begin{align*}
\dirac_{L,(1)}\left( V_j  R^{\pm,\UU}_{-\frac{n}{2}+1+j}\right)
&=
\left(\dirac_{L,(1)} V_j \right) R^{\pm,\UU}_{-\frac{n}{2}+1+j} +\grades\Gamma\cdot V_j \cdot \tfrac{1}{4j} R^{\pm,\UU}_{-\frac{n}{2}+j}.
\end{align*}
For $j=0$ we differentiate Proposition~\ref{FGOmega}~\eqref{gradFG} at $\beta=-\frac{n}{2}$ and combine with Proposition~\ref{FGOmega}~\eqref{FG} to find
\begin{align*}
(\grade\Gamma)\cdot \tfrac{\partial}{\partial\beta}|_{\beta=-\frac{n}{2}} G^{+,\UU}_{\beta} 
&= 
4\,\grade (G^{+,\UU}_{-\frac{n}{2}+1}) - \tfrac{4\rmi}{\pi} \grade (\tfrac{\partial}{\partial\beta}|_{\beta=-\frac{n}{2}}F^{+,\UU}_{\beta+1}) \\
&= 
4\,\grade (G^{+,\UU}_{-\frac{n}{2}+1}) + 4\, \grade (G^{+,\UU}_{-\frac{n}{2}+1}-\Lambda F^{+,\UU}_{-\frac{n}{2}+1}) \\
&= 
8\,\grade (G^{+,\UU}_{-\frac{n}{2}+1}) -\Lambda (\grade\Gamma)\cdot \tfrac{\partial}{\partial\beta}|_{\beta=-\frac{n}{2}} F^{+,\UU}_{\beta} ,
\end{align*}
hence
\begin{equation*}
\grade (G^{+,\UU}_{-\frac{n}{2}+1})
=
\tfrac18 (\grade\Gamma)\cdot \tfrac{\partial}{\partial\beta}|_{\beta=-\frac{n}{2}} (G^{+,\UU}_{\beta} + \Lambda F^{+,\UU}_{\beta})
\end{equation*}
and therefore
\begin{equation*}
\dirac_{L,(1)}\left( V_0  G^{+,\UU}_{-\frac{n}{2}+1}\right)
=
\left(\dirac_{L,(1)} V_0 \right) G^{+,\UU}_{-\frac{n}{2}+1} + \tfrac18 (\grades\Gamma)\cdot V_0 \cdot\tfrac{\partial}{\partial\beta}|_{\beta=-\frac{n}{2}} (G^{+,\UU}_{\beta} + \Lambda F^{+,\UU}_{\beta}) .
\end{equation*}
Finally, 
\begin{align}
\dirac_{L,(1)}\left( V_0  R^{\pm,\UU}_{-\frac{n}{2}+1}\right)
&=
\left(\dirac_{L,(1)} V_0 \right) R^{\pm,\UU}_{-\frac{n}{2}+1} + \grades R^{\pm,\UU}_{-\frac{n}{2}+1} \cdot V_0\notag\\
&=
\left(\dirac_{L,(1)} V_0 \right) R^{\pm,\UU}_{-\frac{n}{2}+1} + \tfrac14 (\grades\Gamma) \cdot V_0\cdot \tfrac{\partial}{\partial\beta}|_{\beta=-\frac{n}{2}} R^{\pm,\UU}_\beta.
\label{eq:Hadhilf4}
\end{align}
Combining \eqref{eq:Hadhilf0}--\eqref{eq:Hadhilf4} yields for $\dirac_{L,(1)} \left(G^\loc_{L} -\frac{1}{2}( G_{L,\ret} +  G_{L,\adv})\right)$ the singularity structure
\begin{align*}
\dirac_{L,(1)} \big(G^\loc_{L} -\tfrac{1}{2}( G_{L,\ret} +&  G_{L,\adv})\big)
=
\sum_{j=0}^{\frac{n+2}{2}} (\dirac_{L,(1)} V_j) \left(  G^{+,\UU}_{-\frac{n}{2}+1+j} -  \tfrac{1}{2} ( R^{+,\UU}_{-\frac{n}{2}+1+j} +   R^{-,\UU}_{-\frac{n}{2}+1+j})\right) \\
&+
\grades\Gamma\cdot\sum_{j=1}^{\frac{n+2}{2}} \tfrac{1}{4j}  V_j \left( G^{+,\UU}_{-\frac{n}{2}+j} + \tfrac{\rmi}{j\pi} F^{+,\UU}_{-\frac{n}{2}+j} -  \tfrac{1}{2} ( R^{+,\UU}_{-\frac{n}{2}+j} +   R^{-,\UU}_{-\frac{n}{2}+j})\right)  \\
&+
\tfrac18\, \grades\Gamma\cdot V_0 \cdot \tfrac{\partial}{\partial\beta}|_{\beta=-\frac{n}{2}} \left(G^{+,\UU}_\beta+\Lambda F^{+,\UU}_\beta - R^{+,\UU}_\beta - R^{-,\UU}_\beta\right)
+ \dirac_{L,(1)} r .
\end{align*}
Now note that in the product case we have $\grade \, \Gamma (t,y,0,y) = -2 tn_\Sigma$. 
Moreover, by Lemma~\ref{lem:HadamardProdukt}, the Hadamard coefficients are independent of the time variable, i.e.\ $V_k(t_1,y_2,t_2,y_2) = V_k(0,y_1,0,y_2)$.
In the following, the constant $C(\beta,n)$ is defined in \eqref{eq:defC} and $\tilde C(\beta,n,\Lambda)$ in \eqref{eq:defCtilde}.
Using $\nS^2=-1$, Lemma~\ref{lem:polylog}, \eqref{eq:Ctildeganz} and $C(\beta,n)=0$ if $\beta$ is a negative integer we find for $|t|<\eps$, $t\neq0$:
\begin{align*}
\phi^\loc_y (t)
&=
\tr\Big(\dirac_{L,(1)} \big(G^\loc_{L} -\tfrac{1}{2}( G_{L,\ret} + G_{L,\adv})\big)(t,y,0,y)\circ\nS\Big) \\
&=
\sum_{j=0}^{\frac{n+2}{2}} \tr(\dirac_{L,(1)} V_j(t,y,0,y)\circ\nS)\, \alpha_y^*\left(  G^{+,\UU}_{-\frac{n}{2}+1+j} -  \tfrac{1}{2} ( R^{+,\UU}_{-\frac{n}{2}+1+j} +   R^{-,\UU}_{-\frac{n}{2}+1+j})\right)(t) \\
&\quad-
t\sum_{j=1}^{\frac{n+2}{2}} \tfrac{1}{2j} \tr(\nS\circ V_j(t,y,0,y)\circ\nS)\, \alpha_y^* \left( G^{+,\UU}_{-\frac{n}{2}+j} + \tfrac{\rmi}{j\pi} F^{+,\UU}_{-\frac{n}{2}+j} -  \tfrac{1}{2} ( R^{+,\UU}_{-\frac{n}{2}+j} +   R^{-,\UU}_{-\frac{n}{2}+j})\right)(t)  \\
&\quad-
\tfrac14\,t \tr(\nS\circ V_0(t,y,0,y)\circ\nS) \cdot \tfrac{\partial}{\partial\beta}|_{\beta=-\frac{n}{2}} \alpha_y^*\left(G^{+,\UU}_\beta+\Lambda F^{+,\UU}_\beta - R^{+,\UU}_\beta - R^{-,\UU}_\beta\right)(t)
+ \O(t) \\
&=
 \sum_{j=0}^{\frac{n+2}{2}} \tr(\dirac_{L,(1)} V_j(0,y,0,y)\circ\nS) \left(\tilde C(\beta,n,\Lambda)-\tfrac{2\rmi}{\pi}C(\beta,n)\log|t|\right)|_{\beta=1+j-\frac{n}{2}}t^{2+2j-n}\\
&\quad+
 \sum_{j=1}^{\frac{n+2}{2}} \tfrac{1}{2j} \tr(V_j(0,y,0,y))\, \left(\tilde C(\beta,n,\Lambda) + \tfrac{\rmi}{j\pi}C(\beta,n)-\tfrac{2\rmi}{\pi}C(\beta,n)\log|t|\right)|_{\beta=j-\frac{n}{2}}t^{1+2j-n} \\
&\quad+
\tfrac14  \tr(V_0(0,y,0,y)) \cdot \tfrac{\partial}{\partial\beta}\tilde C(\beta,n,2\Lambda)|_{\beta=-\frac{n}{2}} t^{1-n}
+ \O(t) \\
&=
 \tr(\dirac_{L,(1)} V_{\frac{n-2}{2}}(0,y,0,y)\circ\nS) \left(\tilde C(0,n,\Lambda)-\tfrac{2\rmi}{\pi}C(0,n)\log|t|\right) \\
&\quad+
 \sum_{j=0}^{\frac{n-4}{2}} \tr(\dirac_{L,(1)} V_j(0,y,0,y)\circ\nS) \tilde C(1+j-\tfrac{n}{2},n,\Lambda)t^{2+2j-n}\\
&\quad+
 \sum_{j=1}^{\frac{n-2}{2}} \tfrac{1}{2j} \tr(V_j(0,y,0,y))\, \tilde C(j-\tfrac{n}{2},n,\Lambda)t^{1+2j-n} \\
&\quad+
\tfrac14  \tr(V_0(0,y,0,y)) \cdot \tfrac{\partial}{\partial\beta} \tilde C(\beta,n,2\Lambda)|_{\beta=-\frac{n}{2}} t^{1-n}
+ \O(t\log|t|) \\
&=
 \tr(\dirac_{L,(1)} V_{\frac{n-2}{2}}(0,y,0,y)\circ\nS) \tfrac{-\rmi 2^{1-n}\pi^{-\frac{n}{2}}}{(\frac{n-2}{2})!} \bigg(\gamma - \log(2) - \tfrac12\sum_{\ell=1}^{\tfrac{n-2}2} \tfrac1\ell + \tfrac{\rmi\pi(\Lambda-1)}{2} + \log|t|\bigg) \\
&\quad+
 \sum_{j=0}^{\frac{n-4}{2}} \tr(\dirac_{L,(1)} V_j(0,y,0,y)\circ\nS) \tfrac{-\rmi(-1)^{j-\frac{n}{2}}  \pi^{-\frac{n}{2}} (\frac{n}{2}-2-j)! }{4^{1+j}j!}t^{2+2j-n}\\
&\quad+
 \sum_{j=1}^{\frac{n-2}{2}}  \tr(V_j(0,y,0,y))\, \tfrac{-\rmi(-1)^{1+j-\frac{n}{2}}  \pi^{-\frac{n}{2}} (\frac{n}{2}-1-j)! }{2\cdot 4^{j}\cdot j!}t^{1+2j-n} \\
&\quad+
\tfrac14  \tr(V_0(0,y,0,y)) \cdot \tfrac{\partial}{\partial\beta} \tilde C(\beta,n,2\Lambda)|_{\beta=-\frac{n}{2}} t^{1-n}
+ \O(t\log|t|) 
\qquad\text{ as }t\to 0.
\qedhere
\end{align*}
\end{proof}


\end{document}